\documentclass[a4paper,11pt]{amsart}
\usepackage{amssymb}
\usepackage{xcolor}
\usepackage{hyperref}
\usepackage{enumitem}

\usepackage[margin=2.9cm]{geometry}

\newtheorem{Main}{Theorem}

\usepackage{color}
\usepackage{caption}
\usepackage{subcaption}
\usepackage{enumerate}
\usepackage{xcolor}
\usepackage[T1]{fontenc}
\usepackage[draft]{todonotes}

\definecolor{orange}{rgb}{1,0.5,0}

\usepackage[backend=biber, maxnames=4, style=alphabetic, maxalphanames=3, minalphanames=1,]{biblatex}
\DeclareLabelalphaTemplate{
  \labelelement{%
    \field[strwidth=3,strside=left,ifnames=1]{labelname}%
    \field[strwidth=1,strside=left]{labelname}%
  }%
}
\DeclareFieldFormat{extraalpha}{%
  \ifnumless{\thefield{year}}{2000}%
    {\the\numexpr\thefield{year}-1900}%
    {\the\numexpr\thefield{year}-2000}}

\usepackage{sidecap}

\usepackage{comment}

\usepackage{mathrsfs}

\DeclareMathAlphabet{\mathpzc}{OT1}{pzc}{L}{it}

\theoremstyle{definition}
\newtheorem{definition}{Definition}[section]
\newtheorem{theorem}[definition]{Theorem}

\newtheorem{proposition}[definition]{Proposition}

\newtheorem{corollary}[definition]{Corollary}
\newtheorem{lemma}[definition]{Lemma}

\newtheorem{remark}[definition]{Remark}

\numberwithin{equation}{section}

\def\geq{\geqslant}
\def\leq{\leqslant}
\def\R{\mathbb{R}}
\def\T{\mathbb{T}}
\def\eps{\varepsilon}
\def\Z{\mathbb{Z}}
\def\N{\mathbb{N}}

\def\Q{\mathbb{Q}}
\newcommand{\be}{\begin{equation}}
\newcommand{\ee}{\end{equation}}
\newcommand{\RQ}{\R \setminus \Q}
\newcommand{\abs}[1]{\left| #1 \right|}
\newcommand{\norm}[1]{\abs{\abs{#1}}}

\title[Lower bounds on mixing rates for a class of mixing flows on surfaces]{Lower bounds on mixing rates for a class of mixing flows on surfaces}
\author[M. Lorenzo--Laguno]{Miriam Lorenzo--Laguno}
\date{}
\begin{document}
\begin{abstract}
We study mixing rates for locally Hamiltonian flows on compact surfaces with asymmetric logarithmic singularities. For a full measure set of such flows, we show that the decay of correlations of smooth observables cannot be uniformly faster than a power of $\log t$. In particular, there exist sequences of times and observables for which correlations admit lower bounds of order $(\log t)^{-2-\nu}$  for any $\nu>0$.

We further show that for a typical Arnol'd flow on $\mathbb T^2$, the self-correlation of every box in the minimal component is bounded below by $(\log t)^{-1}$ along an unbounded sequence of times. Motivated by questions in spectral theory, we also construct examples of such flows for which the self-correlation of a box fails to be square-integrable. 

These results complement previous upper bounds for correlations in both settings, which are also of polynomial order in $\log t$, and show that logarithmic decay rates are essentially sharp along sequences of times.
\end{abstract}

\maketitle

\section{Introduction}
\subsection{Smooth area-preserving flows on surfaces}
Smooth area-preserving flows on surfaces are one of the most fundamental and extensively studied classes of dynamical systems. They provide low-dimensional examples of systems with non-trivial ergodic, mixing, and spectral properties, and their study has driven the development of many central ideas in smooth ergodic theory and Teichmuller dynamics.

Let $\mathcal{M}$ be a compact, connected, orientable smooth surface equipped with a smooth non-degenerate area form $\omega$. Consider an $\omega$-preserving flow $(\phi_t^{\mathcal{M}})_{t \in \mathbb{R}}$ on $\mathcal{M}$ determined by a smooth vector field $W: \mathcal{M} \to T\mathcal{M}$ with finitely many zeroes. Such a flow is \emph{locally Hamiltonian} if the pair $(\phi_t^{\mathcal{M}}, \omega)$ determines a smooth closed $1$-form $\eta$ with $\eta = dH$ locally, for some (smooth) $H:\mathcal{M}\to \R$, and the flow satisfies Hamilton's equations in local coordinates $(x,y)$ with $\omega = dx \wedge dy$:
$
\dot{x} = \frac{\partial H}{\partial y}, 
\dot{y} = -\frac{\partial H}{\partial x}.
$

The fixed points of $(\phi_t^{\mathcal{M}})$ are the critical points of $H$. The structure of such flows depends strongly on the number and nature of the fixed points. If $(\phi_t^{\mathcal{M}})$ has no fixed points, the Poincaré--Hopf index theorem forces $\mathcal{M} = \mathbb{T}^2$, and the flow is a reparametrization of a linear flow. By a classical result of Katok \cite{Katok1967}, no such flow is mixing. The situation is different when the set of fixed points is non-empty and finite. 

A classical decomposition theorem (proved independently by Mayer \cite{Mayer1943}, Levitt \cite{Levitt1982}, and Zorich \cite{Zorich1999}) asserts that $\mathcal{M}$ splits into finitely many $(\phi_t^{\mathcal{M}})$-invariant closed domains $M_i$, whose boundaries are formed by polycycles (unions of saddles and separatrix connections), and on each domain the dynamics is either periodic or \emph{quasi-minimal} (see \cite{KanigowskiOkunevZelada2026} for details). The ergodic and mixing properties of a flow restricted to a quasi-minimal component have been studied intensively over the past five decades.

\subsection{Mixing and mixing rates} Given a flow $(\phi_t^{\mathcal{M}})$ preserving an area form $\omega$ on an invariant component $\mathcal{M'}$ (with $\omega(\mathcal{M'})=1$), the correlation of two observables $f,g \in L^2(\mathcal{M'})$ at time $t$ is defined as
\begin{equation}\label{eq:cor_def}
\mathrm{Cor}_t(f, g) := \int_{\mathcal{M'}} (f \circ \phi_t^{\mathcal{M}}) g \, \omega - \int_{\mathcal{M'}} f \, \omega \int_{\mathcal{M'}} g \, \omega.
\end{equation}
A flow is mixing if $\lim_{t \to \infty} \mathrm{Cor}_t(f, g) = 0$ for all $f,g \in L^2(\mathcal{M'})$.

A fixed point $p$ is called a \emph{saddle} if $H$ does not achieve a local maximum or minimum at $p$; it is \emph{non-degenerate} if the Hessian of $H$ at $p$ is non-singular, and \emph{degenerate} otherwise. Kochergin \cite{kochergin75} constructed mixing flows with degenerate saddles on every surface of genus $\geq 1$, providing the first results on mixing for locally Hamiltonian flows on surfaces. The situation for non-degenerate saddles is more subtle. First, Sinai--Khanin \cite{Sinai1992} showed that a flow on $\mathbb{T}^2$ with one homoclinic loop is mixing for Lebesgue--almost every irrational frequency; in fact,
their proof yields an explicit rate, see \eqref{eq:sk_rate} below, and later Kochergin \cite{kochergin03,Kochergin2004} improved the result to all irrational frequencies. For higher genus surfaces, Ulcigrai \cite{ulcigrai11} showed that almost every flow with symmetric singularities (see Section \ref{sec:susflows}) is not mixing. The same author \cite{ulcigrai07} proved mixing for almost every flow with asymmetric singularities.

One of the main tools underlying most of these results is \emph{renormalization}: a rich dynamics on the space of interval exchange transformations (IETs, see Sections \ref{sec:intro_IETs} and \ref{sec:IETs}) that appear as Poincaré return maps for surface flows. Since renormalization naturally produces results that hold for almost every IET (with respect to Lebesgue measure), most of the above theorems are stated for \emph{almost every} flow. Indeed, Chaika--Wright \cite{chaikawright} established the existence of a mixing flow with symmetric singularities, and a non-mixing flow with asymmetric singularities was recently constructed by Fayad--Kanigowski--Zelada \cite{fkz}. Such almost-everywhere restrictions are intrinsic to the renormalization approach and will also appear in the results proved in this paper.

Quantitative results on the \emph{rate} of mixing have been obtained in both settings of degenerate and non-degenerate saddles. For flows with degenerate saddles, Fayad \cite{Fayad2001} established polynomial decay in $t$ of correlations of indicator functions of boxes for almost every flow in $\T^2$. Fayad--Forni--Kanigowski \cite{FayadForniKanigowski2021} later proved that such flows typically have Lebesgue spectrum with countable multiplicity. Their proof leverages an irregular, non-uniform decay of correlations of smooth observables to establish square-summability of correlations in time, which is then the key input for the spectral conclusion.

In the non-degenerate case, the shearing is much milder, significantly slowing the mixing process. 
In the setting of mixing flows on $\T^2$, an explicit rate follows from the original argument of
Sinai--Khanin \cite{Sinai1992}. For almost every flow, they show that for any rectangles
$A, B$ and every $1/4>\nu>0$, there exists a constant 
$C_{A,B,\nu}>0$ such that, for $t$ large enough,
\be \label{eq:sk_rate}
|\mathrm{Cor}_t(\chi_A, \chi_B)|=|\mu(\phi_t^{\mathcal{M}} A \cap B) - \mu(A)\mu(B)| \leq C_{A,B,\nu}\,(\log t)^{-\nu}.
\ee

Moving beyond the torus, Ravotti \cite{ravotti} studied typical locally Hamiltonian flows with saddle loops producing asymmetric logarithmic singularities on compact surfaces of genus $\geq 1$. For $C^1$ observables compactly supported away from the fixed points on a minimal component $\mathcal{M}'$, he showed correlations decay at least polynomially in $\log t$:
\be\label{eq:round_rav}
|\mathrm{Cor}_t(f, g)|=\left|\int_{\mathcal{M}'} (f\circ\phi_t^{\mathcal{M}} ) g\,\omega- \int_{\mathcal{M}'} f \,\omega\int_{\mathcal{M}'}g\,\omega\right| \leq C_{f,g} (\log t)^{-\sigma},
\ee
for some $1>\sigma>0,$ $C_{f,g}>0$.

Sinai--Khanin and Ravotti follow similar strategies. First, they reduce the problem to the study of suspension flows over interval exchange transformations with asymmetric logarithmic roof singularities (see Sections \ref{sec:intro_IETs}--\ref{sec:susflows} for definitions). Then, the proof relies on a quantitative analysis of the shearing mechanism generated by the asymmetric logarithmic singularities. A sharp control of Birkhoff sums of the derivative of the roof function, combined with some diophantine results for irrational rotations and IETs, respectively, yields decay of correlations for the suspension flow, and hence for the original locally Hamiltonian flow. In this paper, we use similar techniques to prove counterparts to both results: lower bounds on the correlations.

\subsection{Main results}
Let us state the main results of this paper, which establish a lower bound on the decay of correlations. The results are stated for \emph{typical} locally Hamiltonian flows, meaning full measure with respect to Katok's fundamental class (see Section \ref{sec:typical}), and for flows exhibiting asymmetric logarithmic singularities (defined precisely in Section \ref{sec:susflows}).

\begin{Main}
\label{thm:flows}
Let $(\phi_t^{\mathcal{M}})_{t \in \mathbb{R}}$ be a locally Hamiltonian flow and let $\mathcal{M}' \subset \mathcal{M}$ be a minimal component. Then for almost every $(\phi_t^{\mathcal{M}})$ with asymmetric logarithmic singularities on $\mathcal{M}'$ there exist an increasing sequence of times $(t_m)$ and functions  $(f_m),( g_m)\subset C^1(\mathcal{M}')$ such that the following holds. For any $\nu > 0$ and $C > 0$, there exists $M(C,\nu)>0$ such that for any $m\geq M(C,\nu)$,
\[
|\mathrm{Cor}_{t_m}(f_m,g_m)| \geq C\,\|f_m\|_{C^1(\mathcal{M}')}\|g_m\|_{C^1(\mathcal{M}')}\,(\log t_m)^{-2-\nu}.
\]
\end{Main}

Theorem \ref{thm:flows} should be compared with \eqref{eq:round_rav}. The constant $C_{f,g}$ there depends on the specific pair of functions, and since $C^1_c(\mathcal{M}')$ is not a Banach space, there is no natural complete norm on this space, so one cannot invoke a uniform boundedness argument. That is, Ravotti's result does not immediately imply the existence of a uniform $C > 0$ such that
\[
|\mathrm{Cor}_t(f,g)| \leq C\,\|f\|_{C^1(\mathcal{M}')}\|g\|_{C^1(\mathcal{M}')}\,
(\log t)^{-\sigma}
\]
holds for all $f, g \in C^1_c(\mathcal{M}')$. Theorem \ref{thm:flows} addresses exactly this question: it shows that if any such uniform bound holds, the exponent $\sigma$ cannot exceed $2$.

We sharpen the picture with two results in the setting of flows on $\mathbb{T}^2$ with a single non-degenerate saddle point with a homoclinic loop and minimal component $\mathcal{M}'$. These flows are mixing on their minimal component \cite{Sinai1992,kochergin03}. We state that the correlations of box indicators cannot decay faster than $(\log t)^{-1}$ along a sequence of times for a typical flow. Here, a box is a basic set in $\mathcal{M}'$; see Section \ref{sec:flowsT2} for the precise definition.

\begin{Main}
\label{thm:log_decay_flow}
Let $(\phi_t^{\mathcal{M}})_{t \in \mathbb{R}}$ be a locally Hamiltonian flow on $\T^2$ with a single saddle point, with a homoclinic loop and $\mathcal{M}' \subset \T^2$ the minimal component. For almost every such $(\phi_t^{\mathcal{M}})$, the following holds.
For any box $A \subset \mathcal{M}'$, there exists $C > 0$ such that for all $t_0$ there exists $t \geq t_0$ satisfying
\[
|\mathrm{Cor}_t(\chi_{A}, \chi_{A})|
\geq \frac{C}{\log t}.
\]
\end{Main}

Theorem \ref{thm:log_decay_flow} should be compared with \eqref{eq:sk_rate}. Sinai--Khanin, showed that correlations of rectangle indicators on flows with a single homoclinic saddle typically decay at rate $(\log t)^{-\nu}$ for any $\nu < 1/4$; here we show that, along a subsequence of times, the correlations for a full measure set of such flows cannot decay faster than $(\log t)^{-1}$. Together, these two results confirm that the decay of correlations in this setting is typically polynomial in $\log t$.

It is worth noting that, while the local shearing generated near a non-degenerate saddle point is of logarithmic order, a global lower bound of order $(\log t)^{-1}$ is not an automatic consequence of this local feature. On a global scale, as the surface flow transports a box, it repeatedly passes near the singularities and splits into many components. A priori, these pieces could experience mutually compensating asymmetries in their distribution that cancel out rather than accumulating into a macroscopic correlation. The primary difficulty and core technical contribution of our proof lies in showing that, along a curated sequence of times, this local phenomenon persists globally.

The last result is motivated by connections to spectral theory: recall that in \cite{FayadForniKanigowski2021} Fayad--Forni--Kanigowski prove countable Lebesgue spectrum for certain flows on the torus using 
square-integrability (or square-summability) of correlations. The following theorem shows that this square-integrability fails for a class of flows in $ \T^2$, providing an obstruction to applying those spectral methods directly. In contrast to Theorem \ref{thm:log_decay_flow}, the set of flows for which this holds has measure zero.

\begin{Main} \label{thm:C}
There exists a class of locally Hamiltonian flows $(\phi_t^{\mathcal{M}})_{t \in \mathbb{R}}$ on $\T^2$ such that there exists a box $A \subset \mathcal{M}'$ with
\[
\int_0^{\infty} |\mathrm{Cor}_t(\chi_{ A}, \chi_{ A})|^2\, dt = \infty.
\]
\end{Main}

To prove these results, we reduce them to statements for suspension flows over interval exchange transformations. We introduce these objects in Section \ref{sec:mainresults}

\subsection{Organization of the paper}
The remainder of this paper is organized as follows. 

In Section \ref{sec:mainresults}, we establish our framework. Subsections \ref{sec:intro_IETs} and \ref{sec:susflows} provide the definitions of interval exchange transformations and suspension flows, followed by a formulation of their respective measure spaces and typicality parameters in Subsection \ref{sec:typical}. With the background established, Subsections \ref{sec:thm21_statement} and \ref{sec:flowsT2} state the counterparts of our main results on the suspension space: Theorems \ref{thm:IETs}, \ref{thm:log_decay}, and \ref{thm:sum_cor}.

Because Theorems \ref{thm:log_decay_flow} and \ref{thm:C} only involve box indicator functions, their translation to the suspension space statements (Theorems \ref{thm:log_decay} and \ref{thm:sum_cor}) is immediate. In contrast, Theorem \ref{thm:flows} involves $C^1$ observables whose norms are distorted near the saddle points. Section \ref{sec:reduction} carries out the geometric analysis required to bound this distortion and establish the reduction of Theorem \ref{thm:flows} to Theorem \ref{thm:IETs}.

The proof of Theorem \ref{thm:IETs} is subsequently divided into three sections. Section \ref{sec:IETs} introduces two crucial quantitative conditions \eqref{eq:IET1}--\eqref{eq:IET2} on IETs, shown to hold for a full measure set. In Section \ref{sec:testfunctions}, under a set of generic hypotheses \eqref{eq:conditions_for_g}, we construct our family of test functions. Section \ref{sec:proofmain} completes the proof of Theorem \ref{thm:IETs} by using properties \eqref{eq:IET1}--\eqref{eq:IET2} to select a sequence of times satisfying \eqref{eq:conditions_for_g} and showing that the corresponding test functions produce a large correlation gap. 

Finally, Section \ref{sec:convexity} contains convexity lemmas quantifying the asymmetric distribution of the flow image of a thin horizontal strip on the suspension space, which are the main geometric input for the proofs of Theorems \ref{thm:log_decay} and \ref{thm:sum_cor} in Sections \ref{sec:logdecay} and \ref{sec:last}, respectively.

\section{Preliminaries and reduction to suspension flows} \label{sec:mainresults}
\subsection{Interval exchange transformations} \label{sec:intro_IETs}
Interval exchange transformations (IETs) form a central class of dynamical systems. They generalize circle rotations and provide fundamental examples of low-complexity, measure-preserving systems. Some important properties of IETs are presented in Section \ref{sec:IETs}; here we recall their definition.

A \emph{right-continuous interval exchange transformation} (IET) on $d$ intervals on $[a,b)$ is determined by a length vector
$\underline\lambda = (\lambda_1,\dots,\lambda_d) \in \mathbb{R}^d_+,$ $\sum_{i=1}^d \lambda_i = b-a,
$
and an irreducible permutation $\pi \in S_d$. The interval $[a,b)$ is partitioned into subintervals
\[
I_i = \Bigl[a+\sum_{k<i}\lambda_k,a+ \sum_{k\leq i}\lambda_k\Bigr), \qquad i=1,\dots,d,
\]
and the transformation $T$ acts by translation on each subinterval. In particular, there exist
$
a = y_0 < y_1 < \cdots < y_{d-1} < y_d = b
$
and translation vectors \[\sigma_1,\ldots,\sigma_d \in \R, \qquad \sigma_i=\sum_{\pi(k)<\pi(i)} \lambda_k - \sum_{k<i} \lambda_k\] such that $Tx= x + \sigma_i$ for $x \in I_i=[y_{i-1}, y_i)$. The IET $T$ preserves Lebesgue measure and has discontinuities at $D=\{y_1,\ldots,y_{d-1}\}$.
For simplicity, we will assume $a=0$ and $b=1$.

\subsection{Suspension flows} \label{sec:susflows}
Let $T$ be a measure-preserving endomorphism on $\mathbb{T}$, which we identify with $[0,1)$. We consider the special flow $(\phi_t)_{t\in\mathbb{R}}$ over $T$ built under the roof function $\varphi:[0,1]\to\R^{+}\cup\{\infty\}$, satisfying
\begin{equation}
\gamma := \inf_{x \in \mathbb{T}} \varphi(x) > 0,\qquad
\int_{\mathbb{T}}\varphi(x)\,dx=1.
\label{eq:def_gamma}
\end{equation}
Note that $\gamma\leq1$.
The flow is defined on the space
\[
X = \left\{(x,s)\in\T\times\mathbb{R} : 
0\leq s<\varphi(x)\right\}
\]
with the identification $(x,\varphi(x))\sim (Tx,0)$ at the points where $\varphi(x)<\infty$, and acts by vertical translation
\[
\phi_t(x,s)=(x,s+t)
\]
followed by the canonical identifications whenever the second coordinate exits the interval $[0,\varphi(x))$. Given a function $h:\mathbb{T}\to\mathbb{R}$, we write its Birkhoff sums
\[
S_l(h)(x):=\sum_{k=0}^{l-1} h(T^k x),
\]
and for $t\geq 0$ define $
N(x,t):=\max\{k\geq 0 :t-S_k(\varphi)(x)\geq 0\}.$
Clearly $ S_l(\varphi)(x) \geq \gamma l$, for any $  x \in \T$ and $ l\geq 1$, so that $N(x,t)  \leq t/\gamma$. Then, the flow $\phi_t(x,s)$ can be defined explicitly as
\[
\phi_t(x,s)=\left( T^{N(x,t)}x, t - S_{N(x,t)}(\varphi)(x) \right).
\]

It is a classical result (see \cite{KanigowskiOkunevZelada2026} and the references therein) that any locally Hamiltonian flow with finite, non-degenerate singularities, when restricted to a minimal component $\mathcal{M}'$, is measure-theoretically isomorphic to a special flow $ (\phi_t^{\mathcal{M}}\vert_{\mathcal{M}'},\mathcal{M}')\simeq(\phi_t,X)$ over a minimal IET $T:[0,1)\to [0,1)$, under a roof function that has logarithmic singularities at the (finitely many) discontinuity points of $T$ and at 0, and is smooth elsewhere. We denote the isomorphism by $\Phi:X\to \mathcal{M}'$.

In particular, $\Sigma:=\Phi([0,1]\times\{0\})\subset \mathcal{M}'$ is a smooth closed curve transversal to the flow $(\phi_t^{\mathcal{M}})$, smoothly parametrized by $\Phi(x,0)$, and $\varphi(x)$ is the return time of $(\phi_t^{\mathcal{M}})$ to $\Sigma$ under this parametrization, so that $    \phi_{\varphi(x)}^{\mathcal{M}}\Phi(x,0)=\Phi(Tx,0).$

Letting $D\subset[0,1)$ be the set of discontinuity points of $T$, we can write $\varphi$ as
\begin{equation}\begin{aligned}
\varphi(x) &= \psi(x) + c_0^+\,|\log x| 
+ c_0^-\,|\log(1-x)\,|\\&+\sum_{y\in D}c_y^+\,|\log \,(x-y)\,|\,\chi_{\{x>y\}}+\sum_{y\in D}c_y^-\,|\log \,(y-x)\,|\,\chi_{\{x<y\}},
\end{aligned}
\label{eq:def_phi}
\end{equation}
where $\psi \in L^\infty(\mathbb{T})$ has bounded variation and is smooth on $\T\setminus D$, and $c_0^\pm$ and $\{c_y^\pm\}_{y\in D}$ are {positive} constants. 
In particular, there exist real coefficients $\{a_y^{\pm},b_y^{\pm}\}_{y\in D\cup\{0\}}$ such that $\psi(x)=\psi_0(x)+\psi_1(x)+\psi_2(x)$ with
\be \label{eq:def_psi0}
\begin{aligned}
\psi_0(x) &=b_0^+\,x\,|\log x| 
+ b_0^-\,(1-x)\,|\log(1-x)\,|\\&+\sum_{y\in D}b_y^+\,(x-y)\,|\log \,(x-y)\,|\,\chi_{\{x>y\}}+\sum_{y\in D}b_y^-\,(y-x)\,|\log \,(y-x)\,|\,\chi_{\{x<y\}},
\end{aligned}
\ee
\be \label{eq:def_psi1}
\begin{aligned}
\psi_1(x) &=a_0^+\,x^2|\log x| 
+ a_0^-\,(1-x)^2|\log(1-x)\,|\\&+\sum_{y\in D}a_y^+\,(x-y)^2|\log \,(x-y)\,|\,\chi_{\{x>y\}}+\sum_{y\in D}a_y^-\,(y-x)^2|\log \,(y-x)\,|\,\chi_{\{x<y\}},
\end{aligned}
\ee
and the weak derivatives 
\be \label{eq:def_psi2}
\psi_1',\psi_2',\psi_2''\in L^\infty(\T).
\ee

We say that the roof $\varphi $ is \textit{symmetric} if $c_0^++\sum_{y\in D}c_y^+=c_0^-+\sum_{y\in D}c_y^-$. Otherwise, it is \textit{asymmetric}, and we say that $(\phi_t^{\mathcal{M}})$ has asymmetric logarithmic singularities on $\mathcal{M}'$.

Denote by $\mu$ the Lebesgue measure  on $X$, and for $f,g \in L^2(X)$ and $t \in \R$ write
\[
\mu(f) := \int_X f\, d\mu, \qquad
\langle f \circ \phi_t, g \rangle := \int_X (f \circ \phi_t)\, g\, d\mu,
\]
and define the correlation at time $t$ by
\[
\mathrm{Cor}_t(f,g) := \langle f \circ \phi_t, g \rangle - \mu(f)\mu(g).
\]
Since $\Phi: X \to \mathcal{M}'$ is a measure-preserving isomorphism, this definition is consistent with \eqref{eq:cor_def}. That is,
for $f,g\in L^2(X)$, define $\tilde{f} = f \circ \Phi^{-1}$ and $\tilde{g} = g \circ \Phi^{-1}$ on $\mathcal{M}'$. 
Then 
\[
|\mathrm{Cor}_t(f,g) |= \left|\langle f \circ \phi_t, g \rangle - \mu(f)\mu(g)\right|= |\int_{\mathcal{M}'} (\tilde f\circ\phi_t^{\mathcal{M}} )\tilde g\,\omega- \int_{\mathcal{M}'} \tilde f \,\omega\int_{\mathcal{M}'} \tilde g\,\omega|=|\mathrm{Cor}_t(\tilde f, \tilde g) |.
\]

\subsection{Measures and typicality} \label{sec:typical}
In order to prove Theorems \ref{thm:flows}, \ref{thm:log_decay_flow}, and \ref{thm:C}, we will show related results stated in the suspension flow representation in Section \ref{sec:susflows}.

The results are stated for \emph{typical} locally Hamiltonian flows, typical IETs, and typical rotation parameters $\alpha$. For IETs, typicality refers to a full measure set of length vectors $\underline\lambda$ for each irreducible permutation. IETs and the relevant measure class are recalled in Section \ref{sec:IETs}. Rotations are a particular case of IETs, for which we consider Lebesgue measure in (0,1). For locally Hamiltonian flows, full measure is with respect to Katok's fundamental class, a natural measure class on the space of smooth closed $1$-forms; we refer the reader to \cite{Katok1973, ravotti} for its definition and properties, and note only that it corresponds to a full measure set of IETs (respectively rotations) under the reduction described in Section \ref{sec:susflows}.

\subsection{Suspension flow equivalent of Theorem \ref{thm:flows}}\label{sec:thm21_statement}
We first introduce some notation before stating a result in the setting of the suspension flow representation that will imply Theorem \ref{thm:flows}. 

Since $X$ is defined as a quotient of $\mathbb{T}\times\mathbb{R}$, it does not carry a natural manifold structure globally: the identification $(x,\varphi(x))\sim(Tx,0)$ creates a gluing boundary along which the functions are not smooth. This is why we work with functions defined on rectangles in $\mathbb{T}\times\mathbb{R}$ before passing to the quotient.

For a closed set $A$, we define $C^1_c(A) :=C^1_c(A^\circ)$, i.e. $C^1$ functions with compact support contained in the interior of $A$. Given a rectangle $A=[a,b]\times[0,t_0]\subset\mathbb{T}\times\mathbb{R}$ such that the quotient map $q:\mathbb{T}\times\mathbb{R}\to X$ restricts to an injection on $A$, any $f,g\in C^1_c(A)$ descend to a well-defined function on $X$ via $q$, extended by zero outside $q(A)$. 
For $f=f(x,s)$, we set
\[
\|f\|_{C^1(A)} := \|f\|_\infty 
+ \|\partial_x f\|_\infty + \|\partial_s f\|_\infty,
\]
with all norms taken on $A$, and define
\[
\langle f,g\rangle:=\langle f\circ q^{-1},\, g\circ q^{-1}\rangle,
\qquad\mu(f):=\mu(f\circ q^{-1}), \]\[\mathrm{Cor}_t(f,g) := \mathrm{Cor}_t(f\circ q^{-1},\, g\circ q^{-1}).
\]
Throughout, whenever $f\in C^1_c(A)$ appears in a correlation, it is implicit that $A$ satisfies the above injectivity condition and that the $C^1$ norm is computed on $A$.

We can now state Theorem \ref{thm:IETs}.

\begin{theorem}
\label{thm:IETs}
For a typical IET $T$, consider the suspension flow $(\phi_t)$ over $T$ and under a roof function $\varphi$ as in \eqref{eq:def_phi}, with singularity set $D_0 \subset \mathbb{T}$. There exist $C(\varphi,T)>0$ and an increasing sequence $(t_m)$ such that for every $m\geq 1$ there exist $s_m>0$, $0 < a_m< b_m< 1$, and a function $g_m \in C_c^1(A_m)$, where $A_m:= [a_m,b_m] \times [0,s_m]$ embeds injectively into $X$ via $q$, such that
\begin{equation}
\label{eq:IET_cor_lower}
|\mathrm{Cor}_{t_m}(g_m,g_m)|>C.
\end{equation}
Moreover, for any $\nu > 0$ there exist $M(\varphi,T,\nu)\geq 1$ such that for $m\geq M$
\begin{equation}
\label{eq:IET_dbounds}
\norm{g_m}_{\infty}=1,\qquad\norm{\partial_xg_m}_{\infty}\leq(\log t_m)^{1+\nu},\qquad\norm{\partial_sg_m}_{\infty}\leq(\log t_m)^{-1+\nu},
\end{equation}
and for every $x\in [a_m,b_m]$,
\begin{equation}
\label{eq:IET_dist}
\min_{k \in \{0, \cdots,  N(x,s_m)\}}\bigl\{\mathrm{dist}(T^kx,\,D_0)\bigr\} \geq (\log t_m)^{-1-\nu}, \quad N(x,s_m)\leq (\log t_m)^{1+\nu}.
\end{equation}
\end{theorem}

Theorem \ref{thm:IETs} holds for any roof function $\varphi$ as in \eqref{eq:def_phi}, symmetric or asymmetric. The restriction to asymmetric $\varphi$ in Theorem \ref{thm:flows} is not a limitation of the method, it reflects the context. Indeed, by Ulcigrai \cite{ulcigrai11}, almost every locally Hamiltonian flow with symmetric singularities is not mixing, so a lower bound on the rate of mixing is only meaningful in the asymmetric case, where mixing is known to hold almost surely \cite{ulcigrai07} and Ravotti's \cite{ravotti} upper bound \eqref{eq:round_rav} applies.

The control on the support of $g_m$ in Theorem \ref{thm:IETs} allows us to transfer the lower bound to locally Hamiltonian flows. Since the supports are bounded away from the singularities of $\varphi$ in the suspension flow representation by \eqref{eq:IET_dist}, the corresponding points $\Phi(x,s)$ on $\mathcal{M}'$ are uniformly bounded away from the saddles of $H$.
This allows one to compare the $C^1$-norms of functions on $X$ and $\mathcal{M}'$ via the isomorphism $\Phi$. This comparison of norms is carried out in Appendix \ref{app:norm_estimates}.

\subsection{Suspension flow equivalents of Theorems \ref{thm:log_decay_flow} and \ref{thm:C}}\label{sec:flowsT2}
For Theorems \ref{thm:log_decay_flow} and \ref{thm:C} we consider flows on $\mathbb{T}^2$ with a single non-degenerate saddle point with a homoclinic loop, corresponding to suspension flows over a rotation $R_\alpha$. 
In this setting, \eqref{eq:def_phi}--\eqref{eq:def_psi2} reduce to
\be \label{eq:phi_basic}
\varphi(x) = \psi(x) +2c_0\,|\log x\,| + c_0\,|\log(1-x)\,|,
\ee
where $\psi\in L^\infty(\T)$ is of bounded variation and smooth on $(0,1)$, with
\be \label{eq:phi_Dbasic}
\begin{aligned}
\psi(x)=\psi_0(x)+\psi_1(x)+\psi_2(x)&
=2b_0\,x\,|\log x| + b_0\,(1-x)\,|\log(1-x)\,|\\&+2a_0\,x^2|\log x| + a_0\,(1-x)^2|\log(1-x)\,|+\psi_2(x),
\end{aligned}
\ee
for some $c_0>0$, $a_0,b_0\in\R$ and $\psi_2',\psi_2''\in L^{\infty}(\T)$. 
These flows are mixing for any $\alpha \in \mathbb{R} \setminus \mathbb{Q}$ \cite{Sinai1992,kochergin03,Kochergin2004}. The asymmetry $2c_0 \neq c_0$ is what drives mixing, and its role will be apparent in the proofs of Sections \ref{sec:logdecay}--\ref{sec:last}. 

We say that $A = [a,b] \times [c,d] \subset X$ is a \emph{box} if $0 \leq a < b \leq 1$ and $0 \leq c < d < \gamma/3$ (recall $\gamma:= \inf\varphi $). Equivalently, using the notation from Section \ref{sec:susflows}, a box $\tilde A\subset \mathcal{M}'$ is a segment of a closed curve $\Gamma:=\Phi([a,b]\times\{0\})\subset \Sigma$ transversal to the flow $(\phi_t^{\mathcal{M}})$, flowed for a bounded time: \[\tilde A=\bigcup_{t\in[c,d]}\phi_t^{\mathcal{M}}(\Gamma).\]

Theorem \ref{thm:log_decay_flow} shows that, in this setting, the correlations of box indicators cannot decay faster than $(\log t)^{-1}$ along a sequence of times for a typical flow, i.e., for a full-measure set of rotations. By our previous observations, Theorem \ref{thm:log_decay_flow} is a corollary of the following result about the system $(X,\mu,\phi_t)$.

\begin{theorem}
\label{thm:log_decay}
Consider the suspension flow $(\phi_t)$ under a roof function as in \eqref{eq:phi_basic}, over a rotation $R_\alpha$. Then for a full measure set of $\alpha \in (0,1)$, for any box $A \subset X$, there exists $C > 0$ such that for all $t_0$ there exists $t \geq t_0$ satisfying
\[
|\mathrm{Cor}_t(\chi_A, \chi_A)|
= \left|\mu(A \cap \phi_t A) - \mu(A)^2\right|
\geq \frac{C}{\log t}.
\]
\end{theorem}
The precise Diophantine condition on $\alpha$ is introduced in Section \ref{sec:logdecay}.

\begin{remark}
The asymmetric logarithmic singularities of the roof function $\varphi$ naturally yield a shearing of logarithmic order. However, this local feature could cancel out rather than accumulate into a macroscopic correlation.
We exploit the strict convexity of the roof function near the singularities (as quantified via the convexity lemmas in Section \ref{sec:convexity}) to rule out, along a sequence of times, full global cancellations, yielding the global lower bound.
\end{remark}

Finally, for any $\alpha\in\RQ$ there exists an Arnol'd flow $(\phi_t^{\mathcal{M}})$ in $\T^2$ with return map $R_\alpha$ on the suspension flow equivalent to $(\phi_t^{\mathcal{M}}\vert_{\mathcal{M}'})$ on its minimal component $\mathcal{M}'$ \cite{ConzeFraczek2011}. Hence, the proof of Theorem \ref{thm:C} is reduced to that of an equivalent result in the suspension flow, which holds for a Lebesgue measure zero set of $\alpha$.

\begin{theorem}
\label{thm:sum_cor}
There exists $\alpha \in \mathbb{R} \setminus \mathbb{Q}$ such that for the suspension flow $(\phi_t)$ over $R_\alpha$ under \eqref{eq:phi_basic}, there exists a box $A \subset X$ with 
\be
\int_0^{\infty} |\mathrm{Cor}_t(\chi_A, \chi_A)|^2\, dt = \infty.
\label{eq:thm3}
\ee
\end{theorem}

\subsection{Notation}
Throughout this paper, we identify $\mathbb{T}$ with $[0,1)$. We denote by $\Pi_x$ and $\Pi_s$ the projections onto the first and second coordinates of elements of $X$. We write $\mu$ for the Lebesgue measure on $X$ and $\lambda$ for the Lebesgue measure on $\mathbb{T}$; the length vector of an IET is always written $\underline{\lambda} = (\lambda_1,\dots,\lambda_d)$ to distinguish it from $\lambda$. 

Given a measurable function $h: \mathbb{T} \to \mathbb{R}$ and a base map $T$, its Birkhoff sums $S_l(h)(x)$ are defined for $l \in \mathbb{Z}$ by
\[
S_l(h)(x) := \begin{cases}
\sum_{k=0}^{l-1} h(T^k x)\qquad  &\text{for}\quad l > 0,\\ 
0\qquad  &\text{for}\quad l = 0,\\
-\sum_{k=l}^{-1} h(T^k x) \qquad  &\text{for}\quad l < 0.
\end{cases}
\] 
We write $N(x,t) := \max\{k: t - S_k(\varphi)(x) \geq 0\}$ to denote the number of iterates of the base map visited by the orbit of $(x,0)$ on $X$ up to time $t$. We set $D_0 := D \cup \{0\} \subset \mathbb{T}$, where $D$ is the set of discontinuities of $T$, and write $\|x\| := \mathrm{dist}(x, D_0 + \mathbb{Z})$ for the distance from $x$ to the nearest singularity.

Constants that appear across multiple results are numbered $C_1, C_2, \dots$ within each section, in order of first appearance, with dependencies noted when introduced. Constants used only within a single proof are denoted $\hat{C}$ or $\tilde{C}$. 

We use the notation $A \lesssim B$ (resp. $A \gtrsim B$) when there exists a uniform constant $C > 0$ such that $A \leq CB$ (resp. $A \geq CB$), and write $A \asymp B$ when both relations hold simultaneously. Asymptotic comparisons follow standard conventions: $f(t)=O(g(t))$ if $\limsup_{t\to \infty}{|\tfrac{f(t)}{g(t)}|}<\infty$, $f(t)=o(g(t))$ if $\limsup_{t\to \infty}{\tfrac{f(t)}{g(t)}}=0$, and $f(t)\sim g(t)$ if $f(t) = g(t) + o(f(t))$.

\section{Reduction of Theorem \ref{thm:flows} to Theorem \ref{thm:IETs}}
\label{sec:reduction}
In this section, we deduce Theorem \ref{thm:flows} from Theorem \ref{thm:IETs} using the measure-theoretic isomorphism $\Phi: X \to \mathcal{M}'$ introduced in Section \ref{sec:susflows}. The argument relies on two facts: first, that the measure-theoretic isomorphism $\Phi: X \to \mathcal{M}'$ preserves the correlations exactly and, secondly, that the $C^1(\mathcal{M}')$ norm of the pushforward $\tilde{g} _m = g_m \circ q^{-1} \circ \Phi^{-1}$ is of the same
order as the $C^1(A_m)$ norm of $g_m$, both being $O((\log t_m)^{1+\nu})$.

\begin{proposition}
\label{prop:C1norm}
Let $g$ satisfy \eqref{eq:IET_dbounds}--\eqref{eq:IET_dist} with $\nu <1$ and define $\tilde g=g\circ q^{-1}\circ \Phi^{-1}$. Then
\be
\|\tilde{g}\|_{C^1(\mathcal{M}')} = \|\tilde{g}\|_\infty + \|d\tilde{g}\|_\infty = O((\log t)^{1+\nu}).
\ee
\end{proposition}
We leave the proof of Proposition \ref{prop:C1norm} to Appendix \ref{app:norm_estimates} and proceed with the proof of Theorem \ref{thm:flows}.

\begin{proof}[Proof of Theorem \ref{thm:flows}]
Let $(\phi_t^{\mathcal{M}})_{t \in \mathbb{R}}$ be a locally Hamiltonian flow and let $\mathcal{M}' \subset \mathcal{M}$ be the minimal component as in the statement. By the classical representation theorem recalled in Section \ref{sec:susflows}, the restriction $(\phi_t^{\mathcal{M}}|_{\mathcal{M}'})$ is measure-theoretically isomorphic to a suspension flow $(\phi_t, X)$ over an IET $T$ under a roof function $\varphi$ as in \eqref{eq:def_phi}, via a measure-preserving isomorphism $\Phi: X \to \mathcal{M}'$. By \cite{ravotti}, the set of flows for which $T$ is typical has full measure with respect to Katok's fundamental class.

Fix $\nu > 0$ and $C' > 0$. By Theorem \ref{thm:IETs} (applied with exponent $\nu/4$), there exist an increasing sequence of times $(t_m)$, rectangles $(A_m)$, functions $(g_m)$ with $g_m \in C^1_c(A_m)$, and a constant $C>0$, all independent of $\nu$, such that for any $m$
\[
|\mathrm{Cor}_{t_m}(g_m, g_m)| \geq C,
\] 
and, for $m$ large enough, $g_m$ satisfies \eqref{eq:IET_dbounds}--\eqref{eq:IET_dist} for $\nu/4$.

For each $m$, define $\tilde{g} _m = g_m \circ q^{-1}\circ \Phi^{-1}$. Since $\Phi$ is measure-preserving and intertwines $(\phi_t)$ with $(\phi_t^{\mathcal{M}})$: $\text{Cor}_{t_m}(\tilde g_m,\tilde g_m) =\text{Cor}_{t_m}(g_m,g_m). $ 

By Proposition \ref{prop:C1norm}, for $m$ (hence $t_m$) sufficiently large
\[
\|\tilde{g} _m\|_{C^1(\mathcal{M}')}^2 \lesssim  (\log t_m)^{2+\nu/2}\leq \frac{C}{C'}(\log t_m)^{2+\nu}.
\]
Therefore, for any prescribed $C' > 0$, if $m$ is large enough
\begin{align*}
|\mathrm{Cor}_{t_m}(\tilde{g} _m, \tilde{g} _m)|
&\geq C= C\cdot\frac{\|\tilde{g} _m\|_{C^1(\mathcal{M}')}^2}{\|\tilde{g} _m\|_{C^1(\mathcal{M}')}^2} \geq C'\|\tilde{g} _m\|_{C^1(\mathcal{M}')}^2 (\log t_m)^{-2-\nu}.
\end{align*}

\end{proof}

\section{IETs: Background and Notation} \label{sec:IETs}

In this section, we recall the basic framework and notation for IETs and the key structures that will be used throughout; we refer the reader to \cite{Viana2008} for an extensive discussion.

Let $d \geq 2$. Recall that a $d$-interval exchange transformation $ T : \mathbb{T} \to \mathbb{T}$ is
determined by a length vector $\underline\lambda= (\lambda_1,\dots,\lambda_d) \in\Delta_{d}$, where
\[
\Delta_{d} := \left\{\underline\lambda\in\mathbb{R}^d_+ : \sum_{j=1}^d \lambda_j = 1\right\}
\]
is the $d$-simplex, and a permutation $\pi\in S_d$. The unit interval is partitioned into subintervals
\[
I_j = \Bigl[\sum_{k<j}\lambda_k,   \sum_{k\leq j}\lambda_k\Bigr),   \qquad j=1,\dots,d,
\]
and $T$ acts on each $I_j$ as a translation, rearranging them according to $\pi$ so that $T(I_j)$ occupies the $\pi(j)$-th position in the partition of $\mathbb{T}$.

Throughout this work, we assume that $T$ satisfies the Keane condition, namely that the orbits of the discontinuity points are infinite and pairwise disjoint. This condition ensures minimality and guarantees that the \emph{Rauzy--Veech} induction is well defined at every step. We will work with its accelerated version, the \emph{Zorich induction}.

The Rauzy--Veech induction is a renormalization procedure that, given an IET $T = (\underline\lambda, \pi)$, produces a new IET by inducing $T$ on the subinterval $[0, 1-\min(\lambda_{\pi^{-1}(d)}, \lambda_d))$, i.e. by removing the shorter of the two rightmost subintervals. The resulting IET has the same combinatorial type up to an update of $\pi$, and its length vector is obtained from $\underline\lambda$ by a matrix in $SL(d, \mathbb{Z})$. The Zorich induction $\mathcal{Z}$ is an accelerated version that repeatedly applies the Rauzy--Veech induction until the winning (the longest) subinterval changes, thereby grouping consecutive induction steps of the same type into a single step.

For each $n \geq 0$, let $\mathcal Z^{n}T=\left(\frac{\underline{\lambda}^{(n)}}{\lambda^{(n)}},\pi^{(n)}\right)$ denote the IET obtained after $n$ steps of the Zorich induction. This transformation is defined by inducing on a subinterval
\[
I^{(n)} \subset [0,1), 
\qquad 
I^{(n)} = \bigsqcup_{j=1}^d I_j^{(n)},
\]
where $\lambda^{(n)}:=\lambda(I^{(n)})$, $\lambda_j^{(n)}:=\lambda(I_j^{(n)})$, and $\underline\lambda^{(n)} := (\lambda_1^{(n)},\dots,\lambda_d^{(n)} )$.

The combinatorial and metric data of the induction are encoded by the Zorich cocycle matrices $\mathcal{B}_n \in \mathrm{SL}(d,\mathbb{Z})$, obtained as 
the product of the elementary Rauzy--Veech matrices corresponding to each substep grouped into the $n$-th Zorich step. The associated return times (or tower heights) are given by the column sums
\[
h_j^{(n)} := \sum_{i=1}^d (\mathcal{B}_n)_{i,j}h_i^{(n-1)}=\sum_{i=1}^d (\mathcal{B}_n\cdots\mathcal{B}_1)_{i,j}, \qquad j=1,\dots,d.
\]
These determine a Rokhlin tower decomposition of the interval: for each $j$,
\[
\mathcal{T}_j^{(n)} := \bigcup_{k=0}^{h_j^{(n)}-1} T^k I_j^{(n)}.
\]
Up to a countable set, the collection $\{\mathcal{T}_j^{(n)}\}_{j=1}^d$ forms a partition of $[0,1)$.

We fix a Rauzy class $\mathfrak{R} \subset \mathfrak{S}_d^{\mathrm{irr}}$, i.e., an equivalence class of irreducible permutations under the Rauzy induction moves. Since the Zorich induction preserves $\mathfrak{R}$, that is $\pi^{(n)} \in \mathfrak{R}$ for all $n \geq 0$ whenever $\pi \in \mathfrak{R}$, the induction $\mathcal{Z}$ acts as a self-map of $\mathcal{M}_d(\mathfrak{R}) := \Delta_{d} \times \mathfrak{R}$. 

We work with the Zorich induction rather than the raw Rauzy-Veech induction for two reasons. First, the Zorich induction is ergodic with respect to a natural absolutely continuous invariant measure $\nu^{\mathcal{R}}_{\mathrm{Zorich}}$ on $M_d(\mathcal{R})$, a fact we will use in the proof of Proposition \ref{prop:full_meas_IET}. Second, the quantitative results of Marmi-Moussa-Yoccoz \cite{marmi2005cohomological} on Roth-type conditions for IETs, are stated and proved in the language of the Zorich cocycle. Working with the same induction allows us to import their estimates directly.

In addition to the Keane condition, we impose two quantitative assumptions that control the geometry of the induced partitions. In particular, the IETs for which we will proof Theorem \ref{thm:IETs} satisfy: 
\begin{itemize}
\item[] For any $\theta > 0$ there exists $C_\theta > 0$ such that, for each $n$,
\begin{equation}
\min_i \lambda_{i}^{(n)} \geq C_\theta \left(\max_i \lambda_i^{(n)}\right)^{1+\theta}.
\tag{H1}\label{eq:IET1}
\end{equation}

\item[] For any $\varepsilon > 0$ there exist a sequence $\{r_m\}$ and an index $j^* \in \{1,\dots,d\}$ such that
\begin{equation}
\frac{\lambda_{j^*}^{(r_m)}}{\lambda^{(r_m)}} \geq \frac{3}{4},
\qquad \lambda\left(\mathcal{T}_{j^*}^{(r_m)}\right) \geq 1 - \varepsilon,
\quad \forall\, m.\tag{H2}
\label{eq:IET2}
\end{equation}
\end{itemize}

Condition \eqref{eq:IET1} provides uniform control on the relative sizes of subintervals under induction, while \eqref{eq:IET2} ensures the existence of a sequence of times at which a single Rokhlin tower captures almost the entire measure.

\begin{proposition}\label{prop:full_meas_IET}
The Keane condition together with \eqref{eq:IET1}--\eqref{eq:IET2} define a full measure subset in the space of interval exchange transformations.
\end{proposition}
\begin{proof}
It suffices to prove the statement for a fixed Rauzy class $\mathfrak{R}$, with respect to the product measure $\nu_{\mathfrak{R}} = \mathrm{Leb}_{\Delta_{d}} \otimes \mu_{\mathfrak{R}}$ on $\mathcal{M}_d(\mathfrak{R})$, where $\mu_{\mathfrak{R}}$ is the uniform measure on $\mathfrak{R}$. We treat the conditions in order, based on the results of Viana \cite{Viana2008} and Marmi-Moussa-Yoccoz \cite{marmi2005cohomological}.

\textit{The Keane condition.} For any fixed $\pi \in \mathfrak{R}$, the set of $\lambda \in \Delta_{d}$ violating the Keane condition is a countable union of hyperplanes defined by rational affine relations among the $\lambda_i$. These are $\mathrm{Leb}_{\Delta_{d}}$-null, so the Keane condition holds for $\nu_{\mathfrak{R}}$-almost every $(\underline\lambda,\pi) \in \mathcal{M}_d(\mathfrak{R})$.

\textit{Condition \eqref{eq:IET1}.} Condition \eqref{eq:IET1} is precisely the Roth-type condition on the Zorich induction introduced in Section 1.3 of \cite{marmi2005cohomological}, where it is shown to hold for $\nu_{\mathfrak{R}}$-almost every $(\underline\lambda,\pi) \in \mathcal{M}_d(\mathfrak{R})$.

\textit{Condition \eqref{eq:IET2}.} To verify this condition, we present an argument analogous to \cite[Prop. 3.14]{berk2025}, leveraging the topological openness of parameters in a natural extension of the space $\mathcal{M}_d(\mathfrak{R})$. We carry this out in Appendix \ref{app:H2}.
\end{proof}

By Proposition \ref{prop:full_meas_IET}, it is enough to restrict attention to interval exchange transformations satisfying \eqref{eq:IET1}–\eqref{eq:IET2} in order to prove Theorem \ref{thm:IETs}. We begin by constructing the test functions in a general framework in Section \ref{sec:testfunctions}, and then establish the theorem in Section \ref{sec:proofmain}.

\section{Construction of the test functions}\label{sec:testfunctions}
The goal of this section is to construct, given a pair $(I,h)$ consisting of a base interval and a return time, a test function $g \in C_c^1(A)$ for some set $A(I,h)\subset\T\times\R$.
The function $g$ is designed so that $\int_A g$ is close to the measure of $A$, while its $C^1$-norm grows linearly in $h$. The precise conditions on $(I,h)$ that make this possible are collected in hypothesis \eqref{eq:conditions_for_g} below.

Define a base interval $I\subset[0,1]$ and let $h\geq 2$ be the return time of $I$ under the action of $T$. That is, let 
\[
h:=\min\{k\geq1:T^kI\cap I\neq\varnothing\}.
\]
Let $T$ be an IET on $\T$ and $\varphi$. We say that $(I,h)$ satisfy the hypotheses \eqref{eq:conditions_for_g} with respect to $T$ if there exists a $C>0$ such that they satisfy the following:
\begin{equation}\label{eq:conditions_for_g}
\begin{aligned}
&\cdot  T^k \text{ is a translation on }I\text{ for }0\leq k\leq h,\\
&\cdot h^{-1}\leq C\lambda(I).
\end{aligned}
\tag{G} 
\end{equation}

Let $X$ be the suspension flow space of the roof function $\varphi$ over $T$, and recall $\gamma:=\inf_\T\varphi>0$.

\begin{lemma}
\label{lemma:g1}
Given an interval $I$ and a positive integer $h$ satisfying \eqref{eq:conditions_for_g}, setting $A := I \times [0,\gamma h/4] \subset \mathbb{T} \times \mathbb{R}$, $A$ embeds injectively into $X$ via the quotient map $q:\mathbb{T}\times\mathbb{R} \to X$.
\end{lemma}

\begin{proof}
Two distinct points $(x,s), (x',s') \in A$ are identified in $X$ only if $(x',s') = \phi_u(x,s)$ for some $u \neq 0$, which requires $s' = s + u - S_k(\varphi)(x)$ and $x' = T^k x$ for some $k \geq 1$. In particular, $x' \in T^k I \cap I$, which by \eqref{eq:conditions_for_g} cannot happen for $1 \leq k \leq h-1$. It remains to rule out $k \geq h$. Since $s, s' \in [0, \gamma h/4]$ and $S_k(\varphi)(x) \geq \gamma k \geq \gamma h$ for $k \geq h$, we have
\[
s' = s + u - S_k(\varphi)(x) \leq \frac{\gamma h}{4} -\gamma h < 0,
\]
contradicting $s' \geq 0$. Hence $q|_A$ is injective.
\end{proof}

\begin{lemma}
\label{lemma:g2}
Given $0\leq a<b\leq 1$, $t>0$, and $r\geq e^2$, set $A:=[a,b]\times[0,t] $. Then there exist $g\in C^1_c(A)$ satisfying, for $g=g(x,s)$,
\[ \begin{aligned}
&\cdot \int_A g\geq \left(1-\frac{4}{\log r}\right)t(b-a) ,
\\ &\cdot 0\leq g\leq 1,
\\ &\cdot \|\partial _xg\|_\infty\leq 2\,(b-a)^{-1} \log r,\\ &\cdot \|\partial _sg\|_\infty\leq 2\,t^{-1} \log r.
\end{aligned} \]
\end{lemma}
\begin{proof}
For $r\geq e^2$ let $f\colon[0,1]^2\to\mathbb{R}$ be a smooth non-negative function, $f$ compactly supported on $(0,1)^2$, satisfying
\[\begin{aligned}
 & \cdot f=1 \quad \text{ on } \quad[1/\log r,1-1/\log r]^2,
\\ &\cdot 0\leq f\leq 1,
\\ &\cdot \|Df\|_\infty\leq 2 \log r.
\end{aligned} \]
Such an $f$ exists by a standard smooth cutoff construction. Define $g: A \to \mathbb{R}$ by
\[
g(x,s) := f\bigl(\tfrac{x-a}{b-a},\,
\tfrac{s}{t}\bigr).
\]
Clearly $g\in C^1(A)$ and, since $f$ is compactly supported in $(0,1)^2$, $g$ is compactly supported in $A^\circ$. Hence $g\in C^1_c(A)$. The bound $0\leq g\leq 1$ follows immediately from $0\leq f\leq 1$.

By the chain rule,
\[
\|\partial_x g\|_\infty   = (b-a)^{-1}\|\partial_1 f\|_\infty   \leq 2(b-a)^{-1}\log r,
\quad \|\partial_s g\|_\infty = t^{-1}\|\partial_2 f\|_\infty \leq 2t^{-1}\log r.
\]

Finally, since $f = 1$ on $[1/\log r, 1-1/\log r]^2$ and $0 \leq f \leq 1$ elsewhere,
\[
\int_A g = t(b-a)\int_{[0,1]^2} f \geq t(b-a)\left(1 - \frac{2}{\log r}\right)^2 \geq t(b-a)\left(1 - \frac{4}{\log r}\right).
\]
\end{proof}

Combining the two lemmas, we obtain the main tool of this section. Given $(I,h)$ satisfying \eqref{eq:conditions_for_g}, the following corollary provides a test function on a rectangle that embeds injectively into $X$, with integral close to the area of the rectangle and partial derivatives controlled by $h\log h$ and $h^{-1}\log h$.

\begin{corollary}\label{cor:testfunction}
Let $(I,h)$ satisfy \eqref{eq:conditions_for_g} and set $A := I\times[0,\gamma h/4]$. Moreover, assume $h\geq 8$. Then $A$ embeds injectively into $X$ via $q$, and there exists $g\in C^1_c(A)$ with $0\leq g\leq 1$, 
\be
\int_A g \geq \left(1-\frac{4}{\log h}\right)
\frac{\gamma h}{4}\,\lambda(I),\quad 
\|\partial_x g\|_\infty \leq 2C\, h\log h,\quad 
\|\partial_s g\|_\infty \leq 8\gamma^{-1} h^{-1}\log h,\label{eq:g}
\ee
where $C$ is the constant in \eqref{eq:conditions_for_g}.
\end{corollary}
\begin{proof}
Injectivity follows from Lemma \ref{lemma:g1}. 
Apply Lemma \ref{lemma:g2} with $t = \gamma h/4$ and $r=h$, 
so that $(b-a)^{-1} \leq Ch$ by \eqref{eq:conditions_for_g}, giving the bounds on $\|Dg\|_\infty$.
\end{proof}

\section{Proof of Theorem \ref{thm:IETs}} \label{sec:proofmain}
To prove Theorem \ref{thm:IETs}, we will find a sequence of intervals and integers $\{I_m,h_m\}$ that satisfy the hypothesis \eqref{eq:conditions_for_g}. Then we will apply Corollary \ref{cor:testfunction} to obtain a sequence of rectangles $A_m:=I_m\times[0,h_m/4]$, functions $g_m\in C_c^1(A_m)$ with $0\leq g_m\leq 1$, and times $t_m\nearrow \infty$ for which
\[
|\text{Cor}_{t_m}(g_m,g_m)|\geq C,
\]
for some constant $C>0$ independent of $m$ and, for $g_m=g_m(x,s)$,
\[ \|\partial_x g_m\|_{\infty}\lesssim(\log t_m)^{{(1+\theta)^2}}\log^{2}\log t_m,\quad \|\partial_s g_m\|_{\infty}\lesssim(\log t_m)^{-1}\log^{2}\log t_m,
\]
where $\theta$ is the Diophantine constant in \eqref{eq:IET1}.
The Theorem will follow from this. 

The proof proceeds in three steps. First, in Section \ref{sec:birkhoff} we establish bounds on the Birkhoff sums of $\varphi$ and $\varphi'$ along the induction sequence. Second, in Section \ref{sec:choicegm} we use conditions \eqref{eq:IET1}--\eqref{eq:IET2} to select a sequence of pairs $(I_m, h_m)$ satisfying \eqref{eq:conditions_for_g}, and define the corresponding test functions $g_m$ and times $t_m$. Third, in Section \ref{sec:lowerbounds} we show that the choice of $t_m$ forces the support of $\phi_{t_m}(g_m)$ to be nearly disjoint from that of $g_m$. This, combined with the lower bound on $\mu(g_m)$, produces a gap $|\text{Cor}_{t_m}(g_m,g_m)|\geq \gamma^2/32d^2$. A good control on $\|g_m\|_{C^1} $ yields the desired bounds.

\begin{remark}[Parameter hierarchy]
Throughout the proof of Theorem \ref{thm:IETs} we work with several small parameters whose roles are distinct. The parameter $\theta > 0$ controls the Roth-type condition \eqref{eq:IET1} and ultimately the exponent in \eqref{eq:IET_cor_lower}. It is chosen last in the proof. The parameter $\varepsilon > 0$ controls the fraction of measure captured by the dominant Rokhlin tower in \eqref{eq:IET2}. It is chosen depending on $\varphi$ and an auxiliary parameter $\delta$. The parameter $\delta > 0$ controls the truncation of orbits away from singularities, and it is chosen small but fixed throughout Sections \ref{sec:birkhoff}--\ref{sec:lowerbounds}.
\end{remark}

\subsection{Bounds on the Birkhoff sums} \label{sec:birkhoff}

To construct the sequence $(g_m,t_m)$, we need to first stablish some bounds on the Birkhoff sums of $\varphi$ and $\varphi'$ at each induction step. The techniques used are standard and should be compared to those in \cite{ulcigrai07,ravotti,ChaikaFraczekKanigowskiUlcigrai2019} and the references therein.

 Let $D_0=D\cup \{0\}$, with $D\subset \T$ the set of discontinuities of $T$. Clearly $|D_0|\leq d$. For any $x\in\mathbb{R}$, denote $\|x\| := \mathrm{dist}(x, D_0+\mathbb{Z})$. Recall that the roof function $\varphi$ is given by \eqref{eq:def_phi}--\eqref{eq:def_psi2}. In particular, its derivative for $x\in \T\setminus D_0$ is
\be \label{eq:def_Dphi}
\begin{aligned}
\varphi'(x) &= \psi'(x) - \frac{c_0^+}{x} +\frac{c_0^-}{1-x} -\sum_{y\in D}\frac{c_y^+}{x-y}\chi_{\{x>y\}}+\sum_{y\in D}\frac{c_y^-}{y-x}\chi_{\{x<y\}}\\
&= \psi_1'(x)+\psi_2'(x) - \frac{c_0^+}{x} +\frac{c_0^-}{1-x} -\sum_{y\in D}\frac{c_y^+}{x-y}\chi_{\{x>y\}}
+\sum_{y\in D}\frac{c_y^-}{y-x}\chi_{\{x<y\}}\\
&-b_0^++ b_0^--\sum_{y\in D}b_y^+\chi_{\{x>y\}}+\sum_{y\in D}b_y^-\chi_{\{x<y\}}+b_0^+\,|\log x| 
- b_0^-\,|\log(1-x)\,|\\
&+\sum_{y\in D}b_y^+\,|\log \,(x-y)\,|\,\chi_{\{x>y\}}-\sum_{y\in D}b_y^-\,|\log \,(y-x)\,|\,\chi_{\{x<y\}}.
\end{aligned}
\ee

The following results hold for any IET $T$ for which the induction process is well defined. Hence, for all IET satisfying the Keane condition. 

Fix a small constant $0.01>\delta > 0$. 
Let $n\in \mathbb{N}$, $i\in\{1,\cdots,d\}$. Let \[q_{i}^{(n)}:=1/\lambda_{i}^{(n)},\] let $I_{i}^{(n)}=(a,b)\subset(0,1)$, and let $h_i^{(n)}$ be its first return time to $I^{(n)}$. 
It holds that 
\[
\frac{h_{i}^{(n)}}{q_{i}^{(n)}} =\lambda\left(\mathcal{T}_{i}^{(n)}\right)\leq 1\qquad\implies\qquad h_{i}^{(n)}\leq q_{i}^{(n)}.
\]
Moreover, we recall the classical result for IETs satisfying the Keane condition \[\lim_{n\to \infty}\min_{i}h_{i}^{(n)}=+\infty.\]

Define \[
F_{n,i}=F_{n,i}(\delta):=\left(a+\delta/q_{i}^{(n)},b-\delta/q_{i}^{(n)}\right).
\]

The following lemma gives the key quantitative input: along the orbit of any point in $F_{n,i}$, the roof function $\varphi$ is logarithmically bounded and its derivative $\varphi'$ grows at most linearly in $q^{(n)}_i$, while the Birkhoff sums of $\varphi'$ are controlled by $q^{(n)}_i \log h^{(n)}_i$. These are precisely the bounds required by hypothesis \eqref{eq:conditions_for_g}.

\begin{lemma}
There exists constants $C_1=C_1(\delta,\varphi,T)>0$ and $N=N(T)\geq 1$ such that for all $n\geq N$
\begin{equation}
\label{eq:BSD_IET}
|S_{l} (\varphi')(x)| \leq C_1 q_{i}^{(n)} \log h_{i}^{(n)} \qquad \forall\, x \in F_{n,i},\qquad 0\leq l\leq h_{i}^{(n)}.
\end{equation}
\label{lemma:BSD_IET}
\end{lemma}
\begin{proof}

Recall the definition of $\varphi'$ in \eqref{eq:def_Dphi}. Let $1\leq l\leq h_{i}^{(n)}$ and consider $S_l (\varphi')(x)$.
The term $\psi_1'(x)+\psi_2'(x)$ is uniformly bounded, so \[|S_l (\psi_1'(x)+\psi_2'(x))(x)|\leq l\norm{\psi_1'}_{L^{\infty}}+l\norm{\psi_2'}_{L^{\infty}},\qquad i=1,2.\]

On $F_{n,i}$, the orbit segment $\{T^k x\}_{k=0}^{h_{i}^{(n)}-1}$ avoids intervals of length $2\delta/q_{i}^{(n)}$ around the singularities, and $\{T^k I_i^{(n)}\}_{k=0}^{h_{i}^{(n)}-1}$ are disjoint intervals of $\T$, so the smallest gap between consecutive points when ordered in $(0,1)$ is $\lambda_{i}^{(n)}$. The worst-case sum is therefore comparable to the harmonic series
\[
\sum_{k=0}^{l-1}\left( \frac{\max c_y^{\pm}}{\|T^k x\|}+\max |b_y^{\pm}|(1-\log \|T^k x\|)\right)\lesssim^{\varphi} \sum_{j=0}^{l-1} \frac{q_{i}^{(n)}}{\delta+j} \lesssim^\delta q_{i}^{(n)}{(\log l+ 1)}.
\]
Thus there exists a constant $C_1 = C_1(\delta, \varphi, T) > 0$ such that
\begin{equation*}
{|S_{l} (\varphi')(x)| \leq C_1 q_{i}^{(n)} \log h_{i}^{(n)} \qquad \forall\, x \in F_{n,i},\qquad 0\leq l\leq h_{i}^{(n)}.}
\end{equation*}
\end{proof}

Now, fix $\theta,\eps>0$ and assume $T$ satisfies \eqref{eq:IET1}--\eqref{eq:IET2}. Denote by $C_\theta$ and by $\{r_m\}$ the corresponding constant and sequence. Let \[h_{r_m}=h^{(r_m)}_{j^*},\qquad q_{r_m}=q^{(r_m)}_{j^*},\qquad \text{ and } \qquad F_{r_m}=F_{r_m,j^*}.\] 

For each $m$, define
\[
F'_m=F'_m(\delta) := F_{r_m}\cap T^{h_{r_m}}F_{r_m},
\qquad
E_m=E_m(\delta):= \bigcup_{k=0}^{h_{r_m}-1} T^k F_{r_m}.
\]
The set $E_m$ is the union of $h_{r_m}$ disjoint translates of $F_{r_m}$ under $T$, forming the bulk of the Rokhlin tower over $F_{r_m}$. The intersection $F'_m = F_{r_m} \cap T^{h_{r_m}} F_{r_m}$ serves as the base of a sub-tower whose first return map is well approximated by a translation. Lemma \ref{lemma:def_M_IET} below shows that $S_{h_{r_m}}(\phi)$ is nearly constant on $E_m$, with oscillation at most logarithmic in $h_{r_m}$.

\begin{lemma} \label{lemma:def_M_IET}
There exists some $C_2=C_2(\delta, \varphi, T)>0$ and $M=M(T)\geq 1$ such that for any $m\geq M$,
\begin{equation}
\label{eq:BSV_IET}
\left|S_{h_{r_m}} (\varphi)(x) - S_{h_{r_m}} (\varphi)(y)\right| \leq C_2 \log h_{r_m} \qquad \forall\, x,y \in E_m.
\end{equation} 
\end{lemma}

\begin{proof}
By \eqref{eq:IET2},
\[
\left.
\begin{aligned}
{\lambda_{j^*}^{(r_m)}} \geq 3{\lambda^{(r_m)}}/4,
\\
T^{r_m}I_{j^*}^{(r_m)}\subset I^{(r_m)}
\end{aligned}
\right\}
\]
\[\implies \lambda\left(F_{r_m}\cap T^{h_{r_m}}F_{r_m}\right)\geq \lambda^{(r_m)}/2-2\delta/q_{r_m}\geq \left(1/2-{2\delta}\right)\lambda^{(r_m)}>0.
\]
Therefore $F_m'\neq\emptyset$.

$E_m$ is formed of $h_{r_m}$ disjoint intervals, each corresponding to an iterate under $T$ of $F_{r_m}$. If $x, y$ lie in the same connected component of $E_{m}$,
\[
\left|S_{h_{r_m}} (\varphi)(x) - S_{h_{r_m}} (\varphi)(y)\right| \leq |x-y| \cdot \sup_{\xi \in [x,y]} |S_{h_{r_m}}( \varphi')(\xi)|,
\]
so, provided $m$ large enough, \eqref{eq:BSD_IET} gives
\[
\sup_{x,y \in J} \left|S_{h_{r_m}} (\varphi)(x) - S_{h_{r_m}} (\varphi)(y)\right| \leq C_1 q_{r_m} \log h_{r_m} \lambda(F_{r_m} ) \leq C_1\log h_{r_m}.
\]

Consider the first $q_{r_m}$ iterates of the orbit of $x_0\in F'_m $. Each of them lie in a different interval of $E_m$ and $x_0,T^{q_{r_m}}x_0\in F_{r_m}$. Hence
\[\begin{aligned}
\sup_{0< l < h_{r_m}} \left|S_{h_{r_m}} (\varphi)(T^lx_0) - S_{h_{r_m}} (\varphi)(x_0)\right| &\leq \sup_{0< l < h_{r_m}} \left|S_{l} (\varphi)(x_0) - S_{l} (\varphi)(T^{h_{r_m}}x_0)\right| \\&\leq C_1 \log h_{r_m}.
\end{aligned}
\]

To pass from a single connected component to all of $E_m$, let $x, y \in E_m$ be arbitrary.  If they lie in the same component, the bound $C_1 \log h_{r_m}$ applies directly. If not, pick any $x_0 \in F'_m$ and $l_1, l_2$ such that $T^{l_1} x_0$ and $T^{l_2} x_0$ lie in the same components as $x$ and $y$ respectively. Then the triangular inequality gives the result for $C_2 = 4C_1$.
\end{proof}

\subsection{Choice of $g_m$ and $t_m$}
\label{sec:choicegm}

For $I_{i}^{(n)}=(a,b)\subset(0,1)$, consider the base interval
\be
F_{n,i}=\left(a+\delta/q_i^{(n)},b-\delta/q_i^{(n)}\right).
\label{eq:def_In_IET}
\ee

\begin{lemma}
\label{lemma:goog_In_IET}
There exists $N=N(T)\geq 1$ such that for each $n\geq N$, if $q_i^{(n)}\leq h_i^{(n)}d$, then the pair $(I=F_{n,i},h=h_i^{(n)})$ satisfy \eqref{eq:conditions_for_g} with respect to $T$.
\end{lemma}
\begin{proof}
It is clear by their definitions that $T^k$ is a translation on $I$ for any $0\leq k<h$, and \[\lambda(I)=\lambda(F_{n,i})= \frac{{1-2\delta}}{q_i^{(n)}}\geq \frac{1}{2h_i^{(n)}d}=:C^{-1}h^{-1}.\]
\end{proof}

For the rest of the section, let $T$ satisfy \eqref{eq:IET1}--\eqref{eq:IET2}. 

The next lemma shows that, given the scale $\log h_{r_m}$ imposed by \eqref{eq:IET2}, we can always find an induction level $n$ and an index $i$ such that the pair $(F_{n,i}, h^{(n)}_i)$ satisfies \eqref{eq:conditions_for_g} with return time comparable to $\log h_{r_m}$, and such that $h^{(n)}_i$ is not too small relative to $q^{(n)}_i$. 

\begin{lemma} \label{lemma:hm_vs_hn}
There exists $\hat C_\theta=\hat C_\theta(T)>0$ and $M=M(\delta, \varphi,T)\geq 1$ such that for all $m\geq M$,
$\exists n=n(m,\delta,\varphi,T)\in\N$ and $i\in\{1,\dots,d\}$ satisfying 
\be
 \hat C_\theta\left(h_i^{(n)}\right)^{\frac{1}{(1+\theta)^2}}\leq \frac{2C_2 }{\gamma}\log h_{r_m}\leq  h_i^{(n)},
\label{eq:hm_vs_hn}
\ee
for $C_2$ as in Lemma \ref{lemma:def_M_IET}, and
\[
\frac{q_i^{(n)}}{d}\leq h_i^{(n)}\leq q_i^{(n)}.
\]
\end{lemma}
\begin{proof}
Given ${m}$, for each induction step $k$, let $S_k := \left\{j \in \{1,\dots,d\} : h_j^{(k)}\lambda_j^{(k)} \geq \frac{1}{d}\right\}$ be the set of dominant components. Since $\sum_{j=1}^d h_j^{(k)}\lambda_j^{(k)} = 1$, $S_k \neq \emptyset$ for all $k$. We define 
\[n:=\min\left\{k:\min_{j\in S_k} h_j^{(k)}\geq \frac{2C_2}{\gamma}\log h_{r_m}\right\}.\] 
Clearly $n\to\infty$ as $m\to\infty$. Pick $m$ large enough so that $n\geq 2$.

By definition of $n$, at step $n$ we can choose an index $i \in S_n$ such that 
\[
\frac{1}{d}\leq h_i^{(n)}\lambda_i^{(n)}\leq 1.
\]
Fixing this $i$, we automatically have $h_i^{(n)} \geq \min_{j \in S_n} h_j^{(n)} \geq \frac{2C_2}{\gamma}\log h_{r_m}$. 

Furthermore, by the minimality of $n$, the condition must fail at step $n-1$, meaning there exists an index $j \in S_{n-1}$ such that $h_{j}^{(n-1)} < \frac{2C_2}{\gamma}\log h_{r_m}$. Since $j \in S_{n-1}$, it satisfies $h_{j}^{(n-1)} \geq \frac{1}{d\lambda_{j}^{(n-1)}} \geq \frac{1}{d\max_j\lambda_j^{(n-1)}}$. Combined with \eqref{eq:IET1}, we obtain
\[\begin{aligned}
h_i^{(n)}&\geq  \frac{2C_2}{\gamma}\log h_{r_m}  \geq h_{j}^{(n-1)}\geq \frac{1}{d\max_j\lambda_j^{(n-1)}}\geq \frac{1}{d C_\theta\left(\min_j \lambda_j^{(n-1)}\right)^{\frac{1}{1+\theta}}}\\
&\geq \frac{1}{d C_\theta\left(\max_j \lambda_j^{(n)}\right)^{\frac{1}{1+\theta}}}
\geq \frac{1}{d C_\theta^{\frac{2+\theta}{1+\theta}}\left(\min_j \lambda_j^{(n)}\right)^{\frac{1}{(1+\theta)^2}}}\geq \frac{1}{d C_\theta^{\frac{2+\theta}{1+\theta}} \left(\lambda_i^{(n)}\right)^{\frac{1}{(1+\theta)^2}}}\\
&\geq \frac{1}{d C_\theta^{\frac{2+\theta}{1+\theta}}}\left(h_i^{(n)}\right)^{\frac{1}{(1+\theta)^2}}=:\hat C_\theta \left(h_i^{(n)}\right)^{\frac{1}{(1+\theta)^2}},
\end{aligned}
\]
where the final inequality follows from the total area constraint, implying $\lambda_i^{(n)} \leq 1/h_i^{(n)}$. This completes the proof.
\end{proof}

Now, given $m$, let $n$ and $i$ as in Lemma \ref{lemma:hm_vs_hn}, and let $q_n:=q_{i}^{(n)},$ $ h_n :=h_{i}^{(n)}$ so that 
\[ 
\hat C_\theta \left(h_n\right)^{\frac{1}{(1+\theta)^2}}\leq \frac{2C_2 }{\gamma}\log h_{r_m}\leq h_n, \qquad \frac{q_n}{d}\leq h_n,
\]
and let $I_n:=F_{n,i}$.

\begin{lemma}
\label{lemma:def_tn_IET}
There exists $M = M(\delta, \varphi, T,\theta) \geq 1$ such that for all $m \geq M$ and its corresponding $n$, the function $g_m \in C^1(A_m)$ given by Corollary \ref{cor:testfunction} associated with $(I_n, h_n, T)$ is well-defined and supported on $A_m := I_n \times [0, \gamma h_n/4]$. Moreover, defining
\begin{equation}
\label{eq:def_tn_IET}
t_m := \sup_{F_{r_m}} S_{h_{r_m}}(\varphi) + \frac{\gamma h_n}{4},
\end{equation}
it holds that
\begin{equation}
\label{eq:tn_inequality}
t_m \leq \inf_{F_{r_m}} S_{h_{r_m}}(\varphi) + \frac{3\gamma h_n}{4},
\end{equation}
and there exists $C_3 = C_3(\delta, \varphi, T) > 0$ such that
\begin{equation}
\label{eq:tn_qn_IET}
C_3^{-1} h_{r_m} \leq t_m \leq C_3 h_{r_m}.
\end{equation}
\end{lemma}
The time $t_m$ is chosen so that the flow image $\varphi_{t_m}(B_m)$ of the bulk of $A_m$ lands just above $A_m$, missing it entirely. This is why $t_m$ is defined as the supremum of $S_{h_{r_m}}(\varphi)$ on $F_{r_m}$, shifted by $\gamma h_n/4$.

\begin{proof}
The existence of $g_m$ for $m$ large enough is a direct application of Lemma \ref{lemma:goog_In_IET}, and it follows from \eqref{eq:BSV_IET} and \eqref{eq:hm_vs_hn} that
\[
t_m - \inf_{F_{r_m}} S_{h_{r_m}} (\varphi)-3{\gamma h_n}/{4}\leq C_2\log h_{r_m}-\gamma h_n/2\leq 0.
\]

To prove \eqref{eq:tn_qn_IET}, first note that it is enough to prove it for $\sup_{F_{r_m}} S_{h_{r_m}}(\varphi)$ instead of $t_m$, since \eqref{eq:hm_vs_hn} implies that the quotient 
\[
\frac{h_n}{h_{r_m}}\lesssim\frac{(\log h_{r_m})^{(1+\theta)^2}}{h_{r_m}}\to 0
\] 
as $m\to \infty$.

Now, by definition of $\gamma$, 
\[
S_{h_{r_m}}(\varphi)(x) \geq \gamma h_{r_m} , \qquad \forall x\in \T,
\]
so it is only left to show the upper bound.

To show the upper bound we will use the same argument as in the proof of Lemma \ref{lemma:BSD_IET}.
The term $\psi(x)$ is uniformly bounded, so 
\[
S_{h_{r_m}}(\psi)(x)\lesssim h_{r_m},\qquad \forall x \in \T.
\] 
On $F_{r_m}$, the orbit segment $\{T^k x\}_{k=0}^{h_{r_m}-1}$ avoids intervals of length $2\delta/q_{r_m}$ around the singularities, and the smallest gap between consecutive points is $q_{r_m}^{-1}$. The worst-case sum is therefore comparable to the series
\[
-\sum_{k=0}^{h_{r_m}-1} \log{\|T^k x\|}\leq \sum_{j=0}^{h_{r_m}-1} \log \frac{q_{r_m}}{\delta+j} \leq \log \frac{q_{r_m}}{\delta}+\int_0^{h_{r_m}}\log\frac{q_{r_m}}{\delta+x}dx\lesssim^\delta h_{r_m},
\]
for any $x\in F_{r_m}$, where we have used that \eqref{eq:IET2} implies $h_{r_m}\geq q_{r_m}/2$.
\end{proof}

\begin{remark}
    \label{rk:dep} Note that, for $\varepsilon$ and $\delta$ fixed, the sequences $(r_m)$ and $(n(m))$ do not depend on the choice of $\theta$, and hence neither do $(t_m)$, $(A_m)$, and $(g_m)$.
\end{remark}

\subsection{Lower bounds on the self-correlation of $g_n$ at time $t_n$}  \label{sec:lowerbounds}

In this section we finish the proof of Theorem \ref{thm:IETs}. 

We now show that the choice of $t_m$ forces $\phi_{t_m}(A_m)$ to be nearly disjoint from $A_m$, which makes the self-correlation $\langle g_m \circ \phi_{t_m}, g_m \rangle$ substantially smaller than $\mu(g_m)^2$, giving the desired lower bound.

\begin{lemma}
There exists $C=C(\delta,\varphi,T,\theta)>0$ and $M=M(\delta, \varphi,T,\theta)\geq 1$ such that for all $m\geq M$, $g_m$ and $t_m$ from Lemma \ref{lemma:def_tn_IET} satisfy
\[
|\langle g_m\circ\phi_{t_m},g_m\rangle-\mu(g_m)^2|\geq \frac{\gamma^2}{32 d^2},\qquad \|g_m\|_\infty=1,
\]
and, letting $g_m=g_m(x,s)$,
\[
\|\partial_xg_m\|_{\infty}\leq C (\log t_m)^{{(1+\theta)^2}}\log\log t_m,
\quad \|\partial_sg_m\|_{\infty}\leq C (\log t_m)^{-1}\log\log t_m.
\]
\label{lemma:final_thm1}
\end{lemma}
\begin{proof}
Define $I'_n=\left(a+\delta/q_n+1/q_n^2,b-\delta/q_n-1/q_n^2\right)$ and recall $A_m=I_n\times[0,\gamma h_n/4].$ This choice of $I'_n$ is made so that $T^{h_{r_m}}I_n'\subset I_n$, for $m$ large enough, while its length is asymptotically close to $\lambda(I_n)=\lambda(I_n')+2/q_n^2$. Indeed, for $m$ large enough, \eqref{eq:IET2} and \eqref{eq:hm_vs_hn} imply
\[
\text{dist}(x,T^{h_{r_m}}x)\leq \frac{\lambda^{(r_m)}}4\leq \frac{1}{3q_{r_m}}\leq \frac{1}{q_n^2}, \qquad x\in F_{r_m}.
\] 

Set
\begin{align*}\begin{cases}
J_m &:= I'_n\cap E_m, \\
B_m &:= J_m\times\left[0,\frac{\gamma h_n}{4}\right]\subset A_m,\\
B'_m &:= \left\{(x,t):x\in T^{h_{r_m}}J_m, t\in\left(\frac{\gamma h_n}{4},S_{q_n}(\varphi)(x)\right)\right\}\subset A_m^c.
\end{cases}
\end{align*}
By definition of $E_m$ and \eqref{eq:IET2}, we have
\[
\mu(B_m)=\lambda(I_n'\cap E_m)\frac{\gamma h_n}{4}\geq\left(\mu(I_n)-\frac{2}{q_n^2}\right)\frac{\gamma h_n}{4}-\int_{\left(\mathcal{T}_{j^*}^{(r_m)}\right)^c}\varphi\geq(1-\delta)\mu(A_m),
\]
for $\varepsilon(\delta,\varphi)>0$ small enough and $n$ large enough. Here we have used that $h_n\leq q_n$ and that $\varphi \in L^1(\T)$ and $\lambda\left(\mathcal{T}_{j^*}^{(r_m)}\right)\geq 1-\eps$.

The rectangle $A_m$ has base length $\frac{1-2\delta}{q_n}$ and height $\gamma h_n/4$, so $\mu(A_m)\geq\frac{(1-2\delta)\gamma}{4d}. $

The conditions \eqref{eq:def_tn_IET}--\eqref{eq:tn_inequality} on $t_m$ forces $\phi_{t_m}(B_m)\cap A_n\subset B'_m\cap A_m=\varnothing$, hence
\[
\mu(\phi_{t_m}(A_m)\cap A_m) \leq \mu(A_m\setminus B_m) \leq \delta\,\mu(A_m).
\]

Recall that $g_m$ is defined using Corollary \ref{cor:testfunction}, and therefore satisfies the estimates \eqref{eq:g}. Consequently,
\[
\left|\mu(g_m)^2-\langle g_m,\, g_m\circ\phi_{t_m}\rangle\right|
\geq \left(\mu\left(A_m\right)\left(1-\frac{4}{\log h_n}\right)\right)^2-\mu\left(A_m\setminus B_m\right)
\]
\[
\geq\mu\left(A_m\right)^2-\delta\mu\left(A_m\right)-O\left({\frac{1}{\log h_{n}}}\right) \geq \frac{(1-2\delta)^2\gamma^2}{16d^2}-\delta\frac{(1-2\delta)\gamma}{4d}-O\left({\frac{1}{\log h_{n}}}\right).
\]
So one can fix $\delta(\varphi)>0$ small enough so that, for $m$ large enough such that the last term is $\leq \frac{\gamma^2}{64d^2}$, it holds
\[
\left|\mu(g_m)^2-\langle g_m,\, g_m\circ\phi_{t_m}\rangle\right|\geq \frac{\gamma^2}{32d^2}.
\]

Finally, using that $g_m$ satisfies the estimates \eqref{eq:g} and $t_m$ relates to $h_{n}$ by \eqref{eq:hm_vs_hn} and \eqref{eq:tn_qn_IET}, we get
\[
\|\partial_xg\|_{\infty}\lesssim h_{n} \log h_{n}\lesssim (\log t_m)^{{(1+\theta)^2}}\log\log t_m,
\]
and 
\[
\|\partial_sg\|_{\infty}\lesssim h_{n}^{-1} \log h_{n}\lesssim (\log t_m)^{-1}\log\log t_m,
\]
where the constants depend on $\delta$, $\varphi$, $T$ and $\theta$. 

As a consequence of \eqref{eq:g}, $\|g_m\|_\infty=1$. This completes the proof.
\end{proof}

\begin{proof}[Proof of Theorem \ref{thm:IETs}]
Choose $\theta > 0$ small enough that $((1+\theta)^2 - 1) < \nu$ and $\delta$ such that \[\frac{(1-2\delta)^2}{16}-\frac{\delta(1-2\delta)d}{4\gamma}\geq\frac{3}{64}.\] Then pick $\eps$ so that $\int_{\left(\mathcal{T}_{j^*}^{(r_m)}\right)^c}\varphi\leq \frac{\delta (1-2\delta)\gamma}{8d}$. 

We proof the result holds for all IETs satisfying \eqref{eq:IET1}--\eqref{eq:IET2}, which is a full measure condition by Proposition \ref{prop:full_meas_IET}.
For such IETs and our choice of $(\theta,\eps)$, Lemma \ref{lemma:final_thm1} produces a constant $C=\gamma ^2/32d^2$, a sequence $t_m \to \infty$, rectangles $A_m$, functions $g_m \in C^1_c(A_m)$, and $M>0$ such that, for every $m\geq M$
\[
|\text{Cor}_{t_m}(g_m,g_m)|=|\langle g_m\circ\phi_{t_m},g_m\rangle-\mu(g_m)^2|\geq C,
\]
and $\|g_m\|_\infty=1$, $\|\partial_xg\|_{\infty}=o((\log t_m)^{1+\nu})$, $\|\partial_sg\|_{\infty}=o((\log t_m)^{-1+\nu})$.

Note that, for $\varphi$ and $T$ fixed, so are $\eps$ and $\delta$, so we recall from Remark \ref{rk:dep} that the sequences $(t_m),(A_m),(g_m)$ are independent of the choice of $\nu$. However, $M$ is not. Without loss of generality, we can assume $M(\nu=1)=1$, so that \eqref{eq:IET_cor_lower} holds for any $m$.

If $m\geq M(\nu)$, by definition of $A_m$ (the support of $g_m$) and using that $h_n$ and $t_m$ are related via \eqref{eq:hm_vs_hn} and \eqref{eq:tn_qn_IET}, for every $(x, s)\in A_m$ 
\[
\mathrm{dist}(T^k x, D_0)> \frac{\delta}{q_n}\geq\hat\delta\,(\log t_{m})^{-(1+\theta)^2}>(\log t_{m})^{-1-\nu},
\]
for $0\leq k< h_n$, and 
\[
N(x,s)\leq \tfrac{1}{4}h_n\leq\frac{1}{\hat\delta}(\log t_m)^{(1+\theta)^2}<(\log t_{m})^{1+\nu},
\]
for $m$ large enough, where $\hat \delta=\hat \delta(\delta,\varphi,T,\theta) > 0$. This concludes the proof.
\end{proof}

\section{Some convexity results} \label{sec:convexity}
The key mechanism behind Theorems \ref{thm:log_decay} and \ref{thm:sum_cor} is that the strict concavity of $-\varphi$ forces the flow image of a thin horizontal strip to be asymmetrically distributed above and below a given height. The three lemmas in this section make this precise. 

Lemma \ref{lemma:convex1} gives the quantitative asymmetry for a single concave function; Lemma \ref{lemma:convex2} integrates this to a measure estimate; and Lemma \ref{lemma:convex3} extends both to the case where the strip only partially overlaps the domain or there is no quantitative control on the derivatives.

\begin{lemma}\label{lemma:convex1}
Let $\eps>0$. Let $f:[\alpha,\beta]\to[-\eps,2\eps]$ be a $C^2$ bijective function. 
Moreover, assume $\exists M,C>0, k \geq 1$ such that 
\[
0< f'(x)\leq \sqrt{k}\,CM ,\qquad \text{ and }\qquad {C}^{-1}kM^2 \leq -f''(x),
\]
on $[\alpha,\beta]$ (see Figure \ref{fig:convexity}). Then $\exists \eps^*=\eps^*(C)>0$ such that
\be
(f^{-1}(s+\eps)-f^{-1}(s))-(f^{-1}(s)-f^{-1}(s-\eps))\geq\frac{\eps^3}{\sqrt{k}M},
\label{eq:diff_cros}
\ee
provided $\eps<\eps^*$, fo any $s\in (0,\eps)$.
\end{lemma}

\begin{figure}[h!]
\centering
\includegraphics[width=0.5\linewidth]{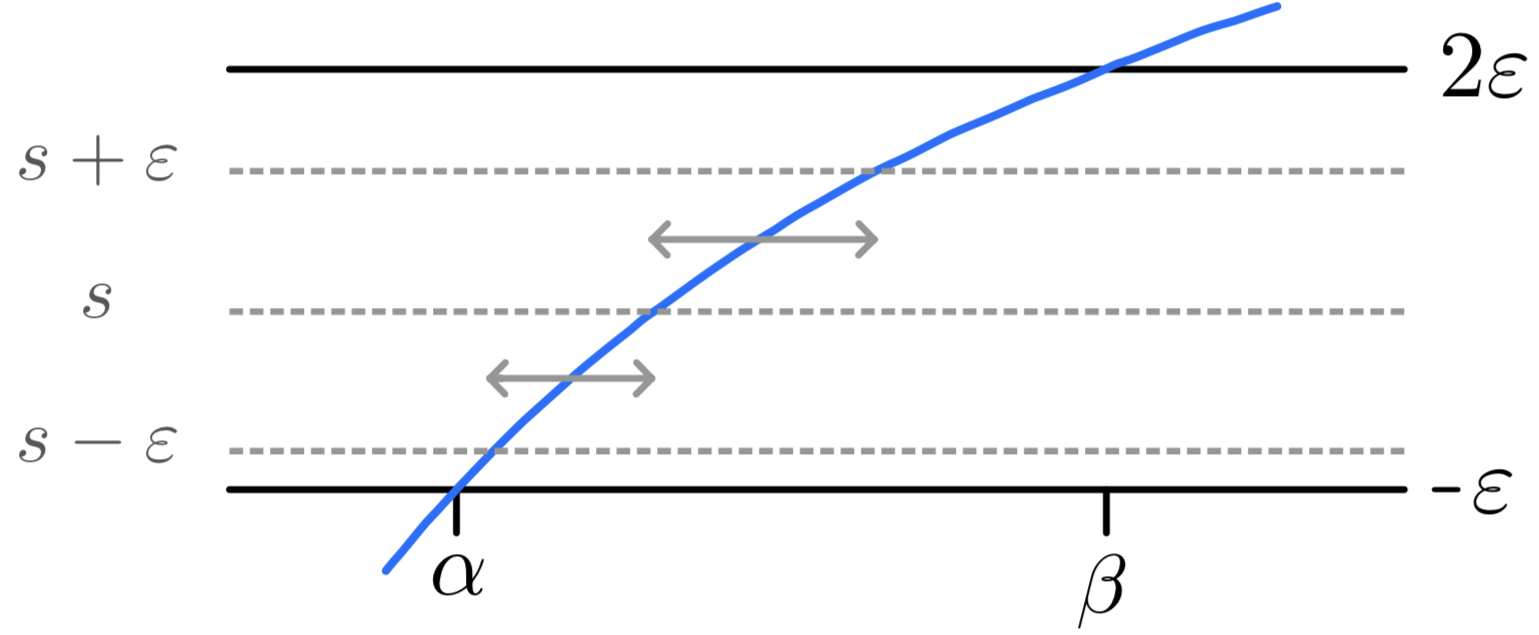}
\caption{Schematics of the setting of Lemma \ref{lemma:convex1}.}
\label{fig:convexity}
\end{figure}

\begin{remark}
The parameter $k \geq 1$ is introduced so that Lemma \ref{lemma:convex1} applies in the proofs of Theorem \ref{thm:log_decay} (where $k=1$) and
Theorem \ref{thm:sum_cor} (where $k$ ranges).
\end{remark}

\begin{proof}
We will first prove the following: for every $s\in (0,\eps)$, $\exists \theta_1(s),\theta_2(s)\in(\alpha,\beta)$ such that
\be
\begin{aligned}
&(f^{-1}(s+\eps)-f^{-1}(s))-(f^{-1}(s)-f^{-1}(s-\eps))=-\sum_{i=1}^2\frac{f''(f^{-1}(\theta_i(s)))}{2f'(f^{-1}(\theta_i(s)))^3}\eps^2.
\label{eq:taylor_crossings}
\end{aligned}
\ee

Let $g=f^{-1}$. Then $g$ is $C^2$ on $[-\eps,2\eps]$. Fix $s\in(0,\eps)$.

We apply the Taylor expansion of $g$ at $s$. For $\eps>0$, we have
\[
g(s+(-1)^i\eps) = g(s)+(-1)^i g'(s)\eps +\frac{g''(\theta_i(s))}{2}\eps^2,
\]
for $i=1,2$, and some $\theta_1(s)\in(s-\eps,s)$ and $\theta_2(s)\in(s,s+\eps)$.

Subtracting gives
\[
(g(s+\eps)-g(s))-(g(s)-g(s-\eps))= \frac{g''(\theta_1(s))}{2}\eps^2+\frac{g''(\theta_2(s))}{2}\eps^2.
\]

It remains to express the derivatives of $g$ in terms of those of $f$. Using the identities for derivatives of the inverse function, we have
\[
g'(s)=\frac{1}{f'(f^{-1}(s))}, \qquad g''(s)= -\frac{f''(f^{-1}(s))}{f'(f^{-1}(s))^3}.
\]
Substituting these identities yields \eqref{eq:taylor_crossings}.

To prove \eqref{eq:diff_cros}, notice that, by hypothesis, $\frac{-f''(x)}{f'(x)^3}\geq \frac{1}{\sqrt{k}MC^4},$ on $[\alpha,\beta]$.
Therefore
\[
(g(s+\eps)-g(s))-(g(s)-g(s-\eps))\geq \frac{2}{\sqrt{k}MC^4}\eps^2 \geq \frac{\eps^3}{\sqrt{k}M},
\]
for $\eps<\eps^*$ small enough. 
\end{proof}

Lemma \ref{lemma:convex1} controls pointwise asymmetry. Integrating over $s \in (0,\varepsilon)$ converts this into the following area estimate, which is the form we will actually apply.
\begin{lemma}
Let $f$ be a $C^2$ function satisfying all the assumptions of Lemma \ref{lemma:convex1} for some $0<\eps<\eps^*$. Let 
\[ 
E=\{(x,y):\,f(x)\leq y\leq f(x)+\eps,\,x\in[\alpha,\beta]\},\]\[ D^-=[\alpha,\beta]\times[0,\eps],\qquad D^+=[\alpha,\beta]\times[\eps,2\eps].
\]
Then
\be
    \mu(E\cap D^+)-\mu(E\cap D^-)\geq\frac{\eps^4}{\sqrt{k}M}.
\ee
\label{lemma:convex2}
\end{lemma}
\begin{proof}
Observe that, with $g=f^{-1}$, the difference becomes
\[
\mu(E\cap D^+)- \mu(E\cap D^-) = \int_{0}^{\eps} \Bigl[(g(s+\eps)-g(s))-(g(s)-g(s-\eps))\Bigr]\,ds.
\]

By Lemma \ref{lemma:convex1}, for every $s\in(0,\eps)$, the term on the integral is lower bounded by $\frac{\eps^3}{\sqrt{k}M}$. Integrating yields the result.
\end{proof}
Here $E$ is the region between the graph of $f$ and a horizontal translate by $\varepsilon$, while $D^-$ and $D^+$ are the lower and upper horizontal slabs of height $\varepsilon$. The lemma asserts that $E$ intersects the upper slab more than the lower one, by an amount proportional to the concavity of $f$.

Finally, we need one last result. Although the conclusion is weaker, we will use Lemma \ref{lemma:convex3} in the cases on which the function $f$ need not map onto all of $[-\varepsilon, 2\varepsilon]$, or where the concavity bound is not required to be quantitative.
\begin{lemma}
\label{lemma:convex3}
Let $\eps>0$. Let $f:[\alpha,\beta]\to[-\eps,2\eps]$ be a $C^2$ bijective function such that $f'(x)\neq 0$ and $0 > f''(x)$ on $[\alpha,\beta]$. Let $E$ and $D^\pm$ be as in Lemma \ref{lemma:convex2}. Then $\mu(E\cap D^+)-\mu(E\cap D^-)\geq 0$. 

Moreover, if $f:[\alpha,\beta]\to[-\hat\eps,2\eps]$ satisfies all the conditions above for some $-2\eps<\hat{\eps}<\eps$, then $\mu(E\cap D^+)-\mu(E\cap D^-)\geq 0$.
\begin{proof}
If $f'>0$, apply Lemmas \ref{lemma:convex1} and \ref{lemma:convex2} with $C=\max\{\max f',(\min |f''|)^{-1}\}$ and $k=M=1$. If $f'<0$, same with $\hat f(x)=f(\beta+\alpha-x)$.

To prove the last statement, first note that if $\hat\eps\leq-\eps$, then $\mu(E\cap D^-)=0$ and the result is trivial. Otherwise, let $g=f^{-1}$, defined on $[-\hat\eps,2\eps]$. Slicing horizontally,
\[
\mu(E\cap D^+)-\mu(E\cap D^-)=\int_0^\eps[g(s+\eps)-g(s)]\,ds - \int_0^{\hat\eps}[g(\eps-s)-g(-s)]\,du.
\]
so that, substituting $t=\eps-s$ in the second integral
\[ 
\begin{aligned}
\mu(E\cap D^+)-\mu(E\cap D^-)&=\int_0^{\eps-\hat\eps}[g(s+\eps)-g(s)]\,ds\\&+\int_{\eps-\hat\eps}^{\eps}\bigl[(g(s+\eps)-g(s))-(g(s)-g(s-\eps))\bigr]\,ds.
\end{aligned}
\]
The first integral is non-negative since $g$ is increasing. For the second, every $s\in(\eps-\hat\eps,\eps)$ satisfies $s-\eps\geq-\hat\eps$, so all three arguments of $g$ lie in $[-\hat\eps,2\eps]$, and the integrand is non-negative by concavity of $g$.

\end{proof}
\end{lemma}

\section{Proof of Theorem \ref{thm:log_decay}}\label{sec:logdecay}

The argument has four steps. In Section \ref{sec:7.1} we introduce the Diophantine condition \eqref{H} on $\alpha$ and reduce Theorem \ref{thm:log_decay} to a more quantitative statement, Theorem \ref{thm:log_decay2}. In Section \ref{sec:7.2} we establish uniform bounds on the Birkhoff sums of the derivatives of $\varphi$ along orbits of $R_\alpha$, which hold for any irrational $\alpha$. In Section \ref{sec:7.3} we identify a sequence of times at which the flow image of $A$ is geometrically positioned to produce a large correlation. Finally, in Section \ref{sec:7.4} we carry out the correlation estimate by partitioning $A$ into thin horizontal strips and applying the convexity results of Section \ref{sec:convexity} strip by strip.

\subsection{Notation and first steps} \label{sec:7.1}

For $\alpha\in (0,1)$, let $\alpha=[0;a_1,a_2,\dots]$ be its continued fraction expansion and let $p_n/q_n$ be the $n$th convergent. 

Letting $\norm{x}=\text{dist}(x,\Z)$, recall the following basic continued-fraction facts:
\[
\frac{1}{2q_{n+1}}\leq \frac{1}{q_{n+1}+q_n}<\|q_n\alpha\|<\frac{1}{q_{n+1}},
\qquad q_{n+1}=a_{n+1}q_n+q_{n-1},
\]
for every $n\in \N$, and $\|q_{n-1}\alpha\|\leq \|k\alpha\|, $ for every $1 \leq k < q_n$.

It is well known by the three gap theorem \cite{alloucheshallit2003} that the set $ \{-k \alpha \}_{k=0}^{q_n-1}$ partitions $ \T$ into $q_n$ intervals of lengths $ \norm{q_{n-1} \alpha}$ and $ \norm{q_{n-1} \alpha}+ \norm{q_n \alpha}$.

We say $\alpha\in(0,1)$ satisfies the condition \eqref{H} if it admits an increasing sequence $\{r_n\}$ such that:
\begin{align}
q_{r_n+1}\geq q_{r_n}\log q_{r_n}\log \log q_{r_n},&\quad \forall n. \tag{H}\label{H}
\end{align}

Condition \eqref{H} is calibrated to produce the $(\log t)^{-1}$ lower bound for the correlations: it is strong enough to control the error terms in the Birkhoff sum estimates, yet weak enough to hold for Lebesgue-almost every $\alpha$ by Khinchin's theorem, as the next proposition shows.

\begin{proposition}\label{prop:full_meas}
Lebesgue almost every $\alpha \in (0,1)$ satisfy \eqref{H}.
\end{proposition}

\begin{proof}
Let $\eta(q) = ({2q \log q\log\log q})^{-1}$. Then $q\eta(q) = ({2 \log q\log\log q})^{-1}$ is decreasing and
\[
\sum_{q=2}^\infty \eta(q) = \sum_{q=2}^\infty \frac{1}{2q \log q\log\log q} = \infty.
\]
By Khinchin's theorem \cite{khinchin1964}, for Lebesgue almost every $\alpha$, there exist infinitely many coprime pairs $(p,q)$ such that $ \left|\alpha - \frac{p}{q}\right| < \frac{1}{2q^2 \log q\log \log q}.$

For each $q$, there exists an $n$ such that $ q_n \leq q < q_{n+1}, $ so the best approximation property of continued fractions implies
\[
\frac{1}{2 q_{n+1}} \leq\left|q_n\alpha - p_n\right| \leq \left|q\alpha - p\right|< \frac{1}{2q \log q\log \log q} \leq \frac{1}{2q_n \log q_n\log \log q_n}.
\]
Hence, for infinitely many $n$ for which a Khinchin-good rational $p/q$ lies in $(q_n,$ $q_{n+1})$, we obtain $q_{n+1} \geq \, q_n \log q_n\log \log q_n.$
\end{proof}

It follows from Proposition \ref{prop:full_meas} and \[ |\mathrm{Cor}_t(\chi_A, \chi_A)|+| \mathrm{Cor}_{t-\eps}(\chi_A, \chi_A)|\geq\left|\mu(A\cap\phi_{t}A)-\mu(A\cap\phi_{t-\eps}A)\right|,\] that Theorem \ref{thm:log_decay} is a corollary of the following result.

\begin{theorem} \label{thm:log_decay2}
If $\alpha\in(0,1)$ satisfies \eqref{H}, then for any $A\subset X$ box, there exists $\eps>0$ and an increasing sequence $t_n$ such that $\forall n\geq 1$
\be
\label{eq:thm72}
\left|\mu(A\cap\phi_{t_n}A)-\mu(\phi_{\eps}A\cap\phi_{t_n}A)\right|\geq \frac{\eps^3}{\log t_n}.
\ee
\end{theorem}

\subsection{Bounds on the Birkhoff sums} \label{sec:7.2}

In this section, we present some estimates on the Birkhoff sums of $S_l(\varphi)(x)$ and its derivatives. Bounds of this type on Birkhoff sums of the derivative of the roof function along orbits of an irrational rotation are classical, our main tool being the Denjoy--Koksma inequality \cite{herman1979}; see, for instance, \cite{Sinai1992,kochergin03}.

Recall that for rotations we consider a roof function $\varphi$ as in \eqref{eq:phi_basic}--\eqref{eq:phi_Dbasic}. In particular, the set of its singularities is $\{0\}$ (and 1), and its derivatives for $x\in (0,1)$ are

\be \label{eq:def_Dbasic}
\begin{aligned}
\varphi'(x) &= - \frac{2c_0}{x} + \frac{c_0}{1-x}+ \psi'(x)\\&= - \frac{2c_0}{x} + \frac{c_0}{1-x}-2b_0+b_0+2b_0|\log x|-b_0|\log(1-x)|+\psi_1'(x)+\psi_2'(x),
\end{aligned}
\ee
\be \label{eq:def_DDbasic}
\begin{aligned}
\varphi''(x) &= \frac{2c_0}{x^2} + \frac{c_0}{(1-x)^2}+ \psi''(x)
=\frac{2c_0}{x^2} + \frac{c_0}{(1-x)^2}- \frac{2b_0}{x} - \frac{b_0}{1-x}\\
&-9a_0+4a_0|\log x| + 2a_0|\log(1-x)|+\psi_2''(x),
\end{aligned}
\ee
with $\psi_1',\psi_2',\psi_2''\in L^{\infty}(\T)$. 

For $l \geq 1$ and $x$ let $x^-_{\min,l}$ and $x^+_{\min,l}$ denote the distance to the nearest singularity from the right and from the left, respectively, that is
\[
x^+_{\min,l} := \min_{0 \leq k < l} T^k x, 
\qquad x^-_{\min,l} := \min_{0 \leq k < l} (1 - T^k x),
\]
and let 
\[
x_{\min,l} := \min_{0 \leq k < l} \|T^k x\|=\min\{x_{\min,l}^-,x_{\min,l}^+\}.
\]

The following results hold for any $\alpha\in\RQ$.

\begin{lemma}\label{lemma:bounds_d}
There exists $C_i=C_i(\varphi)>0$, $i=1,2$, and $N=N(\varphi)\geq 1$ such that 
\begin{equation}\label{eq:upper_bound_d}
|S_l(\varphi')(x)|\leq C_1 \max\{q_{n}\log q_{n},\,x_{\min,l}^{-1}\},
\end{equation}
and
\begin{equation}
\label{eq:upper_bound_dd}
|S_l(\varphi'')(x)|\leq C_2 \, \max\{q_n^2,x_{\min,l}^{-2}\},
\end{equation}
for $0\leq l\leq q_{n}$ and at every $x$ point of continuity, for $n \geq N$.

Moreover, $\exists C_3=C_3(\varphi)>0$ such that 
\be 
S_{l}(\varphi'')(x)\geq C_3\, x_{\min,{l}}^{-2}-C_3^{-1}q_n\log l\label{eq:convexity}
\ee
for $n \geq N$ and for any $x\in \T$ continuity point of $S_l(\varphi)$; i.e. $x\not\in\cup_{m=0}^{l-1} \{R_\alpha^{-m}0\}$.
\end{lemma}
\begin{proof}
 Fix $1\leq l\leq q_{n}$. The orbit $\{R_\alpha^k x\}_{k=0}^{l-1}$ is well-separated. That is, by the three-distance theorem, consecutive points in this orbit are separated by at least $\|q_{n-1}\alpha\|$, and the $(2m+1)$-st and $(2m+2)$-nd closest iterates to the singularities lie at distance at least $x_{\min,l} + m/(2q_n)$ from $\{0,1\}$. Ordering by distance and using the expression of $\varphi'$ in \eqref{eq:def_Dbasic} gives
\[
\begin{aligned}
| S_{l}(\varphi'-\psi')(x)| \leq 2c_0(x_{\min,l})^{-1}+\sum_{m=1}^{l-1}2c_0\tfrac{2q_n}{m},
\end{aligned}
\]
and
\[
\begin{aligned}
| S_{l}(\psi')(x)| \leq -2|b_0|\log x_{\min,l}-\sum_{m=1}^{l-1}2|b_0|\log \tfrac{m}{2q_n} +2l|b_0|+ l\norm{\psi_1'}_{L^\infty}+ l\norm{\psi_2'}_{L^\infty}.
\end{aligned}
\]
Hence, for $n$ large enough, the existence of $C_1$ follows. 

The proof of \eqref{eq:upper_bound_dd} is analogous to that of \eqref{eq:upper_bound_d}, the main difference being that (see \eqref{eq:def_DDbasic})
\[
\begin{aligned}
| S_{l}(\psi'')(x)| &\leq
3|b_0|x_{\min,l}^{-1}-6|a_0|\log x_{\min,l}+\sum_{m=1}^{l-1}\left[3|b_0|\tfrac{2q_n}{m}+3|a_0|(3-2\log \tfrac{m}{2q_n})\right]\\
&+ l\norm{\psi_2''}_{L^\infty}\lesssim x_{\min,l}^{-1}+q_n\log l,
\end{aligned}
\]
and
\[
\frac{1}{q_n^2}\left[\sum_{m=1}^{l-1}3c_0(\tfrac{2q_n}{m})^2 +\sum_{m=1}^{l-1}3|b_0|\tfrac{2q_n}{m}+\sum_{m=1}^{l-1}3|a_0|(3-2\log \tfrac{m}{2q_n})+ l\norm{\psi_2''}_{L^\infty}\right],
\]
is uniformly bounded on $n$ (because $\sum m^{-2}<\infty$), and the dominant contribution comes from the single iterate closest to the singularity.

Finally, we have $\varphi''(x)\ge{c_0}{\norm x}^{-2}+\psi''(x),$ for any $x$. Hence, at every continuity point $x$,
\[
S_{l}(\varphi)''(x)\geq c_0 x_{\min,l}^{-2}-| S_{l}(\psi'')(x)|= c_0 x_{\min,l}^{-2}-O(x_{\min,l}^{-1}+q_n\log l).
\]
\end{proof}

For $n\ge1$, let
\[
\Gamma_{n}:=\{x:S_{q_{n}}(\varphi')(x)=0\}.
\]
The zeros $\Gamma_n$ play a central role. They are the points at which $S_{q_n}(\varphi)$ attains its minimum on each continuity interval, and the time $t_n$ will be defined in terms of these minimums. The following lemma shows that every point of $\Gamma_n$ is at distance comparable to $(q_n \log q_n)^{-1}$ from the nearest singularity.

\begin{lemma}
There exists $\hat\delta=\hat\delta(\varphi)>0$ and $N=N(\varphi)\geq 1$ such that 
\be
\frac{\hat \delta}{ q_{n}\log q_{n}}\leq y_{\min,q_{n}}\leq \frac{1}{\hat \delta q_{n}\log q_{n}},
\label{eq:location_min}
\ee
for any $y\in\Gamma_{n}$, $ n\geq N $.

Moreover, there exists $C_4:=C_4(\varphi)>0$ such that, provided $x_{\min,q_n} > (\hat\delta q_n \log q_n)^{-1}$,
\be 
S_{q_n}(\varphi')(x) \leq -C_4^{-1} q_n \log q_n.
\label{eq:rk2}
\ee
\label{lemma:location_min}
\end{lemma}

\begin{proof}
The function $S_{q_{n}}(\varphi)(x)$ is continuous on $\T\setminus\cup_{k=0}^{q_{n}-1}\{-k\alpha\}$, and on each of the $q_{n}$ intervals of continuity, the function is $C^2$.

Note $x_{\min,q_n}\leq 1/q_n$ for any $x$, so by \eqref{eq:convexity}, taking $q_n/\log q_n>C_3^{-2}$ we get $S_{q_{n}}(\varphi)$ is strictly convex on each interval of continuity, and $\to\infty$ at the singularities. In particular, $S_{q_{n}}(\varphi')$ is strictly monotone there and has exactly one zero. Hence $\Gamma_{n}=\{x:S_{q_{n}}(\varphi')(x)=0\}$ consists of exactly $q_{n}$ points, one in each interval.

Fix $1>\hat\delta>0$ and let $\tilde\varphi'$ denote the restriction of $\varphi'$ to $[\frac{1}{\hat \delta q_{n}\log q_{n}},1-\frac{1}{\hat \delta q_{n}\log q_{n}}]$. 
Applying the Denjoy--Koksma inequality to $\tilde\varphi'$, we get that there exists $C_4=C_4(\varphi)>{4}/{c_0}$ such that 
\[
\left|S_{q_{n}}(\tilde\varphi')(x)-q_{n}\int_\T\tilde\varphi'\right|\leq C_4\hat\delta q_{n}\log q_{n}\leq C_4 q_{n}\log q_{n},
\]
for $n$ large enough. Hence
\[
\begin{aligned}
\left|S_{q_{n}}(\tilde\varphi')(x)\right|&\leq c_0 q_{n}\log q_{n}+2q_n \norm{\psi}_{L^{\infty}}+C_4 q_{n}\log q_{n}\leq (2c_0+C_4)q_{n}\log q_{n},
\end{aligned}
\]
and
\[
-S_{q_{n}}(\tilde\varphi')(x)\geq c_0 q_{n}\log q_{n}-2q_n \norm{\psi}_{L^{\infty}}-C_4 \hat\delta q_{n}\log q_{n}\geq (c_0/2-C_4\hat\delta)q_{n}\log q_{n},
\]
for $n$ large enough.

Let $n$ be large enough such that $\frac{1}{\hat \delta q_{n}\log q_{n}}\leq\frac{1}{2q_{n}}.$ 

For the upper bound of \eqref{eq:location_min} and \eqref{eq:rk2}, consider $x$ such that $x_{\min,q_{n}}>({\hat \delta q_{n}\log q_{n}})^{-1}$. Then, by the previous inequality, it is enough to take $\hat\delta$ such that $\hat\delta<c_0/(4C_4)$ to get $S_{q_{n}}(\varphi')(x)=S_{q_{n}}(\tilde\varphi')(x)\leq -C_4^{-1}\, q_{n}\log q_{n}<0$.

The lower bound of \eqref{eq:location_min} follows from noting that $x_{\min,q_n}^{+}+x_{\min,q_n}^-\geq \norm{q_{n-1}\alpha}\geq \frac{1}{2q_n}$ and $|{\psi'(x)}|\lesssim|\log x|$ for any $x$, so that for $x_{\min,q_{n}}<{\hat \delta}({ q_{n}\log q_{n}})^{-1}$,
\[
\begin{aligned}
|S_{q_{n}}(\varphi')(x)|&\geq \varphi'({x_{\min,q_{n}}})-|S_{q_{n}}(\tilde\varphi')(x)|\geq \frac{c_0}{x_{\min,q_{n}}}-|{\psi'(x_{\min,q_{n}})}|-8q_n{c_0}-|S_{q_{n}}(\tilde\varphi')(x)|
\\&\geq (c_0/2\hat\delta-2c_0-C_4)\, q_{n}\log q_{n}>0,
\end{aligned}
\]
for $n$ large enough and taking $\hat\delta$ such that $\hat\delta<c_0(4c_0+2C_4)^{-1}$.

\end{proof}
\begin{remark}
Lemma \ref{lemma:location_min} shows that if $x_{\min,q_n} > (\hat\delta q_n \log q_n)^{-1}$, $S_{q_n}(\varphi')(x) <0$. Since $S_{q_n}(\varphi)$ is convex, its derivative is strictly increasing on each continuity, so the unique zero $y \in \Gamma_n$ satisfies $y^-_{\min,q_n} = y_{\min,q_n}$.
\label{rk:2}
\end{remark}

\subsection{Choice of $t_n$} \label{sec:7.3}

Let $A=[a,b]\times[c,d]\subset X$ be a box (i.e. $0\leq c<d\leq1$ and $0\leq c < d < \gamma/3$) and let $0<\eps<{\min\{\gamma/3-d,\gamma/6\}}$. Let $\alpha$ satisfy \eqref{H} for a sequence $\{r_n\}$ and fix $n\in \N$. 

We now define the sequence of times $t_n$ at which the correlation lower bound \eqref{eq:thm72} will be verified. The idea is to choose $t_n$ so that the flow image $\varphi_{t_n}(A)$ straddles the top of $A$: points near $\Gamma_{r_n}$ are mapped to just above $A$. The convexity of $S_{q_{r_n}}(\varphi)$ then forces an asymmetry that the lemmas of Section \ref{sec:convexity} quantify. 

We set
\begin{equation}
t_n:=\max_{y\in\Gamma_{r_n}}S_{q_{r_n}}(\varphi)(y)+2\gamma/3-\eps.    
\label{eq:defT}
\end{equation}

\begin{lemma} 
 There exists $N=N(\varphi)\geq 1$ such that $\forall n\geq N$ and $\forall y\in \Gamma_{r_n}$
\label{lemma:aligned}
\be
t_n<S_{q_{r_n}}(\varphi)(y)+2\gamma/3<S_{q_{r_n}+1}(\varphi)(y)-d-\eps.
\label{eq:aligned}\ee
Moreover,
\begin{equation}
t_n\asymp^{\varphi} q_{r_n}.\label{eq:approxT}
\end{equation}   
\end{lemma}

\begin{proof}
Let $y\in\Gamma_{r_n}$ and $l<q_{r_n}$. Since $S_{q_{r_n}}(\varphi')(y)=0$, we write
\[
S_{q_{r_n}}(\varphi')(y-l\alpha)= S_{q_{r_n}}(\varphi')(y-l\alpha)-S_{q_{r_n}}(\varphi')(y).
\]

Combining the bound \eqref{eq:upper_bound_dd} with \eqref{H} by applying the mean value theorem, we obtain
\[
|S_{q_{r_n}}(\varphi')(y-l\alpha)|\leq |S_{l}(\varphi'')(\theta)|\norm{q_{r_n}\alpha} \leq \frac{C_2\theta_{\min,q_{r_n}}^{-2}}{q_{r_n+1}}\leq \frac{C_2\theta_{\min,q_{r_n}}^{-2}}{q_{r_n}\log q_{r_n}\log\log q_{r_n}},
\]
for some $\theta\in(y-l\alpha,y+(q_{r_n}-l)\alpha)$ (or $(y+(q_{r_n}-l)\alpha,y-l\alpha)$). To show that $S_{q_{r_n}}(\varphi')$ is continuous on this interval, let $\hat\delta$ be as in Lemma \ref{lemma:location_min}. For $n$ large enough, 
\[
\norm{q_{r_n}\alpha}\leq\frac{1}{q_{r_n}\log q_{r_n}\log\log q_{r_n}}\leq\frac{\hat\delta}{2q_{r_n}\log q_{r_n}}.
\]
Therefore
\[
\frac{\hat\delta q_{r_n}\log q_{r_n}}{2}\leq q_{r_n}\log q_{r_n}\left(\frac{1}{\hat\delta} +\frac{\hat\delta}{2}\right)^{-1}\leq \theta_{\min,q_{r_n}}^{-1}\leq \frac{2q_{r_n}\log q_{r_n}}{\hat\delta}.
\]

Let $y\in\Gamma_{r_n}$ be the rightmost point when ordered in $(0,1)$, such that for any $0\leq l<q_{r_n}$, 
\[
y_{\min,q_{r_n}}=y_{\min,q_{r_n}}^-=(y-l\alpha)_{\min,q_{r_n}}^-=(y-l\alpha)_{\min,q_{r_n}}.
\] 
By the convexity result \eqref{eq:convexity}, provided $q_{r_n}/\log q_{r_n}>2C_3^{-2}$
\[ 
S_{q_{r_n}}(\varphi'')(x)\geq C_3\, x_{\min,{l}}^{-2}-C_3^{-1}q_{r_n}\geq \tfrac{C_3}{2}\, x_{\min,{l}}^{-2}.
\]
Therefore if $y'\in\Gamma_{r_n}$ lies in the same continuity interval as $y-l\alpha$, then
\[ 
\begin{aligned}
|y'-y+l\alpha|&\leq\frac{|S_{q_{r_n}}(\varphi')(y-l\alpha)|}{\min_{(y'\wedge(y-l\alpha),y'\vee(y-l\alpha))} S_{q_{r_n}}(\varphi'')}\\
&\leq \frac{C_2}{q_{r_n}\log q_{r_n}\log\log q_{r_n}}\left(\frac{2q_{r_n}\log q_{r_n}}{\hat\delta}\right)^2 \frac{2}{C_3}\left(\frac{1}{\hat\delta q_{r_n}\log q_{r_n}}\right)^2\\
&\leq \tilde C \left(q_{r_n}\log q_{r_n}\log\log q_{r_n}\right)^{-1},
\end{aligned}
\]
for some $\tilde C=\tilde C(\varphi,\hat\delta)>0$.

Let $y'\in\Gamma_{r_n}$. Using the previous step,
\[
\begin{aligned}
&|S_{q_{r_n}}(\varphi)(y)-S_{q_{r_n}}(\varphi)(y')|
\\&\leq |S_{q_{r_n}}(\varphi)(y)-S_{q_{r_n}}(\varphi)(y-l\alpha)|+ |S_{q_{r_n}}(\varphi)(y-l\alpha)-S_{q_{r_n}}(\varphi)(y')|
\\&= |S_{l}(\varphi)(y-l\alpha+q_{r_n}\alpha)-S_{l}(\varphi)(y-l\alpha)|+ |S_{q_{r_n}}(\varphi)(y-l\alpha)-S_{q_{r_n}}(\varphi)(y')|.
\end{aligned}
\]
Both terms are bounded by the mean value theorem, using the estimate \eqref{eq:upper_bound_d} and the small distance between $y-l\alpha$ and both $y-l\alpha+{q_{r_n}\alpha}$ and $y'$:
\[ 
|S_{l}(\varphi)(y-l\alpha+q_{r_n}\alpha)-S_{l}(\varphi)(y-l\alpha)|\leq\norm{q_{r_n}\alpha}2C_1q_{r_n}\log q_{r_n}/\hat\delta=O((\log\log q_{r_n})^{-1}),
\]
\[ 
|S_{q_{r_n}}(\varphi)(y-l\alpha)-S_{q_{r_n}}(\varphi)(y')|\leq|y'-y+l\alpha|\,2C_1q_{r_n}\log q_{r_n}/\hat\delta=O((\log\log q_{r_n})^{-1}).
\]

Altogether, $|S_{q_{r_n}}(\varphi)(y)-S_{q_{r_n}}(\varphi)(y')|$ decays asymptotically as $({\log\log q_{r_n}})^{-1}$ with $n$, for any $y,y'\in\Gamma_{r_n}$. 
Since $S_{q_{r_n}}(\varphi)$ is asymptotically constant on $\Gamma_{r_n}$, by definition of $t_n$ in \eqref{eq:defT}, for $n$ large enough we obtain
\[
t_n < S_{q_{r_n}}(\varphi)(y)+2\gamma/3\quad \text{for all } y\in\Gamma_{r_n}.
\]
This gives \eqref{eq:aligned}.

To prove \eqref{eq:approxT}, pick $y\in\Gamma_{r_n}$. Using $\int_{\T}\varphi=1$, Denjoy--Koksma estimates on the restriction of $\varphi$ to $[\frac{\hat{\delta}}{q_{r_n}\log q_{r_n}},1-\frac{\hat{\delta}}{q_{r_n}\log q_{r_n}}]$ give
\[
\begin{aligned}
&\left|S_{q_{r_n}}(\varphi)(y)-q_n\int_\T\varphi\chi_{\{\norm{x}\geq\hat\delta(q_{r_n}\log{q_{r_n}})^{-1}\}}\right|
\leq \text{Var}\left(\varphi\chi_{\{\norm{x}\geq\hat\delta(q_{r_n}\log{q_{r_n}})^{-1}\}}\right)
\\&\leq \varphi(\tfrac{\hat{\delta}}{q_{r_n}\log q_{r_n}})+\varphi(1-\tfrac{\hat{\delta}}{q_{r_n}\log q_{r_n}})+\int_{\hat\delta(q_{r_n}\log{q_{r_n}})^{-1}}^{1-\hat\delta(q_{r_n}\log{q_{r_n}})^{-1}}\frac{3c_0}{x}dx+\text{Var}(\psi)+ 2\norm{\psi}_{L^\infty}\\
&=O(\log q_{r_n}),
\end{aligned}
\]
where we have used that $\psi$ has bounded variation.
Note that the integral on the left hand side $\to 1$ as $n\to \infty$, and recall $\hat\delta=\hat\delta(\varphi)$, so by definition of $t_n$, this yields
\[
t_n\asymp^{\varphi} q_{r_n}.
\]

\end{proof}

\begin{lemma}\label{lemma:rk3} For every $s \in (0,\varepsilon)$, $\phi_{t_n - s}$ maps the set $(\Gamma_{r_n} \cap A) \times (c,d)$ outside $A$. That is, $[c,d+\eps]\cap\Pi_s\phi_{t_n}((\Gamma_{r_n}\cap A)\times[c,d+s])=\emptyset$.
\end{lemma}
This is the key disjointness condition used in the proof of Proposition \ref{prop:strips}.
Moreover, since $S_{q_{r_n}}\varphi$ attains its minimum in $\Gamma_{r_n} $, by definition of $t_n$: $N(x,t_n+d+\eps)\leq q_{r_n}$ for any $x\in \T$.
\begin{proof}
Conditions \eqref{eq:defT}--\eqref{eq:aligned} together imply that \[\Pi_s\phi_{t_n+s}(\{y\}\times[c,d])\subset[2\gamma/3-\eps+c+s,2\gamma/3+d+s]\subset[\gamma/3+d,\gamma],\] placing the image above $A$ entirely.
\end{proof}

\subsection{Computation of the correlations}\label{sec:7.4} 

We want to compute a lower bound for 
\[
\mu(\varphi_{t_n}A \cap \varphi_\varepsilon A) - \mu(\varphi_{t_n}A \cap A).
\]
We will do so by dividing $A$ into thin horizontal strips and computing the contribution to this difference from each pair of strips. The gain comes from the convexity of the Birkhoff sums of $\varphi$ established in Section \ref{sec:7.2}: the flow image of each strip intersects the $\varepsilon$-shifted version more than the unshifted one. The contribution from the points where this convexity argument does not directly apply will be shown to be negligible. In particular, we will prove the following proposition.

\begin{proposition} Let $\{A_i\}_{i=1}^{N_0}$ be a partition of $A$ in $N_0$ strips of width $b-a$ and height $\eps:=\tfrac{d-c}{N_0}$. There exists $ C_5=C_5(\varphi,A)>0$, $\eps^*>0$, and $N=N(\varphi,A)\geq 1$ such that, provided $\eps^*>\eps>0$ and $n\geq N$,
\[
\mu(\phi_{t_n}A_i\cap \phi_\eps A_j)-\mu(\phi_{t_n}A_i\cap A_j)>\frac{C_5\eps^4}{\log q_{r_n}},
\]
for any $1\leq i,j\leq N_0.$ 
\label{prop:strips}
\end{proposition}
Let us first show that Theorem \ref{thm:log_decay2} follows from Proposition \ref{prop:strips}.
\begin{proof}[Proof of Theorem \ref{thm:log_decay2}] We can pick $\eps>0$ small enough such that for each $t_n$, with $n$ large enough, Proposition \ref{prop:strips} holds, yielding
\[
\begin{aligned}
\mu(\phi_{t_n}A\cap \phi_\eps A)-\mu(\phi_{t_n}A\cap A)&=\sum_{i,j=1}^{N_0}\mu(\phi_{t_n}A_i\cap \phi_\eps A_j)-\mu(\phi_{t_n}A_i\cap A_j)\\&\geq N_0^2\frac{C_5\eps^4}{\log q_{r_n}}=\frac{C_5(d-c)^2\eps^2}{\log q_{r_n}}.
\end{aligned}
\]
The result follows from the estimate \eqref{eq:approxT}.
\end{proof}

To prove Proposition \ref{prop:strips}, we need to introduce and prove some auxiliary lemmas. 
First, let us define a subset of $(a,b)$, $F_n$, where we have good control of the Birkhoff sums, and show that we can discard $(a,b)\setminus F_n$.  

Let $\hat\delta>\delta>0$, and define
\[
F_{n}'=F_{n}'(\delta):=\mathbb{T}\setminus\bigcup_{k=0}^{q_{r_n}-1} R_\alpha^{-k}\Big(\frac{-\delta}{q_{r_n}\log q_{r_n}},\frac{\delta}{q_{r_n}\log q_{r_n}}\Big).
\]

Let $N_*:=q_{r_n}-\frac{2C_1 }{\delta}\log q_{r_n}-1$. For each $1\leq l\leq q_{r_n}$, let
\[ 
G_{n,l}:=\left\{x\in \T:(x_{\min,l}^+,x_{\min,l}^-)=(x_{\min,q_{r_n}}^+,x_{\min,q_{r_n}}^-)\right\}\setminus \bigcup_{k=0}^{l-1}\{-k\alpha\},
\] 
and let 
\[
F_{n}(\hat t)=\left\{x\in(a,b)\cap G_{n,{N(x,t_n+\hat t+\eps)}} \right\}\cap\{N(x,t_n+\hat t+\eps)\geq N_*\},
\]
for $\hat t\in[c,d-\eps]$.

\begin{lemma} There exists $C_6=C_6(\varphi,A)>0$ and $N=N(\varphi,A)\geq 1$ such that for $n\geq N$ and for any $\hat t\in[c,d-\eps]$,
\be 
\lambda((a,b)\setminus F_n(\hat t))\leq\frac{\delta C_6}{\log q_{r_n}}.
\ee 
\label{lemma:good_fn}
\end{lemma}

\begin{proof}
First, note that $N(y,t_n+\hat t+\eps)=q_{r_n}$ $\forall y\in\Gamma_{r_n}$, and it is the maximum on $\T $ by Lemma \ref{lemma:rk3}. By the definition of $N(\cdot,\cdot)$ and \eqref{eq:upper_bound_d}, for any $x$ in the same component of $F_n'$ as $y$
\[
\begin{aligned}
N(x,t_n+\hat t+\eps)&\geq N(y,t_n+\hat t+\eps)-N(y,(S_{N(x,t_n)}(\varphi)(x)-S_{N(x,t_n+\hat t+\eps)}(\varphi)(y)))-1
\\&\geq N(y,t_n+\hat t+\eps)-\frac{1}{\gamma}(S_{N(x,t_n+\hat t+\eps)}(\varphi)(x)-S_{N(x,t_n+\hat t+\eps)}(\varphi)(y))-1
\\&\geq q_{r_n}-\frac{1}{\gamma}\max_{F_n'}S_{N(x,t_n+\hat t+\eps)}(\varphi'){|y-x|}-1\geq q_{r_n}-\frac{C_1 q_{r_n}\log q_{r_n}}{\delta}\frac{2}{q_{r_n}}-1
\\&\geq q_{r_n}-\frac{2C_1 }{\delta}\log q_{r_n}-1.
\end{aligned}
\]

Consider each of the $q_{r_n}$ segments of $\T\setminus\cup_{m=0}^{q_{r_n}-1}\{R_\alpha^{-m}0\} $. 
For each of them, there exist $k_1$, $k_2$ such that the interval is $(R_\alpha^{-k_1}0,R_\alpha^{-k_2}0)$. Hence
\[
x_{\min,q_{r_n}}^+=R_\alpha^{k_1} x, \qquad \text{and} \qquad
x_{\min,q_{r_n}}^-=1-R_\alpha^{k_2} x,
\]
for $x$ on the segment. In particular, $k_1$ and $k_2$ each take every value in $\{0,\dots,q_{r_n}-1\}$ exactly once.
This means that all but at most $\frac{4C_1 }{\delta}\log q_{r_n}+2$ of the maximal components $I\subset F_n'$ satisfy 
\[
(x_{\min,q_{n}}^-,x_{\min,q_{n}}^+)=\left(x_{\min,N(x,t_n+ \hat t)}^+,x_{\min,N(x,t_n+ \hat t)}^-\right)
\]
for any $x\in I$.
Hence, the measure of the intervals of $F_n'$ for which it is not true is at most
\[
\left(\frac{4C_1 }{\delta}\log q_{r_n}+1\right)\frac{2}{q_{r_n}}\leq \frac{\delta}{\log q_{r_n}},
\]
for $n$ large enough.

On the other hand, noting that each interval of continuity of $S_{q_{r_n}}(\varphi)$ has length at least $(2q_{r_n})^{-1}$, we get
\[
\lambda((a,b))-\lambda((a,b)\cap F_n')\leq \left[(b-a){2q_{r_n}}+1\right]\tfrac{2\delta}{q_{r_n}\log q_{r_n}}.
\]

The result follows from combining the two inequalities and noting that
\[
(a,b)\cap F_n'\cap\left\{x\in G_{n,{N(x,t_n+\hat t+\eps)}}\right\}\subset F_n( \hat t).
\]
\end{proof}

\begin{lemma}
\label{lemma:rk1}
For any $\hat t_1,\hat t_2\in[c,d-\eps]$ and $N_*\leq l\leq q_{r_n}$, 
\[
F_n(\hat t_1,l):= F_n(\hat t_1)\cap R_\alpha^{-l}(a,b)\cap\{t_n+\hat{t_1}- S_l(\varphi)(x)\in[\hat t_2-\eps,\hat t_2+2\eps]\}
\]
is a union of intervals satisfying, for any $I$ such interval,
\begin{itemize}
    \item[i)] $\lambda(I) \leq 2/q_{r_n}$,
    \item[ii)] either $S_l(\varphi')<0$ or $S_l(\varphi')>0$ on $I$,
    \item[iii)] $\sup_I \{t_n+\hat t_1-S_l(\varphi)\}= t_2+2\eps$ for all but at most two intervals.
\end{itemize}
\end{lemma}

\begin{proof}
\textit{i)} Fix $l \leq q_{r_n}$. By definition, $G_{n,l}$ is the union of some of the maximal intervals of $\mathbb{T} \setminus \bigcup_{m=0}^{q_{r_n}-1} \{R_\alpha^{-m}0\}.$ The bound of the length of the intervals then follows from the three-distance theorem.

\textit{ii)} Now, for any $x \in G_{n,l}$, we have $x_{\min,l} \leq \frac{1}{q_{r_n}}$.
Hence, if $q_{r_n}/\log q_{r_n} > C_3^{-2}$, then by \eqref{eq:convexity}, the function $S_l(\varphi)$ is convex on each interval contained in $G_{n,l}$. Moreover, by the definition of $t_n$ (see \eqref{eq:defT}), the minimum of $S_{q_{r_n}}(\varphi)(x)$ on each interval of $\mathbb{T} \setminus \bigcup_{m=0}^{q_{r_n}-1} \{R_\alpha^{-m}0\}$ is bounded above by
\[
t_n - \frac{2\gamma}{3} + \varepsilon \leq t_n + \hat t_1 - \hat t_2 - 2\varepsilon,
\]
and $S_{q_{r_n}}(\varphi)(x)\geq S_{l}(\varphi)(x)$.
Therefore, this minimum is not attained in $F_n(\hat t_1,l)$ for any $l \leq q_{r_n}$, so $S_l(\varphi')$ has constant sign.

\textit{iii)} As we just showed, in the intervals in $\{t_n+\hat{t_1}- S_l(\varphi)(x)\in[\hat t_2-\eps,\hat t_2+2\eps]\}$, $S_l(\varphi)(x)$ is monotone and attains the maximum $t_2+2\eps$. However, note that we are only considering points in $(a,b)\cap R_\alpha^{-l}(a,b)$, so that the leftmost and rightmost intervals of $(a,b)\cap R_\alpha^{-l}(a,b)\cap\{t_n+\hat{t_1}- S_l(\varphi)(x)\in[\hat t_2-\eps,\hat t_2+2\eps]\}$ might get cut (see Figure \ref{fig:2}).

\end{proof}
\begin{figure}[h!]
\centering
\includegraphics[width=0.7\linewidth]{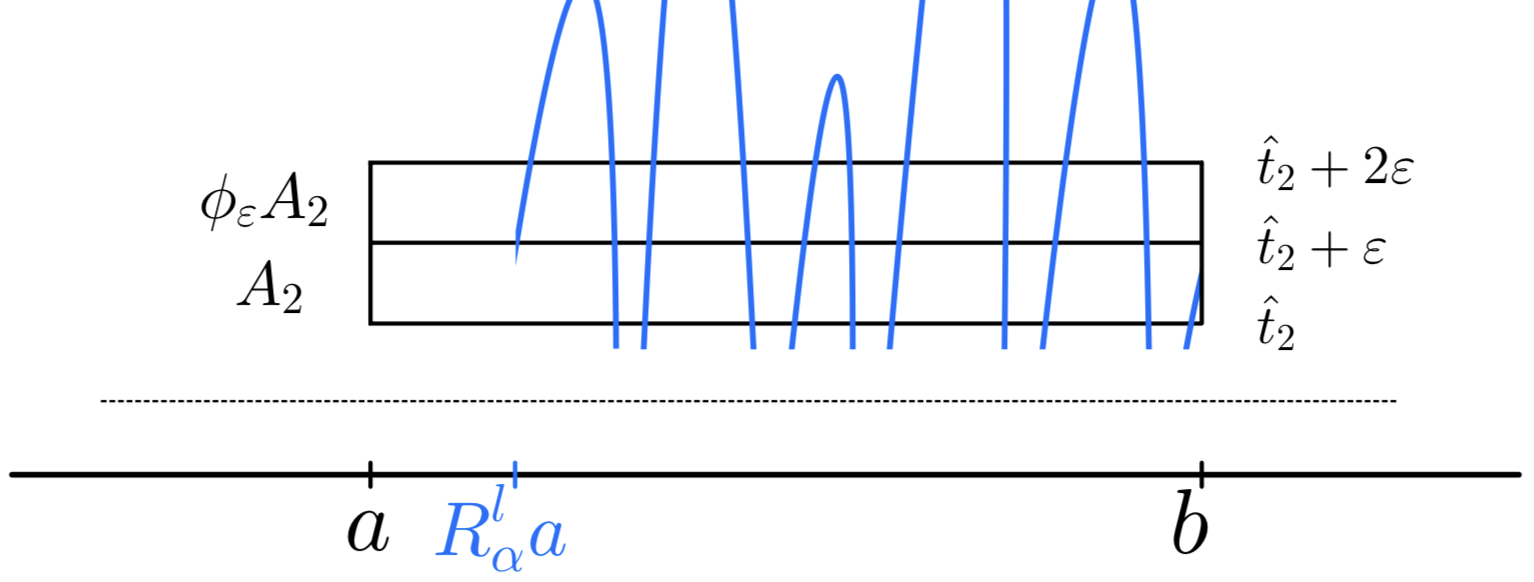}
\caption{Schematics of the graph $(R_\alpha^lx,t_n+\hat t_1-S_l(\varphi)(x))$ for $x\in(a,b)$ (in blue) and its intersection with the strips $A_2:=(a,b)\times[\hat t_2,\hat t_2+\eps]$ and $\phi_\eps A_2=(a,b)\times[\hat t_2+\eps,\hat t_2+2\eps]$.}
\label{fig:2}
\end{figure}

The last ingredient for the proof of Proposition \ref{prop:strips} is that, given $\hat t\in[c,d-\eps]$, the restriction of $S_{q_{r_n}}(\varphi)$ on most of the intervals in $\{t_n-S_{q_{r_n}}(\varphi)+\hat t\in[c-\eps,d+2\eps],$  $S_{q_{r_n}}(\varphi')<0\}$ satisfies (up to translation and restriction to a subinterval) the hypothesis of Lemmas \ref{lemma:convex1} and \ref{lemma:convex2}. This is the content of the next Lemma.

\begin{lemma} \label{lemma:lower_bound_d}Let $H_n:=\{x:t_n-d-2\eps+c\leq S_{q_{r_n}}(\varphi)(x)\leq t_n-c+d, S_{q_{r_n}}(\varphi')(x)<0\}$. There exists $C_7(\varphi)>0$ such that for any $x\in H_n$
\[
S_{q_{r_n}}(\varphi'')(x)\geq {C_7}\, q_{r_n}^2\log^2 q_{r_n}.
\] 

Moreover, at least $(a,b)\tfrac{q_n}{2}-3$ of the intervals of $H_n$ are fully contained in $(a,b)\cap R_\alpha^{-q_{r_n}}(a,b)$. 
\end{lemma}
\begin{proof} 
By Remark \ref{rk:2}, every $y \in \Gamma_n$ satisfies $y^-_{\min,q_n} = y_{\min,q_n}$, i.e. the nearest singularity lies to the right of $y$. We use this throughout the proof.

Let $x\in H_n$. Let $y$ be the zero of $S_{q_{r_n}}(\varphi')$ on the same interval of continuity of $x$. Since $S_{q_{r_n}}(\varphi')(z)<0$ if and only if $z<y$, we have $y_{\min,q_{r_n}}^-\leq x_{\min,q_{r_n}}^-.$

Recall that, by \eqref{eq:aligned}, $t_n<S_{q_{r_n}}(\varphi)(y)+2\gamma/3.$ Hence
\[
S_{q_{r_n}}(\varphi)(x)-S_{q_{r_n}}(\varphi)(y) <d-c+2\gamma/3< \gamma.
\]

From \eqref{eq:rk2} we know that, for $x_{\min,q_{r_n}}^->(\hat \delta q_{r_n}\log q_{r_n})^{-1} $, $\exists C_4>0$ such that
\[ 
S_{q_{r_n}}(\varphi')(x)\leq -{C_4}\, q_{r_n}\log q_{r_n}.
\] 
Let $z<y$ be the point in the same continuity interval of $x$ and $y$ such that $z_{\min,q_{r_n}}^-=(\hat \delta q_{r_n}\log q_{r_n})^{-1}$. Then either 
\[
x<z\implies x_{\min,q_{r_n}}^-<(\hat \delta q_{r_n}\log q_{r_n})^{-1},
\] 
or
\[
\begin{aligned}
x_{\min,q_{r_n}}^-&=y_{\min,q_{r_n}}^-+(y-x)=z_{\min,q_{r_n}}^-+(z-x) \\&\leq(\hat \delta q_{r_n}\log q_{r_n})^{-1} +{\gamma}(C_4 \,q_{r_n}\log q_{r_n})^{-1}, 
\end{aligned}
\]
where we have used
\[
(z-x)C_4 q_{r_n}\log q_{r_n}\leq-\int_x^z S_{q_{r_n}}(\varphi')(u)\,du\leq S_{q_{r_n}}(\varphi)(x)-S_{q_{r_n}}(\varphi)(y). 
\]
The existence of $C_7$ then follows from \eqref{eq:convexity}:
\[
S_{q_{r_n}}(\varphi'')(x)\geq C_3 x_{\min,q_{r_n}}^{-2}-C_3^{-1}q_{r_n}\log q_{r_n}=O((q_{r_n}\log q_{r_n})^2).
\]

Finally, recall that, by convexity, each interval of $H_n$ is nested on a different interval of $\T\setminus\cup_{m=0}^{q_{r_n}-1}\{R_\alpha^{-m}0\}$, which cover $(a,b)\cap R_\alpha^{-q_{r_n}}(a,b)$ up to a measure zero set and have length bounded by $2/q_{r_n}$. Therefore the number of intervals of $H_n$ fully contained in $(a,b)\cap R_\alpha^{-q_{r_n}}(a,b)$ is lower bounded by
\[
\frac{q_{r_n}}{2}(b-a-\norm{\alpha q_{r_n}})-2.
\]
\end{proof}
\begin{remark}
\label{rk:4} By construction of $H_n$, for any $\hat t_1,\hat t_2\in(c,d)$ any interval $I$ such that $S_{q_{r_n}}(\varphi)(x)\geq t_n-\hat t_1$, $\forall x \in I$, and \[\hat t_2+2\eps=\sup\Pi_s( \phi_{t_n}(I\times \{\hat t_1\}) ),\qquad\hat t_2-\eps=\inf\Pi_s( \phi_{t_n}(I\times \{\hat t_1\}) ),\] is contained in $H_n$.
\end{remark}

Now, we can prove Proposition \ref{prop:strips}.

\begin{proof}[Proof of Proposition \ref{prop:strips}]
Without loss of generality, let $\hat t_1,\hat t_2\in\{c,c+\eps,\dots,d-\eps\}$, not necessarily different, and consider $A_1$, $A_2$ such that 
\[
A_i=(a,b)\times[\hat t_i,\hat t_i+\eps],\quad i=1,2.
\]

Notice that $A_2$ and $\phi_\eps A_2$ are stacked one on top of each other. 

Let $N_*\leq l\leq q_{r_n}.$ First, we claim that for any but (at most) the leftmost and rightmost maximal intervals $I\subset F_n( \hat t_1,l)$ it holds 
\[
\mu(\phi_{t_n}(I\times (\hat t_1,\hat t_1+\eps))\cap\phi_\eps A_2)\geq\mu(\phi_{t_n}(I\times (\hat t_1,\hat t_1+\eps))\cap A_2).
\]
It follows from Lemma \ref{lemma:rk1} that the leftmost and rightmost intervals of $F_n(\hat t_1,l)$ may be cut off by the boundary of $(a,b)$, but for the rest of the intervals Lemma \ref{lemma:convex3} applies with $f(x)=t_n-S_{q_{r_n}}(\varphi)(x)+\hat t_1-\hat t_2$.

Secondly, note that Lemmas \ref{lemma:rk1} and \ref{lemma:good_fn} imply that the portion of the area of $A_1$ lying on the extreme intervals of all $F_n(\hat t_1,l)$, together with the portion on $(a,b)\setminus F_n(\hat t_1)$, is bounded above by
\[
\frac{4}{q_{r_n}}(q_{r_n}-N_*+1)\eps + \frac{\delta C_6}{\log q_{r_n}}\eps= \frac{4}{q_{r_n}}\left(\frac{2C_1 }{\delta}\log q_{r_n}+2\right)+ \frac{\delta C_6}{\log q_{r_n}}\eps.
\]

Therefore it is enough to show that there exists enough different subintervals $I$ of $(a,b)$ satisfying
\[
\hat t_2+2\eps=\sup\Pi_s( \phi_{t_n}(I\times \{\hat t_1\}) ),\qquad\hat t_2-\eps=\inf\Pi_s( \phi_{t_n}(I\times \{\hat t_1\}) ),
\]
\be
\mu(\phi_{t_n}(I\times (\hat t_1,\hat t_1+\eps))\cap\phi_\eps A_2)-\mu(\phi_{t_n}(I\times (\hat t_1,\hat t_1+\eps))\cap A_2)>\frac{\eps^4}{q_{r_n}\log q_{r_n}}.
\label{eq:positive_diff}
\ee
We claim that this holds for a subinterval of every interval of $H_n$ fully contained in $(a,b)\cap R_\alpha^{-q_{r_n}}(a,b)$. From Lemma \ref{lemma:lower_bound_d}, this will imply there are at least $(b-a)q_{r_n}/2-3$ of such intervals, and the result will follow from considering $\delta$ small enough (with respect to $\eps^3$) and $n$ large enough.

Now, let $J$ be a maximal interval of $H_n$ fully contained in $(a,b)\cap R_{\alpha}^{-q_{r_n}}(a,b)$, and let $I\subset J$ be the interval for which $    -\eps\leq  t_n-S_{q_{r_n}}(\varphi)(x)+\hat t_1-\hat t_2\leq 2\eps.$
Let $f(x)=t_n-S_{q_{r_n}}(\varphi)(x)+\hat t_1-\hat t_2$ on $I$. By Lemmas \ref{lemma:bounds_d} and \ref{lemma:lower_bound_d}, $f$ satisfy the hypotesis of Lemma \ref{lemma:convex1}, with 
\[
k=1,\qquad M=q_{r_n}\log q_{r_n},\qquad \text{and}\qquad C= \max\{\hat\delta^{-1}C_1,C_7^{-1}\}
\]
so that Lemma \ref{lemma:convex2} applies. Hence \eqref{eq:positive_diff} holds for $I$ and summing over all the intervals gives
\[
\begin{aligned}
\mu(\phi_{t_n}A_1\cap \phi_\eps A_2)-\mu(\phi_{t_n}A_1\cap A_2)>\left((b-a)\frac{q_{r_n}}{2}-3\right)\frac{\eps^4}{q_{r_n}\log q_{r_n}}- \frac{4}{q_{r_n}}N_*\eps - \frac{\delta C_6}{\log q_{r_n}}\eps
\\=\eps^4 O(({\log q_{r_n}})^{-1})-O(\log q_{r_n} q_{r_n}^{-1})-\delta \eps O(({\log q_{r_n}})^{-1}).
\end{aligned}
\]
Taking $  \delta $ small enough and $n$ large enough completes the proof.
\end{proof}

\section{Proof of Theorem \ref{thm:sum_cor}}\label{sec:last}
Through this section, we fix $\alpha\in \RQ$ such that its best denominators satisfy \be q_{n+1}\geq e^{3\log ^2q_{n}}q_{n},\label{h}\tag{h} \ee for every $n$.

The strategy is to construct a sequence of disjoint time intervals $[t_{n,k} - \varepsilon, t_{n,k} + \varepsilon]$ on which $|\mathrm{Cor}_{t_{n,k}}(\chi_A, \chi_A) |\gtrsim ({\sqrt{k} \log q_n})^{-1}$, indexed by pairs $(n,k)$ with $1 \leq k \leq e^{\log^2 q_n}$. Squaring and summing over all such pairs then gives a divergent series, proving \eqref{eq:thm3}. The stronger Diophantine condition \eqref{h} is needed here, compared to \eqref{H} in Theorem \ref{thm:log_decay2}), because we must simultaneously control $e^{\log^2 q_n}$ different scales $k$ at each induction step $n$.

Recall that $\varphi$ satisfies \eqref{eq:phi_basic}--\eqref{eq:phi_Dbasic} and \eqref{eq:def_Dbasic}--\eqref{eq:def_DDbasic}, and that we define, for $l\geq 0$,
\[
x^+_{\min,l} := \min_{0 \leq k < l} T^k x, \quad 
x^-_{\min,l} := \min_{0 \leq k < l} (1 - T^k x),\quad
x_{\min,l} :=\min\{x_{\min,l}^-,x_{\min,l}^+\}.
\]

\subsection{Extension of the definitions and bounds for pairs $(n,k)$}

Fix $n \geq 1$ and $1 \leq k \leq e^{\log^2q_{n}}$. Recall that for our choice of $\alpha$: $e^{3\log ^2q_{n}}\leq a_{n+1}$. Denote by $\mathcal{P}'_{n,k}$ the collection of intervals of continuity of $S_{kq_n}(\varphi)$ on $\mathbb{T}$.

\begin{lemma}
Given $ k \leq a_{n+1}$, $\mathcal{P}'_{n,k}$ is composed by exactly $(k-1)q_{n}$ intervals of length $\|q_{n}\alpha\|$, and exactly $q_{n}$ intervals of length either $\|q_{n-1}\alpha\| - (k-2)\|q_{n}\alpha\|$ or $\|q_{n-1}\alpha\| - (k-1)\|q_{n}\alpha\|$.
\label{lemma:partition}
\end{lemma}

\begin{proof}
The partition $\mathcal{P}'_{n,k}$ is obtained by placing $k$ successive $q_{n}$-orbits on $\mathbb{T}$: 
\[
\mathbb{T} \setminus \bigcup_{m=0}^{k-1} \left( \{-l\alpha\}_{l=0}^{q_{n}-1} - m q_{n} \alpha \right).
\]

By the three-distance theorem applied to the $q_{n}$ points $\{-j\alpha\}_{j=0}^{q_{n}-1}$, the torus $\mathbb{T}$ is partitioned into intervals of exactly two lengths:
\[
\|q_{n-1}\alpha\| + \|q_{n}\alpha\| \quad \text{(type A)} \qquad \text{and} \qquad \|q_{n-1}\alpha\| \quad \text{(type B)}.
\]
We refer to both types collectively as \emph{long} intervals.

Recall $q_{2n}\alpha>p_{2n}$ and $q_{2n+1}\alpha<p_{2n+1}$ for any $n$. Consequently, when the $m$-th orbit $\{-l\alpha\}_{l=0}^{q_{n}-1} - mq_{n}\alpha$ is added for $m \geq 1$, each of its $q_{n}$ points lands exactly $\|q_{n}\alpha\|$ to the left (if $n$ is even) or the right (if $n$ is odd) of the corresponding point of the $(m-1)$-th orbit, thereby cutting a strip of width $\|q_{n}\alpha\|$ (\emph{short} interval) from the left end ($n$ odd) or the right end ($n$ even) of each long interval as left by the $(m-1)$-th orbit.

Since $k \leq a_{n+1}$, the total displacement after $k$ steps satisfies
\[
k\|q_{n}\alpha\| \leq a_{n+1}\|q_{n}\alpha\| < \|q_{n-1}\alpha\|,
\]
where the last inequality uses the continuant relation $\|q_{n-1}\alpha\| = a_{n+1}\|q_{n}\alpha\| + \|q_{n+1}\alpha\|$. Hence every cut lands strictly inside a long interval.

Each new orbit ($m = 1, \ldots, k-1$) cuts every long interval once, producing $(k-1)q_{n}$ short intervals in total. The $q_{n}$ remaining intervals are the left-over pieces of the original long intervals after $k-1$ cuts each. 

A type A interval after $k-1$ cuts of width $\|q_{n}\alpha\|$, has remaining length $ \|q_{n-1}\alpha\| + \|q_{n}\alpha\| - (k-1)\|q_{n}\alpha\| = \|q_{n-1}\alpha\| - (k-2)\|q_{n}\alpha\|.$
A type B interval after $k-1$ cuts, has remaining length $ \|q_{n-1}\alpha\| - (k-1)\|q_{n}\alpha\|. $ This completes the proof.
\end{proof}
Let $\mathcal{P}_{n,k}\subset\mathcal{P}'_{n,k}$ be the $q_{n}$ intervals of length $>\norm{q_{n}\alpha}$. From the previous lemma, we know $\lambda(I)\geq \norm{q_{n.1}\alpha}-k\norm{q_{n}\alpha}\geq \frac{1}{2q_{n}}-\frac{e^{\log^2q_{n}} }{q_{n+1}}\geq \frac{1}{3q_{n}}$, for $n$ large enough.

\begin{lemma}
\label{lemma:rk5} Let $I\in \mathcal{P}_{n,k}$. Then for any $1 \leq m \leq k\leq e^{\log^2q_{n}}$ and every $x\in I$ satisfying $\text{dist}(x,\partial I)\geq e^{-2\log^2q_{n}}$, it holds
\[
|(R_\alpha^{mq_{n}}x)_{\min,q_{r_n}}^\pm-x_{\min,q_{r_n}}^\pm|\leq\tfrac{1}{2} x_{\min,q_{r_n}}^\pm.
\]
\end{lemma}
\begin{proof}
Note that 
\[
|(R_\alpha^{mq_{n}}x)_{\min,q_{r_n}}^\pm-x_{\min,q_{r_n}}^\pm|=m\norm{q_{n}\alpha}\leq e^{-\log^2q_{n}}\tfrac{1}{q_{n+1}}\leq\tfrac{1}{2} x_{\min,q_{r_n}}^\pm,
\]
for $n\geq 2$, where we have used $q_{n+1}\geq e^{3\log^2q_{n}}q_n$ for every $n$.
\end{proof}

\begin{lemma}\label{lemma:bounds_dk}
There exists $N=N(\varphi)\geq 1$ and $C_i'=C_i'(\varphi)>0$, $i=1,2$ such that for any $n\geq N$, $1 \leq k\leq e^{\log^2q_{n}}$ and $x\in I\in \mathcal{P}_{n,k}$ satisfying $\text{dist}(x,\partial I)\geq e^{-2\log^2q_{n}}$ it holds
\be
|S_l(\varphi')(x)| \leq C_1'\, k_0 \max\{q_{n}\log q_{n},\,x_{\min,q_{n}}^{-1}\}+C_1' \max\{q_{n}\log q_{n},\,(R_\alpha^{k_0 q_n} x)_{\min,l_0}^{-1}\},
\label{eq:bound_dk}
\ee
and
\be
|S_l(\varphi'')(x)| \leq C_2'\, k_0 \,x_{\min,q_{n}}^{-2}+C_2' \,\max\{q_n^2,\,(R_\alpha^{k_0 q_n}x)_{\min,l_0}^{-2}\},
\label{eq:bound_ddk}
\ee
for $0\leq l\leq kq_n$, where $k_0=\lfloor l/q_{n}\rfloor$ and $l_0=l-k_0q_{n}$.

Moreover, $\exists C_3'=C_3'(\varphi)>0$ such that, for any $x\in I\in \mathcal{P}_{n,k}$, if $k_0\geq 1$, then
\be 
S_{l}(\varphi'')(x)\geq C_3'\, \sum_{m=0}^{k_0-1} (R_\alpha^{mq_{n}+l_0}x)_{\min,{q_{n}}}^{-2}+C_3'\, (R_\alpha^{k_0 q_n}x)_{\min,{l_0}}^{-2},
\label{eq:convexity_k_general}
\ee
and if additionally $\text{dist}(x,\partial I)\geq e^{-2\log^2q_{n}}$, then
\be 
S_{l}(\varphi'')(x)\geq C_3'\, k_0\, x_{\min,{q_{n}}}^{-2}+C_3'\, (R_\alpha^{k_0 q_n}x)_{\min,{l_0}}^{-2},
\label{eq:convexity_k}
\ee
\end{lemma}

\begin{proof}
From Lemma \ref{lemma:rk5}, for $\text{dist}(x,\partial I)\geq {\norm{q_{n}\alpha}}^{1/2}$, it holds 
\[
\tfrac{1}{2}x_{\min,q_{r_n}}^\pm\leq (R_\alpha^{mq_{n}}x)_{\min,q_{r_n}}^\pm \leq \tfrac{3}{2}x_{\min,q_{r_n}}^\pm,
\]
for any $1\leq m\leq k$.
Hence we can apply Lemma \ref{lemma:bounds_d} to $\{R_\alpha^{k_0 q_n+m} x\}_{m=0}^{l_0-1}$ and to each $q_{n}$-sub-orbit of $\{R_\alpha^{m} x\}_{m=0}^{k_0q_n-1}$, and \eqref{eq:bound_dk}--\eqref{eq:bound_ddk} follow.

Note that \eqref{eq:convexity_k_general} and Lemma \ref{lemma:rk5} imply \eqref{eq:convexity_k} after renaming the constant $C_3'$. To prove \eqref{eq:convexity_k_general}, consider $l\geq q_{n}$. Then
\[
S_{l}(\varphi'')(x)\geq S_{l_0}(\varphi'')(R_\alpha^{k_0q_{n}}x)+ \sum_{m=0}^{k_0-1} S_{q_n}(\varphi'')(R_\alpha^{mq_{n}}x).
\]
We can bound each term using Lemma \ref{lemma:bounds_d}. Since the minimum distance to the singularity of each sub-orbit is upper bounded by $q_{n}^{-1}$, such $C_3'$ exists for $k_0\geq 1$.
\end{proof}

Fix $\gamma/10>\eps>0$. Fix $n$ and define 
\[
\Gamma_{n,k}:=\{x\in\cup_{I\in\mathcal{P}_{n,k}}I:S_{kq_n}(\varphi')(x)=0\}.
\]

\begin{lemma}
\label{lemma:location_min_k}
There exists $N=N(\varphi)\geq 1$ such that for any $n\geq N$ and $1 \leq k\leq e^{\log^2q_{n}}$, $\Gamma_{n,k}$ consists of $q_{n}$ points, each in one of the intervals of $\mathcal{P}_{n,k}$ and there exists $\hat\delta=\hat\delta(\varphi)>0$ such that 
\be
\frac{\hat \delta}{ 2q_{n}\log q_{n}}\leq y_{\min,kq_n}^-\leq \frac{2}{\hat \delta q_{n}\log q_{n}},
\label{eq:location_min_k}
\ee
for any $y\in\Gamma_{n,k}$, $1 \leq k \leq e^{\log^2q_{n}}$. Moreover, for $\hat \delta$ small enough, there exists $C_4'=C_4'(\varphi)>0$ such that for any $x$ with $x_{\min,kq_n}^-$ outside of the range \eqref{eq:location_min_k} it holds
\be 
|S_{kq_n}(\varphi')(x)|\geq C_4'k\, q_{n}\log q_{n}.
\label{eq:lower_bound_dk}
\ee

\end{lemma}
\begin{proof}
Fix $I\in\mathcal{P}_{n,k}$. The function $S_{kq_n}(\varphi)(x)$ is $C^2$ on $I$ and, by \eqref{eq:convexity_k_general}, $S_{kq_n}(\varphi')$ is strictly monotone and has exactly one zero. Note that then \eqref{eq:location_min_k} follows from \eqref{eq:lower_bound_dk}.

The argument is the same as Lemma \ref{lemma:location_min} with two differences. First, the Denjoy--Koksma inequality is applied to $\varphi'$ restricted away from singularities at a slightly smaller scale rather to include each of the $k$ sub-orbits. Second, the bounds have an extra factor $k$.

Fix $1>\hat\delta>0$ and let $\tilde\varphi'$ denote the restriction of $\varphi'$ to $[\frac{1}{2\hat \delta q_{n}\log q_{n}},1-\frac{1}{2\hat \delta q_{n}\log q_{n}}]$. 
Applying the Denjoy--Koksma inequality to $\tilde\varphi'$, we get that there exists $C_4'=C_4'(\varphi)>{4}/{c_0}$ such that 
\[
\begin{cases}
\left|S_{q_{n}}(\tilde\varphi')(x)\right|\leq (2c_0+C_4')q_{n}\log q_{n},\\
-S_{q_{n}}(\tilde\varphi')(x)\geq (c_0/2-C_4'\hat\delta)q_{n}\log q_{n},
\end{cases}
\]
for $n$ large enough. 

Let $n$ be large enough such that $\frac{1}{\hat \delta q_{n}\log q_{n}}\leq\frac{1}{2q_{n}}.$ 

If $x_{\min,q_{n}}\geq({\hat \delta q_{n}\log q_{n}})^{-1}$, from Lemma \ref{lemma:rk5} we have, for $n$ large enough, $x_{\min,kq_n}>(2{\hat \delta q_{n}\log q_{n}})^{-1}$. Therefore, 
\[ 
-S_{kq_n}(\varphi')(x)=-S_{kq_n}(\tilde\varphi')(x)\geq k(c_0/2-C_4'\hat\delta)q_{n}\log q_{n}.
\]

For $x_{\min,q_{n}}^-<{\hat \delta}({ q_{n}\log q_{n}})^{-1}$,
\[
S_{{kq_n}}(\varphi')(x)\geq \sum_{l=0}^{k-1}\left[\tfrac{c_0}{2}(R_\alpha^{lq_{n}}x)_{\min,q_{n}}^{-1} -4q_nc_0-O(|\log (R_\alpha^{lq_{n}}x)_{\min,q_{n}}|)\right]-k(2c_0+C_4')q_{n}\log q_{n},
\]
and, for $x_{\min,q_{n}}^+<{\hat \delta}({ q_{n}\log q_{n}})^{-1}$,
\[
\begin{aligned}
S_{{kq_n}}(\varphi')(x)&\leq \sum_{l=0}^{k-1}\left[-{2c_0}(R_\alpha^{lq_{n}}x)_{\min,q_{n}}^{-1} +4q_nc_0+O(|\log (R_\alpha^{lq_{n}}x)_{\min,q_{n}}|)\right]\\&+k({c_0}/{2}-C_4'\hat\delta)q_{n}\log q_{n}.
\end{aligned}
\]

If $e^{-2\log^2q_{n}}<x_{\min,q_{n}}^\pm<{\hat \delta}({ q_{n}\log q_{n}})^{-1}$, Lemma \ref{lemma:rk5} again implies \[\max_{l\in\{0,\dots,k-1\}}(R_\alpha^{lq_{n}}x)_{\min,q_{n}}^\pm<2x_{\min,q_{n}}^\pm,\] and if $e^{-2\log^2q_{n}}\geq x_{\min,q_{n}}^\pm,$\[\max_{l\in\{0,\dots,k-1\}}(R_\alpha^{lq_{n}}x)_{\min,q_{n}}^\pm\leq x_{\min,q_{n}}^\pm+l\norm{\alpha q_n}\leq e^{-2\log^2q_{n}}+q_n^{-1}e^{-2\log^2q_{n}}\leq (2\hat \delta q_{n}\log q_{n})^{-1},\]for $n$ large enough. Applying this to the previous inequalities and taking $\hat\delta$ small enough gives \eqref{eq:lower_bound_dk}.

\end{proof}

Given $n\geq 1$, $1 \leq k\leq e^{\log^2q_{n}}$, recall $\Gamma_{n,k}:=\{x\in\cup_{I\in\mathcal{P}_{n,k}}I:S_{kq_n}(\varphi')(x)=0\}$. Define 
\be
\label{eq:defTk}
t_{n,k}=\max_{y\in\Gamma_{n,k}}S_{kq_n}(\varphi)(y)+2\eps.
\ee

\begin{lemma}
There exists $N=N(\varphi)\geq 1$ such that the following holds for $n\geq N$ and $1 \leq k\leq e^{\log^2q_{n}}$:
\label{lemma:aligned_k}
\be
t_{n,k}<\min_{y\in\Gamma_{n,k}}S_{kq_n}(\varphi)(y)+3\eps.
\label{eq:aligned_k}\ee
\end{lemma}
\begin{proof}

We will use the same ideas as in the proof of Lemma \ref{lemma:aligned}. 

Let $I\in\mathcal{P}_{n,k}$ be the rightmost interval of $\mathcal{P}_{n,k}$ when ordered in $(0,1)$, and let $\{y\}=I\cap\Gamma_{n,k}$ Every point in the orbit $\{y-l\alpha\}_{l=0}^{q_n-1}$ lies in an interval of $\mathcal{P}_{n,k}$. This is clear by looking at the proof of Lemma \ref{lemma:partition} and noting that the left and the right boundary points of the elements of $\mathcal{P}_{n,k}$ form two $q_n$-orbits. By the same reason, for $0\leq l \leq {q_n-1}$, 
\[
(y-l\alpha)_{\min,q_n}^-=y_{\min,q_n}^-=\norm{y},\quad (y-l\alpha)_{\min,kq_n}^-=y_{\min,kq_n-l}^-=y_{\min,kq_n}^-.
\]

Since $S_{kq_n}(\varphi')(y)=0$, we write
\[
S_{kq_n}(\varphi')(y-l\alpha)= S_{kq_n}(\varphi')(y-l\alpha)-S_{kq_n}(\varphi')(y).
\]

Let $n$ be large enough such that $\norm{kq_n\alpha}\leq\frac{k}{q_{n+1}}\leq({q_{n}e^{2\log^2q_{n}}})^{-1}\leq\frac{\hat\delta}{4q_{n}\log q_{n}}.$
We can combine the bound \eqref{eq:bound_ddk} with \eqref{h} by applying the mean value theorem. We obtain
\[
|S_{kq_n}(\varphi')(y-l\alpha)|\leq |S_{l}(\varphi'')(\theta)|\norm{kq_n\alpha} \leq \frac{C_2'k\,\theta_{\min,kq_n}^{-2}}{q_{n+1}}\leq \frac{C_2'\theta_{\min,kq_n}^{-2}}{q_{n}e^{2\log^2q_{n}}},
\]
for some $\theta\in(y-l\alpha,y-l\alpha+kq_n\alpha)$ or $(y-l\alpha+kq_n\alpha,y-l\alpha)$. 

Therefore, using \eqref{eq:location_min_k} and the bound on $\norm{kq_n\alpha}$, we have
\[
e^{- 2\log ^2 q_n}<\frac{\hat\delta q_{n}\log q_{r_n}}{4}\leq q_{n}\log q_{n}\left(\frac{2}{\hat\delta} +\frac{\hat\delta}{4}\right)^{-1}\leq \theta_{\min,kq_n}^{-1}\leq \frac{4q_{n}\log q_{n}}{\hat\delta}.
\]

Again, following the steps of Lemma \ref{lemma:aligned}, by the convexity result \eqref{eq:convexity_k}, if $y'\in\Gamma_{n,k}$ lies in the same interval of $\mathcal{P}_{n,k}$ as $y-l\alpha$, then
\[ 
\begin{aligned}
|y'-y+l\alpha|&\leq \frac{C_2'}{q_{n}e^{2\log^2q_{n}}}\left(\frac{4q_{n}\log q_{n}}{\hat\delta}\right)^2 \frac{1}{C_3'}\left(\frac{2}{\hat\delta q_{n}\log q_{n}}\right)^2\leq \tilde C \left(q_{n}e^{2\log^2q_{n}}\right)^{-1},
\end{aligned}
\]
for some $\tilde C(\varphi,\hat\delta)>0$. Thus, by \eqref{eq:bound_dk} and \eqref{h}, the main value gives
\[
\begin{aligned}
|S_{q_{r_n}}(\varphi)(y')-S_{q_{r_n}}(\varphi)(y-l\alpha)|&\leq \tilde C \left(q_{n}e^{2\log^2q_{n}}\right)^{-1}C_1'k\frac{2q_{n}\log q_{n}}{\hat\delta}=O(\log q_{n}e^{-\log^2q_{n}}),
\end{aligned}
\]
and
\[
|S_{kq_n}(\varphi)(y)-S_{{kq_n}}(\varphi)(y-l\alpha)|\leq \norm{kq_n\alpha}C_1'\frac{2q_{n}\log q_{n}}{\hat\delta}= O(\log q_{n}e^{-\log^2q_{n}}).
\]

Altogether, $|S_{q_{r_n}}(\varphi)(y)-S_{q_{r_n}}(\varphi)(y')|$ decays asymptotically with $n$, for any $y,y'\in\Gamma_{r_n}$. So, by definition of $t_n$, for $n$ large enough we obtain
\[
t_n < S_{q_{r_n}}(\varphi)(y)+3\eps \quad \text{for all } y\in\Gamma_{r_n}.
\]
\end{proof}

\subsection{Computation of the correlations}

The argument parallels Section \ref{sec:7.4} but now runs over $k$ scales simultaneously. For each pair $(n,k)$, we define a region $H_{n,k}$ where $N(x,t_{n,k}) $ is exactly $kq_n$ and we have good estimates on $S_{kq_n}(\varphi')$, so that the convexity results of Section \ref{sec:convexity} apply. Lemma \ref{lemma:endpoints_k} gives the correlation gain on $H_{n,k}$, Lemma \ref{lemma:final_k} handles the complementary region $F_{n,k} \setminus H_{n,k}$ using Lemma \ref{lemma:convex3}, and the proof of Theorem \ref{thm:sum_cor} assembles these into a divergent series.

Let $A=\T\times [0,\eps]$.
Let \[H_{n,k}:=\{x\in\cup_{I\in\mathcal{P}_{n,k}}I:N(x,t_{k,n}+2\eps)=kq_n,S_{N(x,t_{k,n})}(\varphi')\leq 0\}.\]
By the definition of $t_{n,k}$ and Lemma \ref{lemma:aligned_k}, $H_{n,k}$ is the union of $q_{n}$ intervals, one on each element of $\mathcal{P}_{n,k}$, and bounded to the right by points of $\Gamma_{n,k}$. 

Let $E_{n,k}=H_{n,k}\times[0,\eps]$. 

\begin{lemma} \label{lemma:endpoints_k} There exist $\eps^*=\eps^*(\varphi)>0$ and $N=N(\varphi)\geq 1$ such that for any $0<\eps<\eps^*$, $n\geq N$, $1 \leq k\leq e^{\log^2q_{n}}$ and any $t\in[0,\eps]$
\[
\mu(\phi_{t_{n,k}+t}E_{n,k} \cap \phi_\eps A)-\mu(\phi_{t_{n,k}+t}E_{n,k} \cap A)\geq\frac{\eps^4}{\sqrt{k}\log q_{n}}.
\]
\end{lemma}
\begin{proof}
The factor $\sqrt{k}$ in the denominator of the lower bound reflects the fact that, at scale $k$, the convexity of $S_{kq_n}(\varphi)$ is of order $k q_n^2 \log^2 q_n$ while the first derivative bound is of order $\sqrt{k} q_n \log q_n$. Applying Lemmas \ref{lemma:convex1} and \ref{lemma:convex2} then gives the $1/\sqrt{k}$ decay.

We first claim that there exists $C_5'=C_5'(\varphi)>0$ such that $\forall x \in H_{n,k}$
\[
|S_{kq_n}(\varphi')(x)|\leq C_5' \sqrt{k}\, q_{n}\log q_{n},
\]
for $n$ large enough, for every $1 \leq k \leq e^{\log^2q_{n}}$.

Indeed, fix $I\in\mathcal{P}_{n,k}$ and let $x\in I\cap H_{n,k}$, $\{y\}=I\cap \Gamma_{n,k}$. First, notice that $x$ is to the left of $y$, and by the same argument as in the proof of Lemma \ref{lemma:lower_bound_d}, the estimates \eqref{eq:location_min_k} and \eqref{eq:lower_bound_dk} imply \[e^{-\log^2 q_n}<\hat \delta\,(2{\, q_{n}\log q_{n}})^{-1}\leq x_{\min,kq_n}^- \leq 2\,({\hat \delta\, q_{n}\log q_{n}})^{-1}+{\gamma}\,(C_4' \,k  \,q_{n}\log q_{n})^{-1}.\] Hence $x_{\min,kq_n}=x_{\min,kq_n}^-$ and it follows from \eqref{eq:bound_ddk} and \eqref{eq:convexity_k} that 
\[ 
\tilde C^{-1}\, k\, q_{n}^2\log ^2q_{n} \leq S_{kq_n}(\varphi'')\leq \tilde C\, k\, q_{n}^2\log ^2q_{n}
\] 
on $(x,y)$, for some $\tilde C(\varphi)>0$. Now, the Taylor expansion about $y$ of $f=S_{kq_n}(\varphi)$ yields
\[
\gamma\geq f(x)-f(y)=f''(\theta)\tfrac{(y-x)^2}{2}\geq \tfrac{1}{2}\tilde C^{-1}\, k\, q_{n}^2\log ^2q_{n}\,(y-x)^2,
\]
for some $\theta\in(x,y)$. Therefore 
\[
f'(x)=f''(\theta)(x-y)\geq -\tilde C\, k\, q_{n}^2\log ^2q_{n}\,\sqrt{\frac{2\tilde C\gamma}{kq_{n}^2\log ^2q_{n}}}=:-C_5' \sqrt{k}\, q_{n}\log q_{n}.
\]

Let $I\in\mathcal{P}_{n,k}$ and let $J(t)= I\cap\{t_{n,k}+t-S_{kq_n(\varphi)}(x)\in(-\eps,2\eps)\}\cap  H_{n,k}$. From the definition of $t_{n,k}$ and $H_{n,k}$, $J(t)$ is an interval satisfying 
\[
\inf_{J(t)}[ t_{n,k}+t-S_{kq_n(\varphi)}(x)]=-\eps,\quad\sup_{J(t)}[ t_{n,k}+t-S_{kq_n(\varphi)}(x)]=2\eps .
\]

By the estimate we just showed we can apply Lemmas \ref{lemma:convex1} and \ref{lemma:convex2} to $f(x)=t_{n,k}+t-S_{kq_n}(\varphi)(x)$ with $k$, $M=q_{n}\log q_{n}$, and $C=\max\{C_5',\tilde C\}$. We get
\[
\mu\left[\phi_{t_{n,k}+t}(J(t)\times[0,\eps])\cap \phi_\eps A\right]-\mu\left[\phi_{t_{n,k}+t}(J(t)\times[0,\eps])\cap A\right]>\frac{\eps^4}{\sqrt{k}\,q_{n}\log q_{n}},
\]
for every $t\in[0,\eps]$.
Summing over the $q_{n}$ intervals of $\mathcal{P}_{n,k}$ gives the result.
\end{proof}

Let $\hat\delta>\delta>0$, and define
\[
F_{n,k}'=F_{n,k}'(\delta):=\mathbb{T}\setminus\bigcup_{m=0}^{kq_n-1} R_\alpha^{-m}\Big(\frac{-\delta}{kq_{n}\log q_{n}},\frac{\delta}{kq_{n}\log q_{n}}\Big).
\]
For $n$ large such that $\norm{q_n\alpha}< \tfrac{\delta}{kq_{n}\log q_{n}}$, $F_{n,k}'$ is the union of $q_n$ intervals, one in each element of $\mathcal{P}_{n,k}$. Let 
\[
F_{n,k}=\{x\in\cup_{I\in\mathcal{P}_{n,k}}I:N(x,t_n+\eps)\geq (k-1)q_n\}.
\]

\begin{lemma} There exist $N=N(\varphi)\geq 1$ such that for any $n\geq N$, $1 \leq k\leq e^{\log^2q_{n}}$: $F_{n,k}'\subset F_{n,k}$. \label{lemma:rk6}
\end{lemma}
\begin{proof}
Note that $N(y,t_{n,k})=kq_n$ $\forall y\in\Gamma_{n,k}$ and, by \eqref{eq:bound_dk} and the MVT,
\[
\begin{aligned}
N(x,t_{n,k})&\geq N(y,t_{n,k})-\frac{1}{\gamma}(S_{N(x,t_{n,k})}(\varphi)(x)-S_{N(x,t_{n,k})}(\varphi)(y))-1\\&\geq kq_n-\frac{2C_1' k}{\delta}\log q_{n}-2\geq (k-1)q_n,
\end{aligned}
\]
for any $x$ in the same component of $F_{n,k}'$ as $y$ and for $n$ large enough, where we have used $\frac{\delta}{kq_{n}\log q_{n}}\geq\frac{\delta}{q_{n}\log q_{n}}e^{-\log^2 q_n}\geq e^{-2\log^2 q_n}$.
\end{proof}

\begin{lemma} Let $H_{n,k}':=F_{n,k}\setminus H_{n,k}$. Then exists $N=N(\varphi)\geq 1$ such that for any $n\geq N$, $2 \leq k\leq e^{\log^2q_{n}}$, and $t\in[0,\eps]$
\[
\mu\left(\phi_{t_{n,k}+t}(H_{n,k}'\times[0,\eps])\cap \phi_\eps A\right)\geq\mu\left(\phi_{t_{n,k}+t}(H_{n,k}'\times[0,\eps])\cap A\right).
\]
Therefore, there exist $\eps^*>0$ such that for any $0<\eps<\eps^*$ and any $t\in[0,\eps]$
\be
\mu\left[\phi_{t_{n,k}+t}A\cap \phi_\eps A\right]-\mu\left[\phi_{t_{n,k}+t}A\cap A\right]\geq \frac{\eps^4}{2\sqrt{k}\log q_{n}}.
\label{eq:def_bounds_k}
\ee
\label{lemma:final_k}
\end{lemma}

\begin{proof}
From the definition of $F_{n,k}$, \eqref{eq:convexity_k_general} applies. Therefore, similarly to Lemma \ref{lemma:rk1}, on any $I\in \mathcal{P}_{n,k}$ we have that, for any $t>0$, if $\min_{I} S_l(\varphi)<t-\gamma,$ the set $ \{x\in I:-\eps<t-S_l(\varphi)<2\eps\},$ is either empty or the union of at most two disjoint intervals, one in which $S_l(\varphi')<0$ and one in which $S_l(\varphi')>0$.
Moreover, on both intervals $\inf S_l(\varphi)=t-2\eps$. 

From our choice of $\eps$, $t$, $H_{n,k}$ contains exactly the interval with $l=kq_n$ and $S_l(\varphi')<0$. Therefore for $l<kq_n$ all these intervals are contained in $H_{n,k}'$, as well as the interval for which $l=kq_n$ and $S_l(\varphi')>0$.
The convexity of $S_l(\varphi)$ in $F_{n,k}$ allows us to apply Lemma \ref{lemma:convex3} to each of these intervals. This gives the first result.

To prove \eqref{eq:def_bounds_k}, recall from Lemma \ref{lemma:rk6} that $F_{n,k}'\subset F_{n,k}$, which implies  
\[
\lambda (\T\setminus F_{n,k})=q_n\frac{2\delta}{k\,q_n\log q_n}+(k-1)q_n\norm{q_n\alpha}\leq \frac{2\delta}{k\,\log q_n}+\frac{k q_n}{q_{n+1}}\leq \frac{\eps^3}{2\sqrt{k}\log q_{n}},
\]
for $0<\delta<\delta^*$, $\delta^*(\eps)>0$ small enough, and $n$ large enough, where we have used \eqref{h} to bound the second term of the sum: $\frac{kq_n}{q_{n+1}} \leq e^{-2\log^2 q_n}$.

Therefore, putting together these results and Lemma \ref{lemma:endpoints_k}, and noting
\[
A=[H_{n,k}'\cup H_{n,k}\cup (\T\setminus F_{n,k})]\times[0,\eps],
\]
the result \eqref{eq:def_bounds_k} holds.
\end{proof}

\begin{proof}[Proof of Theorem \ref{thm:sum_cor}]
Let $\alpha$ satisfy \eqref{h} and consider $N$ given by \ref{lemma:final_k}. Then, for $n\geq N$, $2 \leq k\leq e^{\log^2q_{n}}$, and $t\in[0,\eps]$, \eqref{eq:def_bounds_k} holds. Noting $\mu(A)=\mu(\phi_\eps A)$ and writing \eqref{eq:def_bounds_k} as a difference of correlations, we get that at least one of the following hold
\[
    |\text{Cor}_{t_{n,k}+t-\eps}(\chi_A,\chi_{A})|\geq\frac{\eps^4}{4\sqrt{k}\log q_{n}},\quad \text{or}\quad
    |\text{Cor}_{t_{n,k}+t}(\chi_A,\chi_{A})|\geq \frac{\eps^4}{4\sqrt{k}\log q_{n}}.
\]
Squaring and integrating over $t \in [0,\varepsilon]$ gives
\[
\int_{t_{n,k}-\varepsilon}^{t_{n,k}+\varepsilon} |\mathrm{Cor}_t(\chi_A,\chi_A)|^2\,dt \geq \frac{\varepsilon^9}{16k\log^2 q_n}.
\]

To see that the intervals $[t_{n,k} - \varepsilon, t_{n,k} + \varepsilon]$ are pairwise disjoint, note that by definition we can bound the difference $t_{n,k+1} - t_{n,k}$ below by the minimum of $S_{q_n}(\phi)$. Therefore 
\[
t_{n,k+1} - t_{n,k} \geq \gamma q_n - 2\varepsilon \geq 2\varepsilon,
\]
for all $1 \leq k \leq e^{\log^2 q_n}$. Since $k \leq e^{\log^2 q_n} \ll a_{n+1}$, the same applies to $t_{n+1,1} - t_{n,k}$ for any such $k$.

Since the intervals $[t_{n,k}-\varepsilon, t_{n,k}+\varepsilon]$ are pairwise disjoint over all $(n,k)$ with $n \geq N$ and $2 \leq k \leq e^{\log^2 q_n}$, summing gives
\[
\int_0^\infty |\mathrm{Cor}_t(\chi_A,\chi_A)|^2\,dt \geq \sum_{n \geq N} \sum_{k=2}^{e^{\log^2 q_n}} \frac{\varepsilon^9}{16k\log^2 q_n} \geq \sum_{n \geq N} \frac{\varepsilon^9}{32} = \infty,
\]
where we used $\sum_{k=2}^{K} k^{-1} \geq (\log K)/2$. This completes the proof.
\end{proof}

\appendix
\section{Proof of Proposition \ref{prop:C1norm}}
\label{app:norm_estimates}

In this appendix we establish the $C^1$ estimate
\[
\|\tilde{f}\|_{C^1(\mathcal{M}')}  =O((\log t)^{1+\nu}),
\]
for $\tilde{f} = f \circ q^{-1}\circ \Phi^{-1}$, where $f \in C^1_c(A)$ is a test function produced by Theorem \ref{thm:IETs}. The main difficulty is that the differential $d\Phi^{-1}$ blows up near the saddle singularities of the locally Hamiltonian flow. The condition \eqref{eq:IET_dist} keeps the support of $f$ away from these singularities in a quantitative way that we make precise below.

Recall that the flow $(\phi_t^{\mathcal{M}})_{t \in \mathbb{R}}$ on $\mathcal{M}$ preserves the area form $\omega$ and it is determined by a smooth vector field $W$. Equip $\mathcal{M}$ with a Riemannian metric consistent with $\omega$. 
\begin{lemma}\label{lemma:A1} For every point $p=\Phi(x,s)\in\mathcal{M}'$ with $W(p)\neq 0$ the following pointwise bound holds
\begin{equation}
\label{eq:df_bound}
\|d\tilde{f}(p)\| \leq \frac{|\partial_s f|}{\|W\|} + \|W\|\,|\partial_x f| + |\alpha(x,s)|\,\|W\|\,|\partial_s f|,
\end{equation}
for some real valued function $\alpha(x,s)$.
\end{lemma}

\begin{proof}
Recall from Section \ref{sec:susflows} that $\Phi : X \to \mathcal{M}'$ is a measure-preserving isomorphism intertwining $(\phi_t)$ with $(\phi_t^{\mathcal{M}})$. It is defined by $\Phi(x,0) \in \Sigma$ (the base transversal, parametrized by the local Hamiltonian $H$) and $\Phi(x,s) = \phi_s^{\mathcal{M}}(\Phi(x,0))$ for $x\in\T$ and $0 \leq s$. At a smooth point $(x,s)$ with $x \notin D_0$, the differential of $\Phi$ satisfies
\[
d\Phi_{(x,s)}(\partial_s) = W(\Phi(x,s)), \qquad d\Phi_{(x,s)}(\partial_x) = J(x,s),
\]
where $W$ is the Hamiltonian vector field generating $(\phi_t^{\mathcal{M}})$ and
\[
J(x,s) :=\partial_\epsilon\big|_{\epsilon=0}\phi_s^{\mathcal{M}}(\Phi(x+\epsilon,0))
\]
is the transverse variation field.

At each $p \in \mathcal{M}'$ with $W(p)\neq 0$, define the orthonormal frame
\[
e_1(p) :=\frac{W(p)}{\|W(p)\|}, \qquad e_2(p) :=e_1(p)^\perp,
\]
where $e_2$ is the unit vector orthogonal to $e_1$ with $\omega(e_1,e_2)=1$. We decompose the variation field in this frame (see Figure \ref{fig:variation_decomposition})
\[
J(x,s) = \alpha(x,s)\,W(\Phi(x,s)) + \beta(x,s)\,e_2(\Phi(x,s))= \alpha(x,s)\|W\|\,e_1 + \beta(x,s)\,e_2.
\]
Since $\Phi$ is area-preserving and $\Sigma$ is parametrized by $H$ with $dH = \omega(W,\cdot)$, we have at $s = 0$,
\[
\omega(W(\Phi(x,0)),J(x,0)) = \omega\bigl( W(\Phi(x,0)),\partial_x\Phi(x,0)\bigr) = dH\bigl(\partial_x\Phi(x,0)\bigr) = 1,
\]
and, by area preservation $(\phi_s^{\mathcal{M}})^*\omega = \omega$, for $s > 0$ it holds
\[
\omega(W(\Phi(x,s)),J(x,s)) = \omega(W(\Phi(x,0)),J(x,0)) = 1.
\]
Substituting the decomposition of $J$ we get $\beta(x,s)= 1/\|W\|$:
\[
1 = \omega\left(\|W\|\,e_1,\alpha\|W\|\,e_1 + \beta\,e_2\right) = \beta\|W\|\,\omega(e_1, e_2) = \beta\|W\|.
\]
The variation field is therefore
\[
J(x,s) = \alpha(x,s)\|W\|\,e_1 + \frac{1}{\|W\|}\,e_2.
\]

\begin{SCfigure}[0.82][htbp]
\centering
\includegraphics[width=0.5\textwidth]{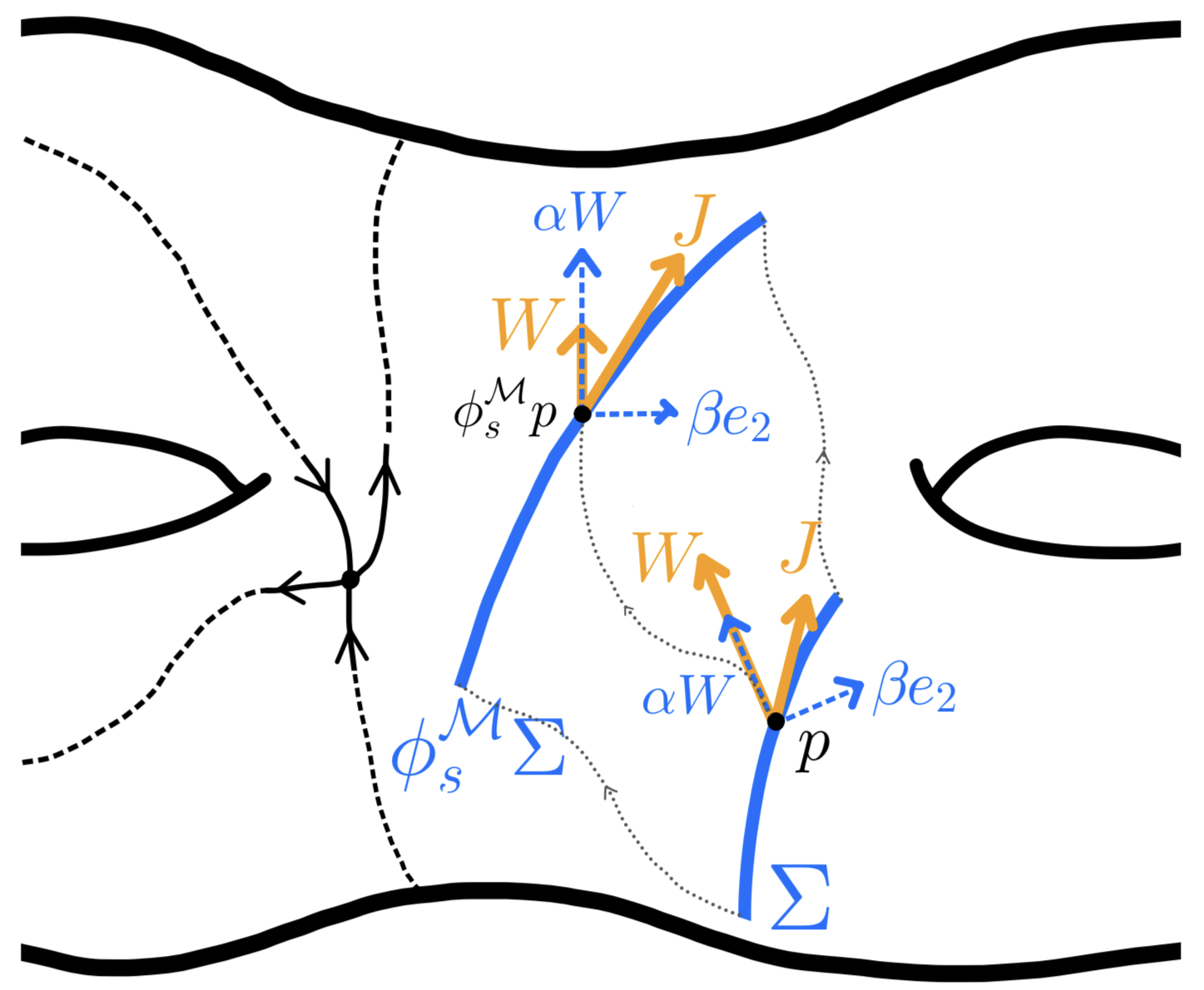} 
\caption{The decomposition of the transverse variation field $J$ along a trajectory of the locally Hamiltonian flow. Starting from the base transversal $\Sigma$ at a point $p$, the variation vector $J(x,s) = d\phi_s^{\mathcal{M}}(J(x,0))$ is tracked along the continuous orbit $\phi_s^{\mathcal{M}}p$, with $s>0$. In the moving orthonormal frame $\{e_1= W/\|W\|, e_2= e_1^\perp\}$, the variation field splits into a longitudinal component $\alpha(x,s)W$ representing a shear, and a transverse component $\beta(x,s)\,e_2 = 1/\|W\|\,e_2$.}
\label{fig:variation_decomposition}
\end{SCfigure}

Applying the the chain rule to $f = \tilde{f} \circ \Phi$ yields
\begin{align*}
\partial_s f &= d\tilde{f}(W) = \|W\|\,d\tilde{f}(e_1), \\
\partial_x f &= d\tilde{f}(J) = \alpha\|W\|\,d\tilde{f}(e_1) + \frac{1}{\|W\|}\,d\tilde{f}(e_2) = \alpha\,\partial_s f + \frac{1}{\|W\|}\,d\tilde{f}(e_2).
\end{align*}
Inverting, we get
\[
d\tilde{f}(e_1) = \frac{\partial_s f}{\|W\|},
\qquad d\tilde{f}(e_2) = \|W\|\bigl(\partial_x f -\alpha(x,s)\,\partial_s f\bigr).
\]
Since $\{e_1,e_2\}$ is orthonormal, we obtain the pointwise bound
\[
\|d\tilde{f}(p)\|  \leq |d\tilde{f}(e_1)| + |d\tilde{f}(e_2)| \leq \frac{|\partial_s f|}{\|W\|} + \|W\|\,|\partial_x f| + |\alpha(x,s)|\,\|W\|\,|\partial_s f|.
\]
\end{proof}

\begin{remark}
The three terms in \eqref{eq:df_bound} reflect three distinct geometric mechanisms. 

The first term is the blowup of $d\tilde{f}(e_1)$ due to the vanishing of $\|W\|$ near the saddles: the flow slows down near $p_i$, so a given increment $\partial_s f$ in suspension time corresponds to a large displacement in the flow-direction derivative on $\mathcal{M}'$. This is controlled by a lower bound on $\|W\|$. 

The second is the blowup of $d\tilde{f}(e_2)$: large transverse oscillations $\partial_x f$ are amplified by $\|W\|$, but since $\|W\|$ is globally bounded this term is controlled directly. 

The third term is the shear term: the variation field $J$ has a component $\alpha\|W\|e_1$ in the flow direction, so that perturbing the initial condition by $\epsilon$ displaces the orbit both transversally and longitudinally. However, the slow $s$-variation $\|\partial_s f\|_\infty = O((\log t)^{-1+\nu})$ in Theorem \ref{thm:IETs} will imply that the second term dominates the norm.
\end{remark}

The coefficient $\alpha(x,s)$ measures the longitudinal displacement between neighboring orbits $\Phi(x,0)$ and $\Phi(x+\epsilon,0)$ at time $s$. We bound $|\alpha(x,s)|$ using \eqref{eq:IET_dist}.

\begin{lemma} \label{lemma:A3} 
For $(x,s) \in A = [a,b]\times[0,t_0]$ satisfying 
\eqref{eq:IET_dist} it holds
\begin{equation}
\label{eq:alpha_bound_final}
|\alpha(x,s)| = O\left((\log t)^{1+\nu}\log\log t\right).
\end{equation}
\end{lemma} 

\begin{proof}
Recall that $\alpha(x,s)$ is the longitudinal component of the variation field $J(x,s) = \partial_\epsilon|_{\epsilon=0}\, \phi_s^{\mathcal{M}}(\Phi(x+\epsilon,0))$, measuring the drift in the flow direction between the orbits of $x$ and $x + \epsilon$ at time $s$.

Let $N=N(x,s)$ be the number of complete returns to $\Sigma$ before time $s$ and define \[\alpha_{\mathrm{exc}}(x,s) := \alpha(x,s) - \alpha(x, S_{N}(\varphi)(x)),\] to be the increment during the incomplete excursion starting from $\Phi(T^Nx,0)$.

Differentiating in $\epsilon$ and applying the chain rule, we get
\begin{align*}
J(x,s) &= \partial_\epsilon|_{\epsilon=0}\,\phi_{s}^{\mathcal{M}}(\Phi(x+\epsilon,0))= \partial_\epsilon|_{\epsilon=0}\,\phi_{s - S_N(\varphi)(x+\epsilon)}^{\mathcal{M}}(\Phi(T^N(x+\epsilon),0)) \\
&= d\phi_{s-S_N(\varphi)(x)}^{\mathcal{M}}(\Phi(T^Nx,0))\left[\partial_\epsilon|_{\epsilon=0}\,\Phi(T^N(x+\epsilon),0)\right]- S_N(\varphi')(x)\cdot W(\Phi(x,s)),
\end{align*}
where the first term is the variation field of $\phi_{s-S_N(\varphi)(x)}^{\mathcal{M}}$ started from the point $\Phi(T^N(x),0) \in \Sigma$, and the second term accounts for the shift in elapsed time, which displaces the orbit by $-S_N(\varphi')(x)$ units along the flow direction $W$. Projecting along the flow direction, the first term on the right contributes $\alpha_{\mathrm{exc}}(x,s)$ (the variation coefficient of the incomplete excursion starting from $\Phi(T^Nx,0)$), while the second term contributes $-S_N(\varphi')(x)$, giving
\begin{equation}
\label{eq:alpha_decomp_explicit}
\alpha(x,s) = \alpha\left(T^Nx,s-S_N(\varphi)(x)\right)-S_N(\varphi')(x).
\end{equation}

Since the iterates $T^k x$ for $k = 0, \ldots, N-1$ all satisfy $\mathrm{dist}(T^k x, D_0) \geq \delta$ by \eqref{eq:IET_dist}, and since $T$ acts by translation on each continuity interval, the iterates falling within distance $r \geq \delta$ of any fixed $y \in D_0$ are mutually separated by at least $\delta$. This well-separation (cf.\ Lemma \ref{lemma:BSD_IET}) gives that each singularity of $\varphi'$ contributes at most $O(\delta^{-1} \log N)$ to the Birkhoff sum. Summing over the finitely many points in $D_0$ yields
\[
|S_{N(x,s)}(\varphi')(x)| \lesssim \frac{\log N(x,s)}{\delta}.
\]

For the incomplete excursion, first note that $\Phi[0,1]=\Sigma$ is a closed curve, so $\alpha(\cdot,0)$ is uniformly bounded by a constant depending only on the choice of $\Sigma$ and the geometry of $\mathcal{M}'$. 
More generally, outside fixed neighborhoods of the saddles $p_i$, the vector field $W$ is bounded away from zeros, and trajectories can spend only a uniformly bounded amount of time outside these neighborhoods before returning to $\Sigma $, generating only bounded contributions to the variation field. The dominant contributions to $\alpha(x,s)$ arise from passages near the saddle points.

For any fixed point $p_i\in \mathcal{M}'$, let $\tau_i(h) \sim -C_i \log h$ be the passage time through a fixed ball around $p_i$ at energy level $h$ (cf.\ \cite{ravotti}). Since $\alpha_{\mathrm{exc}}$ measures the cumulative variation of the passage times of the orbit segment $\{\Phi(x,\sigma)\}_{\sigma\in[S_N(\varphi)(x),\,s]}$ with respect to the initial condition on $T^N x \in \Sigma$, we have
\[
|\alpha_{\mathrm{exc}}| \lesssim \max_i\left|\frac{d}{dx}\tau_i(h(x))\right| = \max_i \left|\tau_i'(h)\frac{dh}{dx}\right| \lesssim \max_i |\tau_i'(h)| \sim \max_i\frac{C_i}{h} \lesssim \frac{1}{\delta},
\]
where we used that $dh/dx$ is uniformly bounded, that the piece of orbit intersects each neighborhood at energy level $h \gtrsim \mathrm{dist}(T^N x, D_0) \geq \delta$ by \eqref{eq:IET_dist}, and that each excursion passes near each singularity at most once (cf. \cite{KanigowskiOkunevZelada2026}).

Combining \eqref{eq:alpha_decomp_explicit} with the two bounds above, and using $(\log N)/\delta \geq 1/\delta$ for $N$ large enough, we get
\[
|\alpha(x,s)| \lesssim \frac{\log N(x,s)}{\delta}.
\]
Substituting $N(x,s) \leq (\log t)^{1+\nu}$ and $\delta = (\log t)^{-1-\nu}$ from \eqref{eq:IET_dist}:
\[
|\alpha(x,s)| \lesssim \frac{(1+\nu)\log\log t}{(\log t)^{-(1+\nu)}} = (\log t)^{1+\nu} \log\log t.
\]
\end{proof}

\begin{lemma}\label{lemma:A4}
For every $ p \in \mathrm{supp}(\tilde{f})$, we have
\begin{equation}
\label{eq:W_bounds}
(\log t)^{-(1+\nu)/2} \lesssim \|W(p)\| \leq C_W,
\end{equation}
where $C_W > 0$ depends only on the geometry of $\mathcal{M}'$.
\end{lemma}

\begin{proof}
Since $\mathcal{M}'$ is compact and $W$ is a smooth vector field, $\|W\|$ is bounded above by some constant $C_W > 0$ on all of $\mathcal{M}'$.

For the lower bound, near a non-degenerate saddle $p_i$ there exist local coordinates $(u,v)$ centered at $p_i$ in which $H(u,v) = uv$ and $W \asymp u\,\partial_u - v\,\partial_v$ (see \cite{ravotti}), so on a level set $H = h$ one has by the AM-GM inequality
$
\|W(p)\| \gtrsim \sqrt{h}$ for $p$ on the level set $H = h.$

The transversal $\Sigma$ is parametrized by $H$, so a point $x \in \T$ at distance $\rho = \mathrm{dist}(x, D_0)$ from the nearest singularity corresponds to a level set with $h \gtrsim \rho$. Since level sets are preserved by the flow, for all $ s \in [0, \varphi(x))$
\[
\|W(\Phi(x,s))\| \gtrsim \sqrt{\mathrm{dist}(x, D_0)} .
\]
By \eqref{eq:IET_dist}, $\mathrm{dist}(T^k x, D_0) \geq \delta := (\log t)^{-1-\nu}$ for all $x \in [a,b]$ and $0\leq k \leq N(x, s)$, giving $\|W(\Phi(x,s))\| \gtrsim \delta^{1/2} = (\log t)^{-(1+\nu)/2}$ on the support. 
\end{proof}

We are now ready to prove Proposition \ref{prop:C1norm}.

\begin{proof}[Proof of Proposition \ref{prop:C1norm}]
We combine Lemmas \ref{lemma:A1}, \ref{lemma:A3} and \ref{lemma:A4} by substituting the bounds from \eqref{eq:IET_dbounds}, \eqref{eq:alpha_bound_final}, and \eqref{eq:W_bounds} into \eqref{eq:df_bound}.

From $\|\partial_s f\|_\infty = O((\log t)^{-1+\nu})$ and $\norm{W}\gtrsim(\log t)^{-(1+\nu)/2}$ we get
\[
\frac{|\partial_s f|}{\|W\|} \lesssim \frac{(\log t)^{-1+\nu}}{(\log t)^{-(1+\nu)/2}} = (\log t)^{(1+\nu)/2 - 1} = (\log t)^{(3\nu-1)/2}.
\]

Since $\|W\| \leq C_W$ uniformly on $\mathcal{M}'$ and $\|\partial_x f\|_\infty = O((\log t)^{1+\nu})$, we have
\[
\|W\|\,|\partial_x f| = O((\log t)^{1+\nu}).
\]

Lastly, \eqref{eq:alpha_bound_final}, $\|W\| \leq C_W$, and $\|\partial_s f\|_\infty = O((\log t)^{-1+\nu})$ give
\[
|\alpha|\,\|W\|\,|\partial_s f| \lesssim (\log t)^{1+\nu}\log\log t \cdot C_W \cdot (\log t)^{-1+\nu} = O((\log t)^{2\nu} \log\log t).
\]
For $\nu<1$, the three terms are $O((\log t)^{1+\nu}).$
Together with $\|\tilde{f}\|_\infty = \|f\|_\infty=1$, we get the desired control on the $C^1$ norm
$
\|\tilde{f}\|_{C^1(\mathcal{M}')} = \|\tilde{f}\|_\infty + \|d\tilde{f}\|_\infty = O((\log t)^{1+\nu}).
$
\end{proof}

\section{Proof of Condition \eqref{eq:IET2}} \label{app:H2}

We assume the reader is familiar with the zippered rectangle construction and the natural extension of the Rauzy--Veech and Zorich inductions, and refer to \cite{Viana2008} for a detailed exposition. We recall here only the definitions necessary to proceed with the proof.

The natural extension of the Zorich induction acts on the space
\[
\widehat{\mathcal{M}}_d^*(\mathfrak{R}) := \bigsqcup_{\pi \in \mathfrak{R}}\left(\Delta_{d} \times \{\pi\} \times \Theta_\pi\right),
\]
where $\Theta_\pi \subset \mathbb{R}^d$ is Veech's zippered rectangle cone, defined by
\[
\Theta_\pi := \left\{ \underline{\tau} \in \mathbb{R}^d \,:\, \sum_{\pi(j) < k} \tau_j < 0 < \sum_{j < k} \tau_j, \quad \forall\, 1 < k \leq d \right\}.
\]
A key property of $\Theta_\pi$ is that, denoting by $\Omega_\pi \in M_d(\mathbb{Z})$ the integer-valued antisymmetric matrix associated to $\pi$ (see \cite{Viana2008}), for any $\underline{\tau} \in \Theta_\pi$ the associated geometric height vector $\hat{\underline{h}} = -\Omega_\pi \underline{\tau}$ is strictly positive, i.e., $\hat{h}_i > 0$ for all $i$. Given a triple $(\underline{\lambda}, \pi, \underline{\tau})$, its area is defined by the scalar product $\mathrm{Area}(\underline{\lambda}, \pi, \underline{\tau}) = \langle \underline{\lambda}, \hat{\underline{h}} \rangle = \sum_{i=1}^d \lambda_i \hat{h}_i$.

To each triple $(\underline{\lambda}, \pi, \underline{\tau}) \in \widehat{\mathcal{M}}_d^*(\mathfrak{R})$ one associates a translation surface via the zippered rectangle construction, and the extended Zorich induction $\hat{\mathcal{Z}}$ acts on $\widehat{\mathcal{M}}_d^*(\mathfrak{R})$ by cutting and regluing the corresponding zippered rectangle. Consider the projection $p:\widehat{\mathcal{M}}_d^*(\mathfrak{R})\to {\mathcal{M}}_d(\mathfrak{R})$ given by $p((\underline{\lambda}, \pi, \underline{\tau}))=(\underline{\lambda}, \pi)$. The Zorich induction $\mathcal{Z}$ and its natural extension $\hat{\mathcal{Z}}$ commute via the projection $p$: $p\hat{\mathcal{Z}}={\mathcal{Z}}p$.

Since the extended Zorich induction preserves area, we can restrict it to the area-one locus
\[
\widehat{\mathcal{M}}_d(\mathfrak{R}) := \left\{ (\underline{\lambda}, \pi, \underline{\tau}) \in \widehat{\mathcal{M}}_d^*(\mathfrak{R}) \,:\, \mathrm{Area}(\underline{\lambda}, \pi, \underline{\tau}) = 1 \right\}.
\]
In $\widehat{\mathcal{M}}_d(\mathfrak{R})$, $\hat{\mathcal{Z}}$ preserves an ergodic absolutely continuous measure $\widehat{\nu}_{\mathfrak{R}}$ whose projection onto $\mathcal{M}_d(\mathfrak{R})$ coincides with the Zorich measure $\nu_{\mathfrak{R}}$. That is, it holds that $\nu_{\mathfrak{R}}(A)=\widehat{\nu}_{\mathfrak{R}}(p^{-1}A)$ for any measurable $A\subset{\mathcal{M}}_d(\mathfrak{R})$ (here $p$ is restricted to $\widehat{\mathcal{M}}_d(\mathfrak{R})$).
Moreover, $\widehat{\nu}_{\mathfrak{R}}$ is equivalent to the product of the Lebesgue measure on $\underline{\lambda}$ and $\underline{\tau}$ on each fiber $\{\pi\}$ (see again \cite{Viana2008}). 

We introduce one last definition. Given $\nu>0$, we say that a vector $\underline{v}=(v_1,\cdots,v_d)$ is $\kappa$--\textit{balanced} if $v_i>0$ for all $i$ and $\max_{i,j} v_i/v_j < \kappa$ or, equivalently, $\min_{i,j} v_i/v_j > \kappa^{-1}$. 

We are now ready to finish the proof of Proposition \ref{prop:full_meas_IET}.

\begin{proof}[Proof of Proposition \ref{prop:full_meas_IET} (Condition \eqref{eq:IET2})]
Fix $\pi \in \mathfrak{R}$ and pick $\underline{\tau_*}\in \Theta_\pi$ such that the associated geometric height vector $\hat{\underline{h_*}}= -\Omega_\pi \underline{\tau_*}$ is strictly positive and set $\kappa=2\max_{i,j} \hat{h_*}_i/\hat{h_*}_j$. 
Let $\{\delta_m\}\subset\R^+$ be a decreasing sequence, $\delta_m\to0$, and let $\underline{\lambda_m}=(1-\tfrac{\delta_m}2,\tfrac{\delta_m}{2(d-1)},\cdots,\tfrac{\delta_m}{2(d-1)})$. 

For each $m$, define the normalized $\underline{\tau_m}={\underline{\tau_*}}{\text{Area}(\underline{\lambda_m}, \pi, \underline{\tau_*})}^{-1}$ so that the triple $(\underline{\lambda_m}, \pi, \underline{\tau_m})$ lies in $\widehat{\mathcal{M}}_d(\mathfrak{R})$; and let $K_m \subset \widehat{\mathcal{M}}_d(\mathfrak{R})$ be a compact neighborhood of $(\underline{\lambda_m}, \pi, \underline{\tau_m})$ on which both $\kappa$-balance of $\hat{\underline{h}}$ and $\lambda_1 > (1-\delta_m)$ hold. Since $\widehat{\nu}_{\mathfrak{R}}$ is equivalent to Lebesgue on each fiber, $K_m$ has positive $\widehat{\nu}_{\mathfrak{R}}$-measure.

By ergodicity of $\widehat{\nu}_{\mathfrak{R}}$, for a.e. $(\underline\lambda, \pi, \underline\tau)$ there exists an increasing sequence $(r_m)$ such that $\hat{\mathcal{Z}}^{r_m}(\underline\lambda, \pi, \underline\tau)\in K_m$, for $m\geq 1$. The $\kappa$-balance of $\hat{\underline{h}}$ implies $\eta$-balance of the Rokhlin tower heights $(h_1^{(r_m)},\dots,h_d^{(r_m)})$, for some constant $\eta(\kappa)$ independent of $m$. 
Moreover, knowing \[\lambda^{(r_m)}-\sum_{j\neq 1}\lambda_j^{(r_m)}=\lambda_1^{(r_m)}>(1-\delta_m) \lambda^{(r_m)},\] and the partition identity $\sum_j \lambda_j^{(r_m)} h_j^{(r_m)} = 1$, we get \[\sum_{j\neq 1}\lambda_j^{(r_m)} <\delta_m\lambda^{(r_m)},\quad\text{ and }\quad h_1^{(r_m)} <((1-\delta_m)\lambda^{(r_m)})^{-1}.\] Therefore
\[
1 - \lambda\bigl(\mathcal{T}_1^{(r_m)}\bigr) = \sum_{j \neq 1} \lambda_j^{(r_m)} h_j^{(r_m)} <\delta_m\lambda^{(r_m)} \cdot \frac{\eta}{(1-\delta_m)\lambda^{(r_m)}} = \frac{\delta_m\eta}{1-\delta_m}.
\]

Projecting into $\mathcal{M}_d(\mathfrak{R})$, we get that the former holds for almost every $(\underline{\lambda},\pi)$. Given any $\eps>0$, we restrict the subsequence to $m$ large enough so that $0<\delta_m<1/4$ and  $\delta_m\eta/(1-\delta_m) < \varepsilon$, yielding ${\lambda_{j^*}^{(r_m)}} > \tfrac{3}{4}{\lambda^{(r_m)}}$ and $\lambda(\mathcal{T}_1^{(r_m)}) > 1-\varepsilon$. 

\end{proof}

\section*{Acknowledgments} The author would like to thank Adam Kanigowski for bringing attention to these problems, for many helpful discussions and insights, and for his support during the preparation of the paper.

\printbibliography

\end{document}